\documentclass{amsart}
\calclayout
\usepackage{amssymb,amsmath,amsfonts,epsfig,latexsym,tikz, amsrefs}
\usepackage{tikz-cd}
\usepackage{hyperref}
\usepackage{enumerate}
\usepackage{mathtools}
\usepackage{verbatim}
\usepackage{cleveref}
\usepackage[shortlabels]{enumitem}
\usepackage{subcaption}
\usepackage{ mathrsfs }

\usetikzlibrary{positioning}
\usetikzlibrary{matrix}
\usetikzlibrary{decorations}
\usetikzlibrary{decorations.pathreplacing, decorations.pathmorphing, angles,quotes}

\newtheorem{theorem}{Theorem}[section]
\newtheorem{definition}[theorem]{Definition}
\newtheorem{proposition}[theorem]{Proposition}
\newtheorem{lemma}[theorem]{Lemma}
\newtheorem{corollary}[theorem]{Corollary}

\def\C{\mathbb{C}}

\def\R{\mathbb{R}}

\def\Z{\mathbb{Z}}
\def\K{\mathcal{K}}

\def\II{\mathscr{I}}

\def\1{\mathbf{1}}

\def\cS{\mathcal{S}}

\def\<{\langle}
\def\>{\rangle}

\DeclareMathOperator{\Ehr}{Ehr}

\DeclareMathOperator{\tr}{tr}

\DeclareMathOperator{\id}{id}

\DeclareMathOperator{\pr}{pr}
\DeclareMathOperator{\Int}{Int}

\DeclareMathOperator{\rank}{rank}
\DeclareMathOperator{\Pyr}{Pyr}

\DeclareMathOperator{\rev}{rev}
\DeclareMathOperator{\Ind}{Ind}
\DeclareMathOperator{\Res}{Res}
\DeclareMathOperator{\ev}{ev}

\DeclareMathOperator{\GL}{GL}
\DeclareMathOperator{\Aff}{Aff}

\DeclareMathOperator{\Cyl}{Cyl}

\DeclareMathOperator{\face}{\mathcal{F}}

\DeclareMathOperator{\cf}{cf}

\theoremstyle{definition}
\newtheorem{remark}[theorem]{Remark}
\newtheorem{example}[theorem]{Example}

\begin{document}
	
	\title[]{Subdivisions of lower Eulerian posets and KLS theory}
	%{Subdivisions of Eulerian posets,  $cd$-indices, and Kazhdan-Lusztig-Stanley theory}
	%{Poset subdivisions, $cd$-indices, and Kazhdan-Lusztig-Stanley theory}
	%{Subdivisions of lower Eulerian posets and Kazhdan-Lusztig-Stanley theory}          
	\date{\today}
	\author{Alan Stapledon}

	\address{Sydney Mathematics Research Institute, L4.42, Quadrangle A14, University of Sydney, NSW 2006, Australia}
	\email{astapldn@gmail.com}
	\begin{abstract}
		In a companion paper, a
		canonical bijection was established between strong formal subdivisions of lower Eulerian posets and triples consisting of a lower Eulerian poset, a corresponding rank function, and a non-minimal element such that the join with any other element exists.
		The main goal of this paper is to relate the local $h$-polynomials of a strong formal subdivision to the Kazhdan-Lusztig-Stanley (KLS) invariants associated to its corresponding lower Eulerian poset under this bijection. As an application, we show that 
		Braden and MacPherson's  relative $g$-polynomials are alternative encodings of corresponding local $h$-polynomials. We also further develop equivariant KLS theory and give  equivariant generalizations of our main results, as well as  an application to 
		equivariant Ehrhart theory. 
	\end{abstract}
	
	\maketitle
%	\tableofcontents
	
	\vspace{-20 pt}

	\section{Introduction}
	
	Kazhdan-Lusztig-Stanley (KLS) theory was introduced by Stanley in \cite{Stanley92} as a way to abstract computations coming from intersection cohomology to the setting of posets.  
	The theory was later developed, for example, in \cite{Stanley92,DyerHecke,BrentiTwistedIncidence,BrentiPKernels,KatzStapledon16,ProudfootAGofKLSpolynomials,FMVChowFunctions,FREulerian}. 
	See \cite{BPIntersectionCohomology} for a recent survey. 
	Given a finite poset $\Gamma$, the \emph{incidence algebra} $I(\Gamma)$ 
	is the set of functions from intervals in $\Gamma$ to $\Z[t]$ (see Section~\ref{ss:KLSbackground} for the $\Z[t]$-algebra structure). Given $p \in I(\Gamma)$ and an interval $[z,z']$ in $\Gamma$, we often write $p([z,z']) = p(z,z') =  p(z,z';t)\in \Z[t]$. Given an element $\kappa_\Gamma \in I(\Gamma)$ called a \emph{$\Gamma$-kernel} (see Definition~\ref{def:kernel}), one may consider three invariants in $I(\Gamma)$: 
		the \emph{right and left Kazhdan-Lusztig-Stanley functions} $f_\Gamma$ and $g_\Gamma$ respectively,  
	%the \emph{right Kazhdan-Lusztig-Stanley function} $f_\Gamma$, 
	and the \emph{$Z$-function} $Z_\Gamma$ (see Theorem~\ref{thm:existenceofg},  Definition~\ref{def:Zfunction}).
	
	Assume that $\Gamma$ is \emph{lower Eulerian}. This means that 
	it contains a unique minimal element, denoted $\hat{0}_\Gamma$, and  every interval between distinct elements contains as many elements of even rank as odd rank. If $\Gamma$ also contains a unique maximal element, denoted $\hat{1}_\Gamma$, then $\Gamma$ is \emph{Eulerian}.  
	A prototypical example of a lower Eulerian poset %from topology
	 is a $CW$-poset. This is 
	the poset  $\face(\K)$ of closed cells, including the empty cell, of a regular $CW$-poset $\K$ under inclusion. For example, the face poset  of a fan  (viewed as a regular $CW$-complex by intersecting with a sphere centered at the origin), or the face poset of a polyhedral subdivision  of a polytope are $CW$-posets, and hence lower Eulerian posets. In this case, there is a natural choice of $\Gamma$-kernel called the \emph{Eulerian kernel} (see Example~\ref{ex:t-1case}).  
	See \cite{FREulerian}*{Theorem~1.2} for a beautiful interpretation of the corresponding $Z$-function. 	Later, when considering the equivariant case, we will see examples of lower Eulerian posets where other choices of  $\Gamma$-kernel naturally arise (see Example~\ref{ex:introequivariantfan}). 
	
	The notion of a \emph{strong formal subdivision} $\sigma: X \to Y$
	between lower Eulerian posets $X$ and $Y$ with rank functions $\rho_X$ and $\rho_Y$ respectively was introduced in \cite{KatzStapledon16}*{Definition~3.17} and further developed in \cite{DKTPosetSubdivisions,StapledonLWPosets}, replacing the notion of a formal subdivision in \cite[Definition~7.4]{Stanley92}, and also generalizing the notion of subdivisions of Gorenstein*-posets introduced in   \cite{EKDecompositionTheoremCDIndex}*{Definition~2.6}.
	We have the following important example. 
	
	\begin{example}\label{ex:introproper}(see  \cite{StapledonLWPosets}*{Example~3.8})
		Consider a linear map $\phi: V' \to V$ inducing a proper, surjective morphism between %full-dimensional 
		fans 
		$\Sigma'$ and $\Sigma$ in real vector spaces $V'$ and $V$ respectively.  Then 
		the corresponding strong formal subdivision is the function $\sigma: \face(\Sigma') \to \face(\Sigma)$ between face posets, where $\sigma(C')$ is the smallest cone of $\Sigma$ containing $\phi(C')$, for all $C' \in \Sigma'$.  
		%Let $s = \dim \ker(\phi)$. 
		The corresponding rank functions are defined as follows:
			 $\rho_{\face(\Sigma')}(C') = \dim C'$ for all $C' \in \Sigma'$, and  
		 $\rho_{\face(\Sigma)}(C) = \dim C + \dim \ker(\phi) = \dim \phi^{-1}(C)$ for all $C \in \Sigma$. 
	\end{example}
	
	We have the following special case. 	Given a polyhedron $Q$ in a real vector space,  let  $C(Q)$ be the cone spanned by $Q$. 
	
	\begin{example}\label{ex:intropolytope}(see \cite{KatzStapledon16}*{Lemma~3.25})
		Let $\cS$ be a polyhedral subdivision of a full-dimensional 
		polytope $P$ in a real vector space $V$. 
		Then the identity map on $V \oplus \R$ induces  a refinement of fans between $\Sigma_\cS = \{ C(G \times \{1 \}) : G \in \cS \}$ and the pointed cone $C(P \times \{1 \})$. 
		The corresponding strong formal subdivision is the function $\sigma: \face(\cS) \to \face(P)$ between face posets,  where $\sigma(F)$ is the smallest face of $P$ containing $F$, for all $F \in \cS$. 
		We have $\rho_{\face(\cS)}(G) = \dim G + 1$ for all $G \in \cS$, and 
		$\rho_{\face(P)}(G) = \dim G + 1$ for all faces $G$ of $P$. 
	\end{example}
	
%	For example, if $\cS$ is  polyhedral subdivision of a polytope $P$, then 
%	the corresponding strong formal subdivision is the function $\sigma: \face(\cS) \to \face(P)$ between face posets,  where $\sigma(F)$ is the smallest face of $P$ containing $F$, for all $F \in \cS$. 
%	By putting $P$ at height $1$ in a higher dimensional space and taking cones over the faces of $\cS$ and $P$, this may be viewed as a special case of the following example. 
%	Consider a linear map $\phi: V' \to V$ inducing a proper, surjective morphism between full-dimensional fans $\Sigma'$ and $\Sigma$ in real vector spaces $V'$ and $V$ respectively (see Example~\ref{TODO}).  Then 
%	the corresponding strong formal subdivision is the function $\sigma: \face(\Sigma) \to \face(\Sigma')$ between face posets, where $\sigma(C)$ is the smallest cone of $\Sigma'$ containing $\phi(C)$, for all $C \in \Sigma$ (see Example~\ref{TODO}).  
	
	Given  $x \in X$ and $y \in Y$ such that $\sigma(x) \le y$, one may consider
	the \emph{local $h$-polynomial} $\ell_\sigma(x,y) = \ell_\sigma(x,y;t)$ in $\Z[t]$ (see Definition~\ref{def:hellpolynomial}). 
	If $Y$ is Eulerian, then we call 
	$\ell_\sigma(\hat{0}_X,\hat{1}_Y)$ the local $h$-polynomial of $\sigma$. 
	The geometric significance of the local $h$-polynomial comes in two distinct flavors. 
	Firstly, when $\sigma$ is induced by a proper, surjective morphism of fans as in Example~\ref{ex:introproper}, the connection with proper maps of toric varieties, sheaves on fans, and the decomposition theorem are explored, for example, in \cite{Stanley92,KatzStapledon16,deCataldoMiglioriniMustata18,KaruRelativeHardLefschetz}.
%	In particular, if we further assume that $\sigma$ is projective, the coefficients of $\Delta \ell$ are nonnegative. \alan{find reference in Karu; use this to guide us for pointed cones??} \cite{KaruRelativeHardLefschetz}
	Secondly, when $\sigma$ is induced by a regular lattice polyhedral decomposition of a lattice polytope as in Example~\ref{ex:intropolytope}, the connection with the action of monodromy on the cohomology of the Milnor fiber of a nondegenerate hypersurface singularity is explored, for example, in \cite{Stapledon17,LPSLocalMotivic,SaitoMixed,STMonodromies}. 
	For further work on the local $h$-polynomial and connections with combinatorics and commutative algebra see, for example, 	\cite{dMGPSS20, ChanLocal, JKSS, LPS1, Athanasiadis12b, Athanasiadis12, LSLefschetz}, and survey papers
	\cite{AthanasiadisSurveySubdivisions, ChanSurvey}. 

The main goal of this paper is to relate the local $h$-polynomials of a strong formal subdivision to the KLS invariants associated to a corresponding lower Eulerian poset. We will use a canonical bijection from the companion paper \cite{StapledonLWPosets} that we now describe. See Section~\ref{ss:backgroundposet} for more details. 
%In fact, this paper is a companion paper to \cite{StapledonLWPosets}. 
We say that an element $q$ of a poset $B$ is \emph{join-admissible} if for any element $z$ in $B$, 
the join (or  least upper bound) $z \vee q$ exists. 

\begin{theorem}\label{thm:introbijection}\cite{StapledonLWPosets}*{Theorem~1.1}
	There is a canonical bijection between  strong formal subdivisions $\sigma: X \to Y$ 	between lower Eulerian posets $X$ and $Y$ with rank functions $\rho_X$ and $\rho_Y$ respectively, and triples $(\Gamma,\rho_\Gamma,q)$, where $\Gamma$ is a lower Eulerian poset
	with 
	rank function $\rho_\Gamma$, and $q$ is a join-admissible element of $\Gamma$ with $q \neq \hat{0}_\Gamma$.
\end{theorem}

%	In \cite{StapledonLWPosets}*{Theorem~1.1}, a canonical bijection is given between strong formal subdivisions $\sigma: X \to Y$, and triples $(\Gamma,\rho_\Gamma,q)$, where $\Gamma$ is a lower Eulerian poset
%	 with 
%	 %unique minimal element $\hat{0}_\Gamma$ and 
%	 rank function $\rho_\Gamma$, and $q$ is a join-admissible element of $\Gamma$ with $q \neq \hat{0}_\Gamma$.
	 
	In Theorem~\ref{thm:introbijection},  $\Gamma$ is obtained from $\sigma$ via the 
	\emph{non-Hausdorff mapping cylinder} construction of  Barmak and Minian \cite{BMSimpleHomotopy}*{Definition~3.6}. That is, $\Gamma$ is the disjoint union of $X$ and $Y$ as sets. As a poset, $\Gamma$ inherits the orderings on $X$ and $Y$. Also, if $x \in X$ and $y \in Y$, then 
	$x \le y$ in $\Gamma$ if $\sigma(x) \le y$ in $Y$. 
	Conversely, we recover $\sigma$ by setting 
	$Y = \{ z \in \Gamma : q \le z \}$, $X = \Gamma \smallsetminus Y$, and  $\sigma(x) = x \vee q$ for all $x \in X$. 
	
%	, and $z \le z'$ in $\Gamma$ if and only if either $z,z' \in X$ and $z \le z' \in X$, or $z,z' \in Y$ and $z \le z' \in Y$, or $z \in X$, $z' \in Y$, and $\sigma(z) \le z' \in Y$. Also, $q = \hat{0}_Y \in \Gamma$. 
%	Conversely,  
	
	For example, if $\sigma$ is induced by a proper, surjective morphism of fans as in Example~\ref{ex:introproper}, then $\Gamma$ is a $CW$-poset (see \cite{StapledonLWPosets}*{Example~7.26}). We have the following important special case. 
	
	\begin{example}\label{ex:introprojectivetocone}\cite{StapledonLWPosets}*{Section~4.1}
	   Suppose that $\sigma: \face(\Sigma') \to \face(\Sigma)$ is induced by a proper, surjective morphism of fans  as in Example~\ref{ex:introproper}. Assume further than the morphism of fans is projective, and $\Sigma$ is a pointed cone.
	   Consider the corresponding triple $(\Gamma,\rho_\Gamma,q)$ under Theorem~\ref{thm:introbijection}. 
	    Then there exists an explicitly constructed polytope $Q$ such that $\Gamma = \face(Q)$, $\rho_\Gamma(G) = \dim G + 1$ for all $G \in \face(Q)$,  and $q$ corresponds to a nonempty face $F$ of $Q$. Conversely, every pair of a polytope $Q$ and a nonempty face $F$ of $Q$ appears in this way. Moreover, $F$ is a vertex if and only $\sigma$ is induced by a regular polyhedral subdivision $\cS$ of a  polytope $P$ as in Example~\ref{ex:intropolytope}. 
	   	See Example~\ref{ex:polytope} for details.
	\end{example}

%	If $\sigma$ is induced by a projective, surjective morphism of fans and $\Sigma'$ is a pointed cone, then $\Gamma = \face(Q)$ is the face lattice of a polytope $Q$, $\rho_\Gamma(G) = \dim G + 1$ for all $G \in \face(Q)$,  and $q$ corresponds to a nonempty face $F$ of $Q$. Conversely, every pair of a polytope $Q$ and a nonempty face $F$ of $Q$ appears in this way, and $F$ is a vertex if and only $\sigma$ is induced by a regular polyhedral subdivision $\cS$ of a  polytope $P$ (see \cite{StapledonLWPosets}*{Section~4.1}).  %Moreover, pairs consisting of a polytope $Q$ and a vertex $v$ correspond to strong formal subdivisions induced by a 
%	See Example~\ref{ex:polytope} for more details.
%	
%%	 $\Gamma$ is the disjoint union of $X$ and $Y$ as a set, $Y = \{ z \in \Gamma : q \le z \}$,  and $\sigma(x) = x \vee q$ for all $x \in X$. 
	
	To state our main theorem, we introduce some notation.	We say that a polynomial $a \in \Z[t]$ is symmetric if there exists a nonnegative integer $m$ (necessarily unique if $a$ is nonzero) such that
	$a = \sum_{i = 0}^m a_i t^i$ for some coefficients $a_i \in \Z$ satisfying $a_i = a_{m - i} \in \Z$ for $0 \le i \le m$. In this case, define $\Delta a = a_0 + \sum_{i = 1}^{m/2} (a_i - a_{i - 1}) t^i \in \Z[t]$. Observe that $\Delta a$ is an alternative encoding of $a$, and $\Delta a$ has nonnegative coefficients if and only if $a$ has nonnegative unimodal coefficients. For example, the local $h$-polynomial 
	$\ell_\sigma(x,y)$ above is symmetric (see Proposition~\ref{prop:symmetry}), and hence may alternatively be encoded by $\Delta \ell_\sigma(x,y)$. 
	
	Consider $\Gamma$ corresponding to $\sigma$ under the bijection in Theorem~\ref{thm:introbijection}.  Fix the Eulerian kernel in $I(\Gamma)$ and consider the corresponding KLS invariants 
	 $f_\Gamma$, $g_\Gamma$, and $Z_\Gamma$. 
	Given an element $p_\Gamma \in I(\Gamma)$, write $p_X$ and $p_Y$ for the restriction of $p$ to $I(X)$ and $I(Y)$ respectively. 
	Then  $f_X$, $g_X$, and $Z_X$ are the corresponding
	KLS invariants associated to $X$ with the Eulerian kernel, and $f_Y$, $g_Y$, and $Z_Y$ are the corresponding
	KLS invariants associated to $Y$ with the Eulerian kernel. 
	% (see \ref{TODO}). 
	
	A consequence of the theorem below is that 
	$g_\Gamma$ determines and is determined by $g_X$, $g_Y$, and the local $h$-polynomials $\{ \ell_\sigma(x,y) : x \in X, y \in Y, \sigma(x) \le y \}$ (see the discussion before Theorem~\ref{thm:maingell}).  In the body of the paper, the statement of the theorem below
 is written concisely using the language of incidence algebras. Below, we expand out the definitions for the benefit of the reader.

\begin{theorem}\label{thm:intromainKLS}(see Theorem~\ref{thm:maingell},Remark~\ref{rem:altdeltaell}, Corollary~\ref{cor:rightKLSfunction}, Corollary~\ref{cor:Zmappingformula})
	Let  $\sigma: X \to Y$ be a strong formal subdivision
	between lower Eulerian posets $X$ and $Y$ with rank functions $\rho_X$ and $\rho_Y$ respectively, corresponding to a triple $(\Gamma,\rho_\Gamma,q)$  under Theorem~\ref{thm:introbijection}. 
	 Fix the Eulerian kernel in $I(\Gamma)$. 
	For any  $x \in X$ and $y \in Y$ such that $\sigma(x) \le y$,
		\[
		\Delta \ell_{\sigma}(x,y) = \sum_{ \sigma(x) \le y' \le y} (-1)^{\rho_Y(y') - \rho_Y(y)} g_\Gamma(x,y') f_Y(y',y),
		\] 
	\[
	g_\Gamma(x,y) = \sum_{ \sigma(x) \le y' \le y} \Delta \ell_{\sigma}(x,y') g_Y(y',y), 
	\] 
	\[
	f_\Gamma(x,y) 
	=  \sum_{ \substack{x \le x' \in X \\ \sigma(x') \le y} } (-1)^{\rho_Y(y) - \rho_X(x')} f_X(x,x') \Delta \ell_{\sigma}(x',y),
	\] 
		\[
	Z_\Gamma(x,y) = \sum_{ \substack{x \le x' \in X \\ \sigma(x') \le y} } (-1)^{\rho_Y(y) - \rho_X(x')} Z_X(x,x') \Delta \ell_{\sigma}(x',y) + \sum_{ \sigma(x) \le y' \le y}  t^{\rho_Y(y') - \rho_X(x) + 1} (\Delta \ell_{\sigma})(x,y';t^{-1}) Z_Y(y',y).
	\] 
\end{theorem}
	
%	For example, if $\sigma(x) = y$, then the first two  equations  say that $g_\Gamma(x,y) = \Delta \ell_\sigma(x,y)$. 
	
	Consider again Example~\ref{ex:introprojectivetocone}. In particular, $Q$ is a polytope with nonempty face $F$, and $\Gamma = \face(Q)$. 
%	Consider again the example when
%	$\Gamma = \face(Q)$ is the face lattice of a polytope $Q$, $\rho_\Gamma(G) = \dim G + 1$ for all $G \in \face(Q)$,  and $q$ corresponds to a nonempty face $F$ of $Q$.
%Recall that the corresponding strong formal subdivision $\sigma: X \to Y$ is
%induced by a  projective, surjective morphism between a fan and a pointed cone. 
	In \cite{BMIntersectionHomologyKalai}*{Proposition~2}, Braden and MacPherson introduced the \emph{relative $g$-polynomial} $g(Q,F)$ associated to the pair $(Q,F)$ (see Example~\ref{ex:polytope}). For example, $g(Q,Q) = g_\Gamma(\Gamma) = g_\Gamma(\hat{0}_X, \hat{1}_Y)$. 
	The corollary below states that the relative $g$-polynomial is obtained by applying the operator $\Delta$ to the local $h$-polynomial of the corresponding strong formal subdivision $\sigma$. 
	%For example, when $F = Q$, it says that the local $h$-polynomial of $\sigma$ is the toric $h$-polynomial of $Q$. \alan{add reference; to example?}
	
	\begin{corollary}(see Example~\ref{ex:polytope})
		With the notation of Example~\ref{ex:introprojectivetocone}, 	$g(Q,F) = \Delta \ell_{\sigma}(\hat{0}_X, \hat{1}_Y).$
	\end{corollary}
	
		Braden and MacPherson proved that the coefficients of $g(Q,F)$ are nonnegative integers
	(see \cite{BMIntersectionHomologyKalai}*{Theorem~4} for the case when $Q$ is rational; the proof was later extended to the general case in \cite{BBFKCombinatorialIntersectionCohomology,BLIntersectionCohomologyNonrational}).  
	By the above corollary, the nonnegativity of the coefficients of $g(Q,F)$ is equivalent to the local $h$-polynomial of the corresponding strong formal subdivision $\sigma$  having nonnegative unimodal coefficients. The latter is known and due to Karu \cite{KaruRelativeHardLefschetz}.
	
	 An equivariant generalization of  Kazhdan-Lusztig-Stanley theory was developed in \cite{ProudfootEquivariantKLS}, generalizing work in 
	\cite{GPYEquivariantKLS} (see Section~\ref{ss:backgroundequivariantKLS} for an overview).
	We further develop this theory in Section~\ref{ss:classfunctions} and 
	Section~\ref{ss:equivariantKLSlowerEulerian}, and prove an equivariant version of Theorem~\ref{thm:intromainKLS} in Theorem~\ref{thm:equivariantmaingell}.  We will not give the full statement in the introduction, but instead present part of the statement in an important example in  Theorem~\ref{thm:introequivariantKLS} below.

	Let $W$ be a finite group acting on a finite poset $\Gamma$. 
	For any $z \in \Gamma$, let $W_z = \{ w \in W : w \cdot z = z \}$ be the stabilizer of $z$. For $z \le z' \in \Gamma$, let $W_{z,z'} = W_z \cap W_{z'}$. 
	There is a notion of an   \emph{equivariant incidence algebra} $I^W(\Gamma)$ (see Definition~\ref{def:equivariantincidencealgebra}). If $p \in I^W(\Gamma)$ and $[z,z']$ is an interval in $\Gamma$, then $p(z,z') \in R(W_{z,z'})[t]$, where $R(W_{z,z'})$ is the \emph{complex representation ring} of $W_{z,z'}$ (see Section~\ref{ss:representationtheory}). There is also the notion of an 
	\emph{equivariant $\Gamma$-kernel}, and 
	corresponding invariants in $I^W(\Gamma)$: 
	the \emph{right and left equivariant Kazhdan-Lusztig-Stanley functions} $f_\Gamma$ and $g_\Gamma$ respectively,  
	and the \emph{equivariant $Z$-function} $Z_\Gamma$.
	Assume further than $\Gamma$ is a lower Eulerian poset. In general, there is not a natural choice of an equivariant $\Gamma$-kernel, as in the non-equivariant case. However, we show that there is a natural choice in some important special cases. 
	%The example below includes the case where $W$ acts affinely on a polytope, preserving a polyhedral subdivision, for example, the trivial subdivision (see Example~\ref{ex:polytopeequiv}). 
	    If $w \in W$ acts on a set $S$, we let $S^w$ be the fixed locus.

	\begin{example}(see Example~\ref{ex:fanequiv}, Example~\ref{ex:fanequivv2})\label{ex:introequivariantfan}
		Suppose that $\Sigma$ is a %full-dimensional 
		fan in a real vector space $V$, and $\psi: W \to \GL(V)$ is a real representation of a finite group $W$ such that $\Sigma$ is $W$-invariant. 
		%Assume that the support $|\Sigma|$ of $\Sigma$ is convex. 
		Then we have a corresponding action of $W$ on $\face(\Sigma)$, and there is a natural choice of equivariant $\face(\Sigma)$-kernel $\kappa_{\face(\Sigma)}$. 
		Explicitly, for any $z \in \face(\Sigma)$, let $V_z$ be the linear span of the corresponding cone in $\Sigma$. For $z \le z' \in \face(\Sigma)$, we have an induced representation $\psi_{z,z'}: W_{z,z'} \to \GL(V_{z'}/V_z)$. Then $\kappa_{\face(\Sigma)}(z,z') = 
			\det(tI - \psi_{z,z'})$, where the coefficient of $t^{\dim (V_{z'}/V_z) - i}$ in $\det(tI - \psi_{z,z'})$ is $(-1)^i$ times the class of the $i$th exterior power of the representation $\psi_{z,z'}$ in $R(W_{z,z'})$ (see \eqref{eq:detI-rhot}).
	
		Fix an element $w \in W$. 
		%If $w$ acts on a set $S$, let $S^w$ be the fixed locus.  
		Then the $w$-fixed subposet $\face(\Sigma)^w$ is the face poset of the fan $\{ C^w : C \in \Sigma, w \cdot C = C \}$ (see Example~\ref{ex:fanactionisEulerian}), and hence is a $CW$-poset, and a lower Eulerian poset. 
		Recall that there is a correspondence between complex representations of a finite group and class functions of the group (see Section~\ref{ss:representationtheory}). 
		 In Lemma~\ref{lem:reducenonequivariant}, we show that by evaluating corresponding class functions at $w$ in $W$, we have a natural choice of $\face(\Sigma)^w$-kernel $\ev_w(\kappa_{\face(\Sigma)})$, 
		 which does not equal the Eulerian kernel when  $\face(\Sigma)^w \neq \face(\Sigma)$ (c.f. Remark~\ref{rem:rankwarning}). 
%		 which agrees with the Eulerian kernel when $w$ is the identity element, but, in general, is different. 
Explicitly, for $z \le z'$ in $\face(\Sigma)^w$, $\ev_w(\kappa_{\face(\Sigma)})(z,z') = \det(tI - \psi_{z,z'}(w))$ is the characteristic polynomial of $\psi_{z,z'}(w) \in \GL(V_{z'}/V_z)$, while the value of the Eulerian kernel at $[z,z']$ is $(t - 1)^{\dim (V_{z'}/V_z)^w}$ (see Example~\ref{ex:t-1case} and \eqref{eq:rhoGammaw}). 
	\end{example}
	
		\begin{example}\label{ex:introequivariantpolytope}(c.f. Example~\ref{ex:polytopeequiv})
		Let $\cS$ be a polyhedral subdivision of a full-dimensional polytope $P$ in a real vector space $V$. For example, the trivial subdivision of $P$ consists of the faces of $P$. Recall from Example~\ref{ex:intropolytope} that we may consider the fan 
		$\Sigma_\cS = \{ C(G \times \{1 \}) : G \in \cS \}$. 
		Suppose that $\psi: W \to \Aff(V)$ is an affine representation of a finite group $W$. That is, there is a  representation
		$\widetilde{\psi}: W \to \GL(V \oplus \R)$ such that $W$ acts trivially on the last coordinate of $V \oplus \R$, and $\psi$ is the induced action on $V \times \{ 1 \}$. Assume that $\cS$ is $W$-invariant, or, equivalently, that 	$\Sigma_\cS$ is $W$-invariant. Then there is a natural choice of equivariant $\face(\cS)$-kernel by  Example~\ref{ex:introequivariantfan} above. 
		
%		
%		% and the pointed cone $C(P \times \{1 \})$. 
%		The corresponding strong formal subdivision is the function $\sigma: \face(\cS) \to \face(P)$ between face posets,  where $\sigma(F)$ is the smallest face of $P$ containing $F$, for all $F \in \cS$. 
%		We have $\rho_{\face(\cS)}(G) = \dim G + 1$ for all $G \in \cS$, and 
%		$\rho_{\face(P)}(G) = \dim G + 1$ for all faces $G$ of $P$. 
	\end{example}
	
	In fact, the above two examples are special cases of the following equivariant version of Example~\ref{ex:introproper}. See Example~\ref{ex:propermapequiv} for more details.
	
	\begin{example}\label{ex:introproperequivariant}
		As in Example~\ref{ex:introproper}, consider a linear map $\phi: V' \to V$ inducing a proper, surjective morphism between %full-dimensional 
		fans $\Sigma'$ and $\Sigma$ in real vector spaces $V'$ and $V$ respectively, with induced strong formal subdivision  
		$\sigma: X = \face(\Sigma') \to Y = \face(\Sigma)$. Let  $(\Gamma,\rho_\Gamma,q)$  be the corresponding triple under Theorem~\ref{thm:introbijection}. Let  $\psi': W \to \GL(V')$ and $\psi: W \to \GL(V)$ be real representations of $W$ such that we have an induced action of $W$ on 
		$\Sigma'$ and $\Sigma$, and 
		$\phi: V' \to V$ is $W$-equivariant. Then we have a corresponding action of $W$ on $\Gamma$ that fixes $q$, and   there is a natural choice of equivariant $\Gamma$-kernel $\kappa_\Gamma$ (see \eqref{eq:kappagammaequiv}). 
		See Example~\ref{ex:polytopeequiv} for the special case when $\Gamma = \face(Q)$ for a corresponding polytope $Q$ as in 	Example~\ref{ex:introprojectivetocone}.
		% and a comparison with 
		% Example~\ref{ex:introequivariantpolytope} for the trivial subdivision of $Q$.
%		
%		In the special case of Example~\ref{ex:introprojectivetocone} when $\Gamma = \face(Q)$ for a corresponding polytope $Q$, $\kappa_\Gamma$ agrees with the equivariant $\Gamma$-kernel described in Example~\ref{ex:introequivariantpolytope} for the trivial subdivision of $Q$. 
		In general, $\kappa_\Gamma$ restricts to the equivariant kernels for $X = \face(\Sigma')$ and $Y = \face(\Sigma)$ described in Example~\ref{ex:introequivariantfan}.
		If $\cS$ is a polyhedral subdivision of a full-dimensional polytope $P$ in a real vector space $V$, and $W$ acts affinely on $V$ preserving $\cS$, then we obtain an equivariant version of Example~\ref{ex:intropolytope}. The corresponding equivariant $\Gamma$-kernel restricts to the equivariant kernels for $X = \face(\cS)$ and $Y = \face(P)$ described in Example~\ref{ex:introequivariantpolytope}. 
	\end{example}

%More generally, as above, consider a linear map $\phi: V' \to V$ inducing a proper, surjective morphism between full-dimensional fans $\Sigma'$ and $\Sigma$ in real vector spaces $V'$ and $V$ respectively, with induced strong formal subdivision  
%$\sigma: \face(\Sigma) \to \face(\Sigma')$. 
%Let  $(\Gamma,\rho_\Gamma,q)$  be the corresponding triple under the canonical bijection of \cite{StapledonLWPosets}*{Theorem~1.1}. 
% Let $\psi: W \to \GL(V)$ and $\psi': W \to \GL(V')$ be real representations of $W$ such that we have an induced action of $W$ on 
%$\Sigma'$ and $\Sigma$, and 
%$\phi: V' \to V$ is $W$-equivariant.
%Then \eqref{eq:kappagammaequiv} describes a natural choice of equivariant $\Gamma$-kernel. See Example~\ref{ex:polytopeequiv} for the special case when $\Gamma$ is the face lattice of a polytope $Q$ with an affine action of $W$, and $q$ corresponds to a nonempty $W$-invariant face $F$. 

Assume the setup of Example~\ref{ex:introproperequivariant}. 
%With this setup, 	
For any   $x \in X$ and $y \in Y$ such that $\sigma(x) \le y$, we introduce the \emph{equivariant local $h$-polynomial} $\ell_{\sigma}(x,y)$ in $R(W_{x,y})[t]$ (see Definition~\ref{def:equivarianthellpolynomial}). As in the non-equivariant case, the equivariant local $h$-polynomial is a symmetric polynomial (see Corollary~\ref{cor:equivariantsymmetry}).
% We describe its behavior under compositions of strong formal subdivisions in Corollary~\ref{cor:equivariantcomposition}, generalizing a similar result in the non-equivariant case \cite{KatzStapledon16}*{Corollary~4.7}. 

In the body of the paper, the statement of the theorem below
is written concisely using the language of equivariant incidence algebras.
 Below, we expand out the definitions for the benefit of the reader. 
  For $z \le z'' \le z \in \Gamma$, let $W_{z,z'',z'} = W_z \cap W_{z''} \cap W_{z'}$. Below, $\Res$ and $\Ind$ denote restriction and  induction respectively, both applied coefficientwise (see Section~\ref{ss:representationtheory}). 
 
 \begin{theorem}(see Theorem~\ref{thm:equivariantmaingell}, Remark~\ref{rem:altdeltaellequiv})\label{thm:introequivariantKLS}
 	Assume the setup of Example~\ref{ex:introproperequivariant}. For any $x \in X$ and $y \in Y$ such that $\sigma(x) \le y$, 
\[
g_\Gamma(x,y) = \sum_{ \sigma(x) \le y' \le y} \frac{|W_{x,y',y}|}{|W_{x,y}|}
\Ind^{W_{x,y}}_{W_{x,y',y}} \left( \Res^{W_{x,y'}}_{W_{x,y',y}} \Delta \ell_\sigma(x,y') \cdot  \Res^{W_{y',y}}_{W_{x,y',y}} g_Y(y',y)  \right). 
\]
There are also formulas for $\Delta \ell_\sigma(x,y)$, $f_\Gamma(x,y)$, and $Z_\Gamma(x,y)$ that are equivariant generalizations of the corresponding formulas in Theorem~\ref{thm:intromainKLS}. 
 \end{theorem}

%As mentioned above, we have an equivariant generalization Theorem~\ref{thm:equivariantmaingell} of Theorem~\ref{thm:intromainKLS}. The statement is written concisely using the language of equivariant incidence algebras. To give the reader a flavor for the result, if we expand this out, then the equivariant analogue of the second equation in Theorem~\ref{thm:intromainKLS} is as follows. For $z \le z'' \le z \in \Gamma$, let $W_{z,z'',z'} = W_z \cap W_{z''} \cap W_{z'}$. 
%For any $x \in X$ and $y \in Y$ such that $\sigma(x) \le y$, 
%\[
%g_\Gamma(x,y) = \sum_{ \sigma(x) \le y' \le y} \frac{|W_{x,y',y}|}{|W_{x,y}|}
%\Ind^{W_{x,y}}_{W_{x,y',y}} \left( \Res^{W_{x,y'}}_{W_{x,y',y}} \Delta \ell_\sigma(x,y') \cdot  \Res^{W_{y',y}}_{W_{x,y',y}} g_Y(y',y)  \right). 
%\]
%

Finally, we consider an application to equivariant Ehrhart theory.
Ehrhart theory concerns the enumeration of lattice points in lattice polytopes. 
Equivariant Ehrhart theory is a generalization of Ehrhart theory that was introduced in \cite{StapledonEquivariant} and \cite{StapledonRepresentations11}, and has subsequently been developed in \cite{StapledonCalabi12,ASV20,ASV21,EKS22,CHK23,StapledonEquivariantCommutativeTriangulations,DDEquivariantHilbertTranslative}. 
Let $N$ be a lattice and let $N_\R = N \otimes_\Z \R$. 
Consider an affine representation $\psi: W \to \Aff(N)$ of a finite group $W$ that preserves a full-dimensional  lattice polytope $P \subset N_\R$. 
The corresponding \emph{equivariant $h^*$-series} $h^*(P,\psi;t)$ is a power series in $R(W)[[t]]$ that precisely encodes the Ehrhart theory of the rational polytopes $\{ P^w : w \in W \}$ (see \eqref{eq:equivhstar}). 
The corresponding \emph{equivariant local $h^*$-series} in $R(W)[[t]]$ first appeared in \cite{StapledonCalabi12}*{Definition~3.3} (see \eqref{eq:equivlocalhstar}). A great deal of interest concerns criterion that guarantee that  $h^*(P,\psi;t)$ and  $\ell^*(P,\psi;t)$ are polynomials, i.e., elements of $R(W)[t]$ (see, for example,  \cite{StapledonEquivariantCommutativeTriangulations}*{Conjecture~1.2,Theorem~1.4}). For example, when $W$ acts trivially, $h^*(P,\psi;t)$ and  $\ell^*(P,\psi;t)$ are the usual $h^*$-polynomial and local $h^*$-polynomial respectively associated to $P$ \cite[Example~7.2, Corollary~7.7]{Stanley92}.

Let $\cS$ be a $W$-invariant lattice polyhedral subdivision of $P$. Recall from Example~\ref{ex:introproperequivariant} that we may consider the corresponding strong formal subdivision $\sigma: X = \face(\cS) \to Y = \face(P)$ and triple $(\Gamma,\rho_\Gamma,q)$ under Theorem~\ref{thm:introbijection}, and a natural choice of equivariant $\Gamma$-kernel. 
Recall that
for any   $x \in X$ and $y \in Y$ such that $\sigma(x) \le y$, we may consider  the equivariant local $h$-polynomial $\ell_{\sigma}(x,y)$ in $R(W_{x,y})[t]$. We may also consider the \emph{equivariant $h$-polynomial} $h_{\sigma}(x,y)$ in $R(W_{x,y})[t]$ (see Definition~\ref{def:equivarianthellpolynomial}). We call $h_{\sigma}(\hat{0}_X,\hat{1}_Y)$ and $\ell_{\sigma}(\hat{0}_X,\hat{1}_Y)$ the equivariant $h$-polynomial and equivariant local $h$-polynomial respectively associated to $\sigma$. 
For any $z \in \face(\cS)$, let $F_z$ denote the corresponding face in $\cS$, and let $\psi_z$ denote the affine representation of $W_z$ acting on the affine span of $F_z$. 
%
%For any $z \in \Gamma$, let $F_z$ denote the corresponding face in $\cS$, if $z \in X$, or the corresponding face of $P$, if $z \in Y$. Let $\psi_z$ denote the affine representation of $W_z$ acting on the affine span of $F_z$. 
The following is a special case of Proposition~\ref{prop:equivariantEhrhartapp} (see \eqref{eq:hstarintermsofh} and \eqref{eq:localhstarintermsofh}). It reduces to known formulas when $W$ acts trivially (see  $(2)$ in \cite{KatzStapledon16}*{Lemma~7.12}). 

\begin{proposition}\label{prop:introequivEhrhart}
	Let $N$ be a lattice and
	consider an affine representation $\psi: W \to \Aff(N)$ of a finite group $W$ that preserves a full-dimensional  lattice polytope $P \subset N_\R$. Let $\cS$ be a $W$-invariant lattice polyhedral subdivision of $P$.
	Assume that for all $z \in \face(\cS)$, $h^*(F_z,\psi_z;t) \in R(W_z)[t]$. 
	Then $h^*(P,\psi;t), 	\ell^*(P,\psi;t) \in  R(W)[t]$. Moreover, 
	\begin{equation*}%\label{eq:hstarintermsofh}
		h^*(P,\psi;t)  = %\sum_{\substack{z \in \face(\cS) \\ \sigma(z) = \hat{1}_{\face(P)}
				\sum_{z \in \face(\cS)} \frac{|W_z|}{|W|} \Ind^{W}_{W_{z}} \left( \ell^*(F_{z},\psi_{z};t) h_\sigma(z,\hat{1}_\Gamma) \right),
			\end{equation*}
			\begin{equation*}%\label{eq:localhstarintermsofh}
				\ell^*(P,\psi;t)  = %\sum_{\substack{z \in \face(\cS) \\ \sigma(z) = \hat{1}_{\face(P)}
						\sum_{z \in \face(\cS)} \frac{|W_z|}{|W|} \Ind^{W}_{W_{z}} \left( \ell^*(F_{z},\psi_{z};t) \ell_\sigma(z,\hat{1}_\Gamma) \right).
					\end{equation*}
\end{proposition}

In the statement of Proposition~\ref{prop:introequivEhrhart}, it follows from the assumptions that $\ell^*(F_{z},\psi_{z};t)$ is a polynomial. 
If $\cS$ is a $W$-invariant lattice triangulation of $P$, then the assumptions of Proposition~\ref{prop:introequivEhrhart} hold, and we recover the formula of $h^*(P,\psi;t) \in R(W)[t]$ from \cite{StapledonEquivariantCommutativeTriangulations}*{Proposition~4.40} (see Example~\ref{ex:triangulations}). 
If we further assume that $\cS$ is a unimodular triangulation then
 $h^*(P,\psi;t)$ is the equivariant $h$-polynomial of $\sigma$, and 
$\ell^*(P,\psi;t)$ is the equivariant local $h$-polynomial of $\sigma$  (see Example~\ref{ex:triangulations}). The former statement recovers \cite{StapledonEquivariantCommutativeTriangulations}*{Remark~4.41} (c.f. \cite{DDEquivariantEhrhart}*{Theorem~5.2}). Finally, we note that in the non-equivariant case, the coefficients of $h_\sigma(z,\hat{1}_\Gamma)$ and $\ell_\sigma(z,\hat{1}_\Gamma)$ above have interpretations as dimensions of canonically defined vector spaces (see \cite{KaruRelativeHardLefschetz}). In the general case, we expect that the coefficients of $h_\sigma(z,\hat{1}_\Gamma)$ and $\ell_\sigma(z,\hat{1}_\Gamma)$, which are virtual representations of $W$ by construction, should be the classes of representations of $W$ acting on these canoncially defined vector spaces.

 	We briefly outline the contents of the paper. In Section~\ref{sec:background}, we recall background material to be used throughout the paper. Then  Section~\ref{sec:lowerEulerian} and Section~\ref{sec:equivariant} contain the non-equivariant and equivariant results described above respectively. See the first paragraph of each section for a more detailed description of its content.

	\emph{Notation and conventions}: All posets are finite posets. 
	An interval of a poset is a closed interval unless specified otherwise. Given elements $z \le z'$ in a poset $B$, we write 
	$[z,z'] = \{ z'' \in B : z \le z'' \le z' \}$. 
	All vector spaces are finite-dimensional. 	All cones in this paper will be polyhedral cones, i.e., the intersection of finitely many half-spaces in a real vector space. All groups are finite groups, and all group representations are complex representations. 
We write $|S|$ for the cardinality of a finite set $S$. 	Given a polynomial $a \in \Z[t]$, let $a_i$ denote the coefficient of $t^i$ in $a$. 
Given a field $k$ and a $\Z$-module $A$, write $A_k = A \otimes_\Z k$.

\subsection*{Acknowledgements}
We thank Luis Ferroni for some very helpful conversations.

%	
%References: 
%
%new paper \cite{FREulerian}
%
%\begin{enumerate}
%	\item see page 6 for the definition of toric $h$-polynomial
%	
%	\item They call the standard kernel for lower Eulerian posets the \emph{Eulerian kernel}
%	
%	\item main result is Theorem 1.2; let $B$ be an Eulerian poset. Then the $Z$-polynomial of $B$ using the Eulerian kernel equals the toric $h$-polynomial of the poset of intervals of $P$ ordered by reverse inclusion. (think this is our definition of $h$-polynomial if excluded the empty interval, which is the maximal element; theorem of Lindstrom says that this poset is Eulerian)
%	
%	\item Section 5.4 give an example of a Gorenstein* poset whose $Z$-polynomial has negative coefficients; unknown if further assume lattices
%\end{enumerate}

\section{Background}\label{sec:background}

We recall background that will be used throughout the paper. In Section~\ref{ss:backgroundfans},  we recall background on fans and polytopes. 
In Section~\ref{ss:backgroundposet}, we recall background on posets, including details on the bijection in Theorem~\ref{thm:introbijection}. In Section~\ref{ss:KLSbackground}, we recall background on Kazhdan-Lusztig-Stanley theory. 

%In Section~\ref{ss:backgroundposet} we explain Theorem~\ref{thm:introbijection} in more detail. In Section~\ref{ss:KLSbackground} we recall background on 
% on Kazhdan-Lusztig-Stanley theory. We also refer the reader to 
% \cite{StapledonLWPosets}*{Section~2.2, Section~2.3} for background on polytopes and fans respectively. \alan{lazy?? what about proper...} 

	 \subsection{Background on fans and polytopes}\label{ss:backgroundfans}
We recall some basic facts on polytopes and fans. We also refer the reader to 
\cite{StapledonLWPosets}*{Section~2.2, Section~2.3} for more entensive background. % on polytopes and fans respectively. 
All cones in this paper are polyhedral cones, i.e., the intersection of finitely many half-spaces in a real vector space. A cone is \emph{pointed}
if its minimal face is the origin. 
Recall that a  \emph{fan} $\Sigma$ in a real vector space $V$ is a finite collection of pointed cones such that 
\begin{enumerate}
	\item if $C \in \Sigma$ and $F$ is a face of $C$, then $F \in \Sigma$, and,
	\item if $C,C' \in \Sigma$, then $C \cap C'$ is a common face of $C$ and $C'$. 
\end{enumerate}
The \emph{support} $|\Sigma|$ of $\Sigma$ is the union of the cones of $\Sigma$ in $V$.  
%The fan is \emph{full-dimensional} if the linear span of $|\Sigma|$ is $V$. 	
The \emph{face poset} $\face(\Sigma)$  is the poset of cones in $\Sigma$ ordered by inclusion.  Then $\face(\Sigma)$ is lower Eulerian with natural rank function $\rho_{\face(\Sigma)}(C) = \dim C$.

Let $\Sigma'$ be a fan in a real vector space $V'$. Let $\phi: V' \to V$ be a linear map. Then $\phi$ defines a \emph{morphism of fans} 
from $\Sigma'$ to $\Sigma$ 
if for every cone $C'$ in $\Sigma'$ there is a cone $C$ in $\Sigma$ such that $\phi(C') \subset C$. 
In that case, we have an induced function
$\sigma: \face(\Sigma') \to \face(\Sigma)$,
where $\sigma(C')$ is the smallest element of $\Sigma$ containing $\phi(C')$.
A morphism of fans is \emph{proper} if $\phi^{-1}(|\Sigma|) = |\Sigma'|$, and is \emph{surjective} if $\phi$ is surjective.
	For example, if $V = V'$ and the identity map on $V$ induces a proper morphism of fans from $\Sigma'$ to $\Sigma$, then we say that $\Sigma'$ is a \emph{refinement} of $\Sigma$.
		A proper morphism of fans is \emph{projective} if there exists
	$p: |\Sigma'| \to \R$ that is 	piecewise linear  with respect to $\Sigma'$, and  
	%	a 
	%	piecewise linear function $p$ on $\Sigma'$ 
	such that the restriction of $p$ to $\phi^{-1}(C)$ is strictly convex with respect to the restriction $\Sigma'|_{\phi^{-1}(C)}$ for every maximal cone $C$ in $\Sigma$. 
%A proper morphism of fans is \emph{projective} if there exists a 
%piecewise linear function $p$ on $\Sigma'$ such that the restriction of $p$ to $\phi^{-1}(C)$ is strictly convex with respect to the restriction $\Sigma'|_{\phi^{-1}(C)}$ for every maximal cone $C$ in $\Sigma$. 

	A \emph{polyhedral subdivision} $\cS$ of a polytope $P$ is a finite collection of polytopes whose union is $P$, such that if $Q \in \cS$, then any face of $Q$ lies in $\cS$, and if $Q' \in \cS$, then $Q \cap Q'$ is a common face of $Q$ and $Q'$. Elements of $\cS$ are called faces. 
	For example, the trivial subdivision of $P$ consists of the faces of $P$. 
	The \emph{face poset} $\face(\cS)$ of $\cS$ is the poset of elements of $\cS$ ordered by inclusion. Then $\face(\cS)$ is a lower Eulerian poset with natural rank function $\rho_{\face(P)}(F) = \dim F + 1$. If $\cS'$ is a polyhedral subdivision of $P$, then $\cS'$ is a \emph{refinement} of $\cS$ if every face in $\cS'$ is contained in a face of $\cS$. As in Example~\ref{ex:intropolytope}, there is a  corresponding refinement of fans from $\Sigma_{\cS'}$ to  $\Sigma_{\cS}$.  
	The polyhedral subdivision is regular if and only if the corresponding refinement is projective. 
	We have an induced function
	$\sigma: \face(\cS') \to \face(\cS)$,
	where $\sigma(F')$ is the smallest element of $\cS$ containing $F' \in \cS'$.

\subsection{Background on lower Eulerian posets}\label{ss:backgroundposet}
In this section, we recall some basic definitions on posets and  explain Theorem~\ref{thm:introbijection} in more detail. We refer the reader to \cite{StapledonLWPosets} for more details including proofs of the statements.
	Let $B$ be a nonempty finite poset. 
%	An interval of $B$ is a closed interval unless specified otherwise. Given elements $z \le z'$ in $B$, we write 
%	$[z,z'] = \{ z'' \in B : z \le z'' \le z' \}$ and  $[z,z') = \{ z'' \in B : z \le z'' < z' \}$. 
%%	and $(z,z') = \{ z'' \in B : z < z'' < z' \}$. 
%%	We say that $z'$ \emph{covers} $z$ if $z < z'$ and $z \le z'' \le z'$ implies that $z'' \in \{ z,z'\}$.  
%%	If $B$ contains a unique minimal element, we will denote it $\hat{0}_B$, and if $B$ contains a unique maximal element, we will denote it $\hat{1}_B$. 
A \emph{lower order ideal} of  $B$ is a subset $I \subset B$ such that if $z \le z'$ in $B$ and $z' \in I$, then $z \in I$. Similarly, an \emph{upper order ideal} of $B$ is a subset $I \subset B$ such that if $z \le z'$ in $B$ and $z \in I$, then $z' \in I$.
	A function $\rho_B: B \to \Z$ is a \emph{rank function} for $B$ if $\rho_B(z') = \rho_B(z) + 1$ whenever $z'$ covers $z$, i.e., the interval $[z,z']$ consists of two elements.
	Then $B$ is a \emph{lower Eulerian poset} if it admits a rank function 
	$\rho_B$, contains a unique minimal element, denoted $\hat{0}_B$, and for any $z < z'$ in $B$,  $\sum_{z \le z'' \le z'} (-1)^{\rho_B(z'')} = 0$. 
	% the length of the longest maximal chain in $B$. 
	If $B$ also contains a unique maximal element, denoted $\hat{1}_B$, then $B$ is an \emph{Eulerian poset}, and we write $\partial B = B \smallsetminus \{ \hat{1}_B \}$. 
	 For any integer $s$, we define a shifted rank function 
	$\rho_B[s] : B \to \Z$ by $\rho_B[s](z) = \rho_B(z) + s$ for all $z \in B$. Note that if $B$ is lower Eulerian then any two rank functions agree for $B$ up to a shift.  In this case, the \emph{natural rank function} is the unique rank function with $\rho_B(\hat{0}_B) = 0$. 	 The \emph{rank} $\rank(B)$ of $B$ is 
	the maximum of $\{ \rho_B(z) - \rho_B(\hat{0}_B) : z \in B \}$.

	\begin{example}\label{ex:boolean}
		For a nonnegative integer $n$, let $B_n$ be the poset of all (possibly empty) subsets of $[n] := \{ 1,\ldots, n\}$. 
		Then $B_n$ is Eulerian  of rank $n$ and is called the \emph{Boolean algebra} on $n$ elements. 	
		The natural rank function $\rho_{B_n}: B_n \to \Z$ is given by
		%A rank function is defined by 
		%$\rho: B \to \Z$, 
		$\rho_{B_n}(J) = |J|$  for all $J \subset [n]$.
		% For example, $B_0$ is the one element poset, and $B_1 = \{ \hat{0}, \hat{1} \}$. 
	\end{example}
	
	\begin{example}\label{ex:productsprepre}
		Let $B$ and $B'$ be posets. The \emph{direct product} is the poset $B \times B' = \{ (z,z') : z \in B, z' \in B' \}$ with $(z_1,z_1') \le (z_2,z_2')$ in $B \times B'$ if and only if $z_1 \le z_2$ in $B$ and $z_1' \le z_2'$ in $B'$.
		The \emph{pyramid} of $B$ is $\Pyr(B) = B \times B_1$. 
		If $B$ and $B'$ are lower Eulerian with rank functions $\rho_B$ and $\rho_{B'}$ respectively, then $B \times B'$ is lower Eulerian with rank function
		$\rho_{B \times B'}(z,z') = \rho_B(z) + \rho_{B'}(z')$ for all $z \in B$ and $z' \in B'$.  For example, $B_{n + 1} = \Pyr(B_n)$.
		If $\sigma: X \to Y$ and $\sigma': X' \to Y'$ are functions between posets, we define $\sigma \times \sigma': X \times X' \to Y \times Y'$ by $(\sigma \times \sigma')(x,x') = (\sigma(x),\sigma(x'))$ for all $(x,x') \in X \times X'$. 
	\end{example}
	
	Let $\sigma: X \to Y$ be a function between lower Eulerian posets $X$ and $Y$ with rank functions $\rho_X$ and $\rho_Y$ respectively.  Then  $\sigma$ is a \emph{strong formal subdivision} \cite[Definition~3.17]{KatzStapledon16} if it satisfies the following conditions
	\begin{enumerate}
		\item 	 $\sigma$ is \emph{order-preserving} in the sense that $x \le x'$ in $X$ implies that $\sigma(x) \le \sigma(x')$ in $Y$.
		\item 	$\sigma$  is \emph{rank-increasing} in the sense that $\rho_X(x) \le \rho_Y(\sigma(x))$ for all $x$ in $X$.
		\item $\sigma$ is \emph{strongly surjective} in the sense that  $\sigma$ is surjective and for all $x \in X$ and $y \in Y$ with $\sigma(x) \leq y$, there exists $x' \in X$ such that $x \le x'$, $\rho_X(x') = \rho_Y(y)$, and $\sigma(x') = y$.
		\item For all $x \in X$ and $y \in Y$ such that $\sigma(x) \le y$, we have
		\begin{equation}\label{eq:strongsubdivisionequality}
			\sum_{ \substack{x \le x' \in X \\ \sigma(x') = y} } (-1)^{\rho_Y(y) - \rho_X(x')} = 1.
		\end{equation}
	\end{enumerate}
	Assume that $\sigma$ is a strong formal subdivision.
	The corresponding \emph{non-Hausdorff mapping cylinder} $\Cyl(\sigma)$ \cite{BMSimpleHomotopy}*{Definition~3.6} is the poset with elements given by the disjoint union %$X \sqcup Y$ 
	of $X$ and $Y$, 
	with $z \le z'$ in $\Cyl(\sigma)$ if and only if one of the following conditions hold:
	\begin{enumerate}
		\item $z,z' \in X$ and $z \le z'$ in $X$,
		\item $z,z' \in Y$ and $z \le z'$ in $Y$,
		\item $z \in X$, $z' \in Y$, and $\sigma(z) \le z'$ in $Y$.
	\end{enumerate}
	Define $\rho_{\Cyl(\sigma)}: \Cyl(\sigma) \to \Z$ by 
	\begin{equation}\label{eq:rhoCyl}
		\rho_{\Cyl(\sigma)}(z) =  \begin{cases}
			\rho_X(z) &\textrm{ if } z \in X, \\
			\rho_Y(z) + 1 &\textrm{ if } z \in Y. 
		\end{cases}
	\end{equation}
	Then %Theorem~\ref{thm:introbijection} says that 
	$\Cyl(\sigma)$ is a lower Eulerian poset with rank function $\rho_{\Cyl(\sigma)}$, $\hat{0}_Y \neq \hat{0}_{\Cyl(\sigma)}$ is join-admissible, and $\sigma$ corresponds to the
	triple $(\Cyl(\sigma),\rho_{\Cyl(\sigma)},\hat{0}_Y)$ under the bijection of Theorem~\ref{thm:introbijection}. 
	
	Conversely, consider a triple $(\Gamma,\rho_\Gamma,q)$,
	 where $\Gamma$ is a lower Eulerian poset
	with 
	rank function $\rho_\Gamma$, and $q$ is a join-admissible element of $\Gamma$ with $q \neq \hat{0}_\Gamma$. 	Let $Y = \{ z \in \Gamma : q \le z \}$ and $X = \Gamma \smallsetminus Y$. Define $\rho_X$ and $\rho_Y$ to be the restrictions of  $\rho_\Gamma$ to $X$ and $Y$ shifted by $0$ and $-1$ respectively. Define $\sigma: X \to Y$ by 
	  $\sigma(x) = x \vee q$ for all $x \in X$. 
	  	Then %Theorem~\ref{thm:introbijection} says that 
	  	$X$ and $Y$ are lower Eulerian with rank functions $\rho_X$ and $\rho_Y$ respectively, and $\sigma$ is the strong formal subdivision corresponding to $(\Gamma,\rho_\Gamma,q)$ under the bijection of  Theorem~\ref{thm:introbijection}.

%t^{r_\Gamma(x,y) - 1}p(x,y;t^{-1})$, i.e., $p(x,y)_i = p(x,y)_{r_\Gamma(x,y) - 1 - i}$ for all $i$. Define an element $\Delta p \in \II_{1/2}(\Gamma)$ by 
%$(\Delta p)_X = (\Delta p)_Y = 0$, and 
%$$(\Delta p)(x,y) = p(x,y)_0 + \sum_{i = 1}^{\lfloor \frac{r_\Gamma(x,y) - 1}{2} \rfloor} (p(x,y)_i - p(x,y)_{i - 1})t^i,$$
%for all $x \in X$ and $y \in Y$  such that $\sigma(x) \le y$.
%Background on 

\subsection{Background on Kazhdan-Lusztig-Stanley theory}\label{ss:KLSbackground}

We recall some background on Kazhdan-Lusztig-Stanley (KLS) theory. We refer the reader to 
 \cite{ProudfootAGofKLSpolynomials} for more details. 
Let $B$ be a poset. Let $\Int(B) = \{ [z,z']  : z  \le z' \in B \}$ denote the set of closed intervals of $B$. The \emph{incidence algebra} $I(B)$ of $B$ is the set of functions from $\Int(B)$ to $\Z[t]$. 
Given a function 
$p: \Int(B) \to \Z[t]$ and an interval $[z,z'] \in \Int(B)$, we often write
$p(z,z') = p(z,z';t)$ to denote $p([z,z']) \in \Z[t]$.
The incidence algebra has the structure of 
%a ring. 
an associative $\Z[t]$-algebra.
Given elements $p,p' \in I(B)$, addition is defined by 
$(p + p')(z,z') = p(z,z') + p'(z,z')$. Multiplication is given by convolution: 
\[
(p \cdot p')(z,z')= \sum_{z \le z'' \le z'} p(z,z'') p'(z'',z').
\]
The identity element $\delta_B \in I(B)$ is defined by
\[
\delta_B(z,z') = \begin{cases}
	1 &\textrm{if } z = z', \\
	0 &\textrm{otherwise. }
\end{cases}
\]
%By associativity, if $p \in I(P)$ is invertible, then the left and right inverses are unique and they coincide.
The $\Z[t]$-module structure is defined as follows:
given $a \in \Z[t]$ and $p \in I(B)$, $(a \cdot p)(z,z') = a p(z,z')$ for all $z \le z' \in B$. 

\begin{example}\label{ex:productspre}
	Let $B$ and $B'$ be posets. 
%	The \emph{direct product} is the poset $B \times B' = \{ (z,z') : z \in B, z' \in B' \}$ with $(z_1,z_1') \le (z_2,z_2')$ in $B \times B'$ if and only if $z_1 \le z_2$ in $B$ and $z_1' \le z_2'$ in $B'$.
%	The \emph{pyramid} of $B$ is $\Pyr(B) = B \times B_1$. 
	Given $p \in I(B)$ and $p' \in I(B')$, define $p \times p' \in I(B \times B')$ by  $(p \times p')((z_1,z_1'),(z_2,z_2')) = p(z_1,z_2) p(z_1',z_2')$ for all intervals $[(z_1,z_1'),(z_2,z_2')]$ in $B \times B'$. 
	Given $p_1,p_2 \in I(B)$ and $p_1',p_2' \in I(B')$, we have
	$(p_1 \times p_1') \cdot (p_2 \times p_2') = (p_1 \cdot p_2) \times (p_1' \cdot p_2') \in I(B \times B')$. 
	
\end{example}

An element $r_B \in I(B)$ is a \emph{weak rank function} in the sense of Brenti \cite{BrentiTwistedIncidence}*{Section~2} if
\begin{enumerate}
	\item $r_B(z,z') \in \Z_{\ge 0} \subset \Z[t]$ for all $z \le z' \in B$,
	\item $r_B(z,z') > 0$ for all $z < z' \in B$,
	\item $r_B(z,z') = r_B(z,z'') + r_B(z'',z')$ for all $z \le z'' \le z' \in B$. 
\end{enumerate}
If $B'$ is a subposet of $B$, then the restriction of a weak rank function $r_B \in I(B)$ to $B'$ is the weak rank function $r_{B'} \in I(B')$ defined by $r_{B'}(z,z') = r_B(z,z')$ for $z \le z' \in B'$. 

For example, suppose that $B$ is a lower Eulerian poset with rank function $\rho_B$. 
With a slight abuse of notation, we will consider the associated weak rank function $r_B = \rho_B$ defined by 
 $\rho_B(z,z') = \rho_B(z') - \rho_B(z)$ for all $z \le z'  \in B$. 
 Since any two rank functions for $B$ agree up to a shift, $r_B$ is  independent of the choice of rank function $\rho_B$.
 We will call $r_B = \rho_B$ the \emph{natural weak rank function}. 
 
% 
% Observe that any shifted rank function $\rho_B[j]$, for some $j \in \Z$, determines the same weak rank function. If we further assume that $B$ contains a unique minimal element, then there is a unique rank function up to  a shift, and the corresponding weak rank function is independent of the choice of rank function $\rho_B$.

%A function $\sigma: B \to B'$ between posets $B$ and $B'$ is 
%\emph{strictly order-preserving} if $\sigma(z) < \sigma(z')$ in $Y$ whenever $z < z'$ in $B$.  Let $r_B: B \to \Z$ be a strictly order-preserving function. 
%For any integer $s$, we define a shifted function 
%$r_B[s] : B \to \Z$ by $r_B[s](z) = r_B(z) + s$ for all $z \in B$.
%With a slight abuse of notation, consider the element $r_B \in I(B)$ 
%defined by  $r_B(z,z') = r_B(z') - r_B(z) \in \Z \subset \Z[t]$ for all $z \le z'$ in $B$. Then $r_B \in I(B)$ is a \emph{weak rank function} in the sense of Brenti \cite{BrentiTwistedIncidence}*{Section~2}. 
%The discussion below only depends on the associated weak rank function, which is invariant under replacing $r_B$ by $r_B[s]$. 
%%Observe that if $B$ has a unique minimal element $\hat{0}$, then $r: B \to \Z$ is determined by its corresponding weak rank function up to an integer shift. 

Fix a weak rank function $r_B \in I(B)$. Let $\deg(a)$ denote the degree of an element $a \in \Z[t]$. Here we set $\deg(0) = - \infty$.  Consider the following subring of $I(B)$: % defined by
\[
\II(B) = \{ p \in I(B) : \deg(p(z, z')) \le r_B(z, z') \textrm{ for all } z \le z' \in B \}.
\]
Then $\II(B)$
admits a ring involution 
$p \mapsto p^{\rev}$, where
$(p^{\rev})(z, z';t) := t^{r_B(z, z')} p(z, z';t^{-1})$ for any $z \le z'$ in $B$. 
An element $p \in \II(B)$ is \emph{symmetric} if $p = p^{\rev}$, and
%An element $p \in \II(B)$ is 
\emph{antisymmetric} if $p = -p^{\rev}$. 
%\alan{TODO: truncation map from $\II(B)$ to $\II_{1/2}(B)$ what it means for anti-symmetric}
%\begin{definition}
%	Let $p \in \II(B)$ be antisymmetric. 
%	Let $\widetilde{p} \in \II_{1/2}(B)$ be the unique element such that 
%	$p = \widetilde{p} - \widetilde{p}^{\rev}$. Explictly, for any $z \le z' \in B$, if $p(z,z';t) = \sum_{i = 0}^{r_B(z,z')} p(z,z')_i t^i$, then 
%	 $\widetilde{p}(z,z';t) = \sum_{i = 0}^{r_B(z,z')/2} p(z,z')_i t^i$. 
%\end{definition}
We will use the following lemma (c.f. \cite{StanleyEnumerative}*{Proposition~3.6.2}). 

\begin{lemma}\label{lem:invertible}\cite{ProudfootAGofKLSpolynomials}*{Lemma~2.1}
	An element $p \in I(B)$ is invertible if and only if $p(z,z) \in \{ -1, 1\} \subset \Z[t]$ for all $z \in B$. If $p  \in \II(B) \subset I(B)$ is invertible, then $p^{-1} \in \II(B)$. 
\end{lemma}

Let 
$U(B) = \{ p \in I(B) : p(z,z) = 1 \textrm{ for all } z \in B \}$. Observe that $U(B)$ is closed under multiplication, and  all elements of $U(B)$ are invertible. 
%, and $U(B)$ is a subring of $I(B)$. 
%We need the following notation.   
Let
\[
\II_{1/2}(B) = \{ p \in \II(B) : \deg(p(z, z')) < r_B(z, z')/2 \textrm{ for all } z < z' \in B \}.
\]
Observe that $\II_{1/2}(B)$ is a subring of $\II(B)$. This notation is slightly different to that in \cite{ProudfootEquivariantKLS}, where $\II_{1/2}(B)$ equals what we denote $\II_{1/2}(B) \cap U(B)$.
Recall that given a polynomial $a \in \Z[t]$,  $a_i$ denotes the coefficient of $t^i$ in $a$. 
Define a $\Z$-linear map 
\[
\Delta: \II(B) \to \II_{1/2}(B),   
\]
\[
p \mapsto \Delta p,
\]
where for any $z \le z' \in B$,
\begin{equation}\label{eq:Delta}
	(\Delta p)(z,z') = \begin{cases}
		p(z,z')_0 + \sum_{i = 1}^{\lfloor \frac{r_B(z,z') - 1}{2} \rfloor} (p(z,z')_i - p(z,z')_{i - 1})t^i &\textrm{ if } z < z', \\
		0 &\textrm{ if } z = z'.
	\end{cases}
\end{equation}
That is, $(\Delta p)(z,z')$ is the truncation of $(1 - t) p(z,z')$ to a polynomial with degree strictly less than $r_B(z,z')/2$.  

Suppose $p \in \II(B)$ satisfies $p(z,z) = 0$ for $z \in B$ and $p(z,z') = 
t^{r_B(z,z') - 1}p(z,z';t^{-1})$ for $z < z' \in B$. Equivalently, assume that $(t - 1) \cdot p$ lies in $\II(B)$ and is antisymmetric. Then $\Delta p$ is the unique element in $\II_{1/2}(B)$ satisfying 
\begin{equation}\label{eq:tildeell}
	(\Delta p)^{\rev} - \Delta p = (t - 1) \cdot p. 
\end{equation}
In particular, in this case, we may recover $p$ from $\Delta p$, and $\Delta p$ may be viewed as an alternative encoding of $p$. For example,
 the coefficients of $(\Delta p)(z,z')$ are nonnegative if and only if the coefficients of $p(z,z')$ are nonnegative and unimodal.

\begin{definition}\label{def:kernel}
	An element $\kappa_B \in \II(B) \cap U(B)$ is a \emph{$B$-kernel} if $\kappa_B^{-1} =  \kappa_B^{\rev}$. 
\end{definition}

We will need the following lemma. 

\begin{lemma}\cite{FMVChowFunctions}*{Lemma~3.1}\label{lem:poleatone}
	Let $\kappa_B$ be a $B$-kernel. 	For every $x < x'$ in $B$, $\kappa_B(x,x')$ is divisible by $t - 1$.  
\end{lemma}

The following theorem is a  generalization of  a result of Stanley \cite{Stanley92}*{Corollary~6.7}. 

\begin{theorem}\label{thm:existenceofg}\cite{ProudfootAGofKLSpolynomials}*{Theorem~2.2} 
	Let $B$ be a poset with weak rank function $r_B$.  Let $\kappa_B$ be a $B$-kernel. Then there exist unique elements $f_B,g_B \in \II_{1/2}(B) \cap U(B)$ such that $f_B^{\rev} = \kappa_B \cdot f_B$  and 
	$g_B^{\rev} = g_B \cdot \kappa_B$.
\end{theorem}

The elements $f_B$ and $g_B$ in Theorem~\ref{thm:existenceofg} are called the 
\emph{right Kazhdan-Lusztig-Stanley function} and \emph{left Kazhdan-Lusztig-Stanley function} respectively. 
%The polynomials 
%$\{ f_B(z,z') : z \le z' \in B \}$ and $\{ g_B(z,z') : z \le z' \in B \}$ are called the \emph{right Kazhdan-Lusztig-Stanley polynomials} and \emph{left Kazhdan-Lusztig-Stanley  polynomials} respectively.
We will need the following remark.

\begin{remark}\label{rem:existenceofgoverk}
	Recall that given a field $k$ and a $\Z$-module $A$, we write $A_k = A \otimes_\Z k$. Then $I(B)_k$ has  the structure of a $k[t]$-algebra,  $\II(B)_k$ has the structure of a $k$-algebra  with subalgebra $\II_{1/2}(B)_k$, and the discussion above still holds. In particular,  the 
	statement of Theorem~\ref{thm:existenceofg} still holds after tensoring with $k$. 		 That is, $f_B,g_B \in \II_{1/2}(B) \cap U(B)$  are the unique elements in $\II_{1/2}(B)_k \cap U(B)_k$ such that $f_B^{\rev} = \kappa_B \cdot f_B$  and 
	$g_B^{\rev} = g_B \cdot \kappa_B$ in $\II(B)_k$.

\end{remark} 

\begin{definition}\cite{ProudfootAGofKLSpolynomials}*{Section~2.3}\label{def:Zfunction}
	The \emph{$Z$-function} is $Z_B = g_B \cdot \kappa_B \cdot f_B \in \II(B) \cap U(B)$. 
	%The polynomials 
	%$\{ Z_B(z,z') : z \le z' \in B \}$ are called the \emph{$Z$-polynomials}. 
\end{definition}

Proudfoot showed that $Z_B = g_B^{\rev} \cdot f_B = g_B \cdot f_B^{\rev}$
\cite{ProudfootAGofKLSpolynomials}*{Proposition~2.6}. 
 In particular, 
$Z_B$ is symmetric, i.e., $Z_B = Z_B^{\rev}$. 

\begin{example}\label{ex:simple}
	If $\kappa_B = \delta_B$ then $\kappa_B^{\rev} = f_B = g_B = Z_B = \delta_B$.  
\end{example}

\begin{example}\label{ex:simpleB1}
	Consider $B_1 = \{ \hat{0}, \hat{1} \}$ and let $r_{B_1}$ be the natural weak rank function. 
	 %for a rank function $\rho_{B_1}$ for $B_1$. 
	 Then there exists $\lambda \in \Z$ such that $\kappa_{B_1}(B_1) = \lambda(t - 1)$. We have 
	$g_{B_1}(B_1) = f_{B_1}(B_1) = \lambda$ and $Z_{B_1}(B_1) = \lambda(1 + t)$. 
\end{example}

\begin{example}\label{ex:products}
	Consider posets $B$ and $B'$ with weak rank functions $r_B$ and $r_{B'}$ respectively. Let $\kappa_B$ be a $B$-kernel, and let 
	$\kappa_{B'}$ be a $B'$-kernel.
%	Consider the right and left Kazhdan-Lusztig-Stanley functions $f_B, f_{B'}$ and $g_B, g_{B'}$ respectively, and the $Z$-functions $Z_B, Z_{B'}$ associated to $B$ and $B'$ respectively.
		Fix the weak rank function $r_{B \times B'} \in I(B \times B')$ defined by $r_{B \times B'}(z,z') = r_B(z,z') + r_{B'}(z,z')$. 
		With the notation from Example~\ref{ex:productspre}, let  $\kappa_{B \times B'} = \kappa_B \times \kappa_{B'}$. Then $\kappa_{B \times B'}$ is a 
%		Then with  the notation from Example~\ref{ex:productspre}, $\kappa_{B \times B'} := \kappa_B \times \kappa_{B'}$ is a 
	$(B \times B')$-kernel  with 
	right and left Kazhdan-Lusztig-Stanley functions $f_{B \times B'} = f_B \times f_{B'}$ and $g_{B \times B'} = g_B \times g_{B'}$ respectively, and $Z$-function $Z_{B \times B'} = Z_{B} \times Z_{B'}$.

\end{example}

\begin{example}\label{ex:lowdegreeterms}
	To help the reader with computations, we consider some low and high degree terms of these invariants. 
	Recall that given a polynomial $a \in \Z[t]$,  $a_i$ denotes the coefficient of $t^i$ in $a$. Then for any $z \le z'$ in $B$, comparing highest degree terms in the expressions $g_B^{\rev} = g_B \cdot \kappa_B$, $f_B^{\rev} = \kappa_B \cdot f_B$, and $Z_B = g_B \cdot \kappa_B \cdot f_B$ evaluated at $[z,z']$ gives
	\[
	g_B(z,z')_0 = f_B(z,z')_0 = \kappa_B(z,z')_{r_B(z,z')} = Z_B(z,z')_0 = Z_B(z,z')_{r_B(z,z')}. 
	\]
	In particular, if $r_B(z,z') \le 2$, then $g_B(z,z') = f_B(z,z') = \kappa_B(z,z')_{r_B(z,z')}$. Assume that $r_B(z,z') > 2$. Then comparing second highest degree terms in the expressions $g_B^{\rev} = g_B \cdot \kappa_B$ and  $f_B^{\rev} = \kappa_B \cdot f_B$ evaluated at $[z,z']$ gives
	\[
		g_B(z,z')_1 = \kappa_B(z,z')_{r_B(z,z') - 1} + \sum_{\substack{z \le z'' \le z' \\ r_B(z,z'') = 1 }} \kappa_B(z'',z')_{r_B(z'',z')},
	\]
	\[
	f_B(z,z')_1 = \kappa_B(z,z')_{r_B(z,z') - 1} + \sum_{\substack{z \le z'' \le z' \\ r_B(z'',z') = 1 }} \kappa_B(z,z'')_{r_B(z,z'')}.
	\]
\end{example}

For the remainder of the section, assume that $B$ is a lower Eulerian poset.  %with the natural weak rank function $\rho_B$. %$\rho_B: B \to \Z$. 
%Assume further that $B$ is a ranked poset with rank function  $\rho_B: B \to \Z$. 
Note that we do not assume that $r_B$ equals  the natural weak rank function $\rho_B$. 
% = \rho_B$.
%Note that $r_B$ need not equal $\rho_B$ up to a shift. 
Consider the ring involution of $\II(B)$ defined by
$p \mapsto \widehat{p}$,
where $\widehat{p}(z, z') = (-1)^{\rho_B(z,z')} p(z, z')$ for $z \le z'$ in $B$.
Observe that this involution commutes with the involution $p \mapsto p^{\rev}$. % and is independent of the choice of $\rho_B$. 
%Also, the involution is determined by $B$, and does not depend on $r_B$. 
%Also,  
%if $B$ contains a unique minimal element $\hat{0}$, then $\rho_B$ is determined up to a shift, and 
%the involution $p \mapsto \widehat{p}$ is independent of the choice of $\rho_B$. 
We say that $p \in \II(B)$ is \emph{rank alternating} if $p^{\rev} = \widehat{p}$.
This is inspired by \cite{ProudfootAGofKLSpolynomials}*{Section~2.4}, where there is no assumption that $B$ is lower Eulerian,  $\widehat{p}$ is defined with  $\rho_B$ is replaced by $r_B$ above, and an element is called alternating if $p^{\rev} = \widehat{p}$.

%We note that the definition of $\widehat{p}$ is slightly different to that in \cite{ProudfootAGofKLSpolynomials}*{Section~2.4}; there $\rho_B$ is replaced by $r_B$ above, and an element is called alternating if $p^{\rev} = \widehat{p}$.
%%We note that the definitions of $\widehat{p}$ and alternating are slightly different to those in \cite{ProudfootAGofKLSpolynomials}*{Section~2.4}, where $\rho_B$ is replaced by $r_B$ above.  

We need the following lemma. 
	The proof is essentially the same as the proof of \cite{ProudfootAGofKLSpolynomials}*{Proposition~2.11}.
%The proof follows in the same way as  \cite{ProudfootAGofKLSpolynomials}*{Proposition~2.11}.
%  which is essentially \cite{ProudfootAGofKLSpolynomials}*{Proposition~2.11}. Since our  definitions are slightly different to those used in \cite{ProudfootAGofKLSpolynomials}, we reproduce the proof for the benefit of the reader. 

\begin{lemma}\label{lem:inverse}
	Let $B$ be a lower Eulerian poset with weak rank function $r_B$ and natural weak rank function $\rho_B$. 
	% \in I(B)$. % and rank function  $\rho_B: B \to \Z$. 
	Suppose that $\kappa_B$ is a rank alternating $B$-kernel. Then 
	%the corresponding Kazhdan-Lusztig-Stanley functions satisfy 
	$\widehat{f_B} = g_B^{-1}$ and $\widehat{g_B} = f_B^{-1}$.
\end{lemma}
\begin{proof} 
 Observe that $\II_{1/2}(B) \cap U(B)$ is closed under multiplication, and the only symmetric function in $\II_{1/2}(B) \cap U(B)$ is the identity element $\delta_B \in \II(B)$. 	We have $\widehat{g_B},f_B \in \II_{1/2}(B) \cap U(B)$, and hence $\widehat{g_B} \cdot f_B \in \II_{1/2}(B) \cap U(B)$. 
Since $g_B^{\rev} = g_B \cdot \kappa_B$, we have $\widehat{g_B}^{\rev} = \widehat{g_B} \cdot \kappa_B^{\rev}$. Then 
\[
(\widehat{g_B} \cdot f_B)^{\rev} = \widehat{g_B} \cdot \kappa_B^{\rev} \cdot f_B^{\rev} = \widehat{g_B} \cdot (\kappa_B \cdot f_B)^{\rev} = \widehat{g_B} \cdot f_B. 
\]
Hence $\widehat{g_B} \cdot f_B$ is symmetric, and $\widehat{g_B} \cdot f_B = \delta_B$, as desired.
Also $g_B \cdot \widehat{f_B} = \widehat{\delta_B} = \delta_B$. 
\end{proof}

We say that  an element $p \in \II(B)$ is \emph{multiplicative} if
$p(z,z') = p(z,z'')p(z'',z')$ for all $z \le z'' \le z'$ in $B$.  The following lemma essentially follows from \cite{Stanley92}*{Proposition~7.1}. 

\begin{lemma}\label{lem:multaltiskernel}
	Let $B$ be a lower Eulerian poset with weak rank function $r_B$. 	Consider an element  $\kappa_B \in \II(B) \cap U(B)$	that is multiplicative and rank alternating. Then $\kappa_B$ is a $B$-kernel.
	
%	Let $B$ be a ranked poset with weak rank function $r_B \in I(B)$ and rank function  $\rho_B: B \to \Z$.
%	Consider an element  $\kappa_B \in \II(B) \cap U(B)$	that is multiplicative and rank alternating. If $B$ is locally Eulerian then $\kappa_B$ is a $B$-kernel.
%	
%	Let $B$ be a ranked poset with weak rank function $r_B \in I(B)$ and rank function  $\rho_B: B \to \Z$.
%	Consider an element  $\kappa_B \in \II(B) \cap U(B)$	that is multiplicative and rank alternating. If $B$ is locally Eulerian then $\kappa_B$ is a $B$-kernel.
\end{lemma}
\begin{proof}
	Let $\rho_B$ be the natural weak rank function for $B$. For $z \le z'$ in $B$,  
	we compute
		\begin{align*}
		(\kappa_B \cdot \kappa_B^{\rev})(z,z') &= \sum_{z \le z'' \le z'} \kappa_B(z,z'')\kappa_B^{\rev}(z'',z') \\
		&= \sum_{z \le z'' \le z'} (-1)^{\rho_B(z'',z')} \kappa_B(z,z'')\kappa_B(z'',z') \\
		&= \kappa_B(z,z') \sum_{z \le z'' \le z'} (-1)^{\rho_B(z'',z')}  \\
		&= \delta_B(z,z'). 
	\end{align*}

\end{proof}

Kazhdan-Lusztig-Stanley theory for the following important example was studied in detail in \cite{KatzStapledon16}*{Section~3} and \cite{Stanley92}*{Section~7}.

\begin{example}\label{ex:t-1case}
	Let $B$ be a lower Eulerian poset.
	% with rank function $\rho_B$. 
	Let $r_B = \rho_B$ be the natural weak rank function, and let 	$\kappa_B(z,z') = (t - 1)^{\rho_B(z,z')}$ for all $z,z' \in B$. Then $\kappa_B \in \II(B)$ is multiplicative and rank alternating, and hence is a $B$-kernel by Lemma~\ref{lem:multaltiskernel}. 
	In this case, $\kappa_B$ is called the \emph{Eulerian kernel}. 
	
	If $B$ is Eulerian, then we will use the notation $f(B)$, $g(B)$, $Z(B)$  for  $f_B(B) = f_B(\hat{0}_B,\hat{1}_B)$, $g_B(B) = g_B(\hat{0}_B,\hat{1}_B)$, and $Z_B(B) = Z_B(\hat{0}_B,\hat{1}_B)$ respectively.  
%    To distinguish this example, we will write $\rho = r_B = \rho_B$, $\kappa = \kappa_B$, $f = f_B$, $g = g_B$, and $Z = Z_B$. 
    We also recall the definition of the \emph{$h$-polynomial} $h(B) \in \Z[t]$ of a lower Eulerian poset $B$ \cite{Stanley92}*{Example~7.2}. Let $\rank(B) = n$, then 
    \[
    t^n h(B;t^{-1}) = \sum_{z \in B} g_B(\hat{0}_B,z;t)(t - 1)^{n - \rho_B(\hat{0}_B,z)}.
    \]
    For example, if $B$ is Eulerian, then $g^{\rev} = g \cdot \kappa$ implies that $h(B) = g(B)$. If $B$ is an Eulerian poset of rank $n > 0$, then 
    $(1 - t)h(\partial B;t) = g(B;t) - t^n g(B;t^{-1})$ (see, for example, \cite{KatzStapledon16}*{Example~3.14}). Here $h(\partial B;t)$ is sometimes  called the \emph{toric $h$-polynomial} of $B$ (see, for example, \cite{FREulerian}*{Section~2.2}). 
    
    If $B$ and $B'$ are lower Eulerian posets, then it follows from Example~\ref{ex:products} and the above definition that 
    $h(B \times B') = h(B)h(B')$. 
\end{example}

\begin{example}\label{ex:lowdegreetermsspecial}
	Consider the setup of Example~\ref{ex:t-1case}. Then by Example~\ref{ex:lowdegreeterms}, for any $z \le z'$ in $B$,
	\[
	g_B(z,z')_0 = f_B(z,z')_0 = Z_B(z,z')_0 =  Z_B(z,z')_{\rho(z,z')} = 1. 
	\]
		In particular, if $\rho_B(z,z') \le 2$, then $g_B(z,z') = f_B(z,z') = 1$.
	For $\rho_B(z,z') > 2$, we have 
	\[
	g_B(z,z')_1 = |\{ z \le z'' \le z' : \rho_B(z,z'') = 1 \}| - \rho_B(z,z'),
	\]
	\[
	f_B(z,z')_1 = |\{ z \le z'' \le z' : \rho_B(z'',z') = 1 \}| - \rho_B(z,z').
	\]
	Also, $h(B)_0 = 1$ and $h(B)_1 = |\{ z \in B : \rho_B(\hat{0}_B,z) = 1 \}| - \rank(B)$ (see, for example, \cite{KatzStapledon16}*{Example~3.13}). 
\end{example}

\begin{example}\label{ex:t-1caseBn}
	Consider Example~\ref{ex:t-1case} with $B = B_n$. 
%	 with the natural rank function and the kernel of Example~\ref{ex:t-1case}. 
	 Then $f(B_n) = g(B_n) = 1$, and $Z(B_n) = (1 + t)^n$ (c.f. Example~\ref{ex:simpleB1}). 
	 %See, for example, \cite{StanleyEnumerative}*{Example~3.16.8}. 
	 This is well-known and follows from Example~\ref{ex:products} using the $B_1$ case in Example~\ref{ex:simpleB1} and the fact that $B_{n + 1} = \Pyr(B_n)$. Also, $h(B_n) =  1$ and, for $n > 0$, $h(\partial B_n) = 1 + t + \cdots + t^{n - 1}$. 
\end{example}

\section{Kazhdan-Lusztig-Stanley theory for %subdivisions of
	 lower Eulerian posets}\label{sec:lowerEulerian}

In this section, we further develop KLS theory for lower Eulerian posets. 
In Section~\ref{ss:statements} we state our main results. Examples will be given in Section~\ref{ss:examplesmain}, and the proofs will be given in Section~\ref{ss:proofsmain}. 	
%Recall that given a polynomial $a \in \Z[t]$,  $a_i$ denotes the coefficient of $t^i$ in $a$.

\subsection{Statements of main results}\label{ss:statements}

We fix the following setup throughout the section. 
	Let $\sigma: X \to Y$ be a strong formal subdivision between lower Eulerian posets with rank functions $\rho_X$ and $\rho_Y$ respectively,
corresponding under Theorem~\ref{thm:introbijection} to a triple $(\Gamma, \rho_\Gamma, q)$, where $\Gamma$ is a lower Eulerian poset
with 
rank function $\rho_\Gamma$, and $q$ is a join-admissible element of $\Gamma$ with $q \neq \hat{0}_\Gamma$. See Section~\ref{ss:backgroundposet} for details. Observe that $\Gamma$ is Eulerian if and only if $Y$ is Eulerian.

%$\Gamma = \Cyl(\sigma)$, $\rho_\Gamma$ determined by \eqref{eq:rhoCyl}, and $q = \hat{0}_Y$. 
%Recall that $X = \Gamma \smallsetminus \Gamma_{\ge q}$, $Y = \Gamma_{\ge q}$, and $\sigma(x) = x \vee q$ for all $x \in X$. 

Fix  a weak rank function $r_\Gamma \in I(\Gamma)$.
Observe that $r_\Gamma$ restricts to weak rank functions on $X$ and $Y$. 
We may then define $\II(X), \II(Y), \II(\Gamma)$ and related invariants as in Section~\ref{ss:KLSbackground}. 
Fix an element $\kappa_\Gamma \in \II(\Gamma) \cap U(\Gamma)$ that is multiplicative and rank alternating. By Lemma~\ref{lem:multaltiskernel}, $\kappa_\Gamma$ is a $\Gamma$-kernel. Moreover, the restrictions $\kappa_X \in \II(X)$ and $\kappa_Y \in \II(Y)$ of $\kappa_\Gamma$ to $X$ and $Y$ respectively are multiplicative and rank alternating, and are
an $X$-kernel  and a $Y$-kernel respectively. Then $f_\Gamma,g_\Gamma, Z_\Gamma \in \II(\Gamma)$ restrict to $f_X,g_X, Z_X \in \II(X)$ respectively, and restrict to $f_Y,g_Y, Z_Y \in \II(Y)$ respectively. 

%Recall that $\Gamma$ is a lower Eulerian poset. 
Our most important example is when 
 $r_B = \rho_B$ is the natural weak rank function, and $\kappa_B$ is the Eulerian kernel (see Example~\ref{ex:t-1case}). 
%A crucial example is when $r_\Gamma$ and $\kappa_\Gamma$ are given by Example~\ref{ex:t-1case}. That is, $r_\Gamma = \rho_\Gamma$ is induced by a rank function $\rho_\Gamma$  for $\Gamma$, and 	$\kappa_\Gamma(z,z') = (t - 1)^{r_\Gamma(z,z')}$ for all $z,z' \in \Gamma$. 
We welcome the reader to read this section with this example in mind, but 
note that it will be crucial for applications in Section~\ref{sec:equivariant} to
develop the theory in slightly greater generality (see Remark~\ref{rem:rankwarning}). The key property we will need is that the restriction of a
weak rank function to a subposet is still a weak rank function. In contrast, the natural weak rank function is not  preserved under restriction. 
% rank functions are not preserved under restriction. 

We next introduce some notation to be used throughout. 
Given $p \in \II(\Gamma)$ and a subset $S \subset \Int(\Gamma)$, define 
$p|_S \in \II(\Gamma)$ by 
\[
p|_{S}(z,z') = \begin{cases}
	p(z,z') &\textrm{ if } [z,z'] \in S, \\
	0 &\textrm{ otherwise.}
\end{cases}
\]
%We use the following shorthand notations
For convenience, we also set
\[
p|_X := p|_{\{ [z,z'] \in \Int(\Gamma) : z,z' \in X \} },
\]
\[
p|_Y := p|_{\{ [z,z'] \in \Int(\Gamma) : z,z' \in Y \} },
\]
\[
p|_{X/Y} := p|_{\{ [z,z'] \in \Int(\Gamma) : z \in X, z' \in Y \} },
\]
\[
p|_{(X/Y)^\circ} := p|_{\{ [z,z'] \in \Int(\Gamma) : z \in X, z' = \sigma(z) \in Y\} }.
\]
Observe that $p = p|_X + p|_Y + p|_{X/Y}$.

We next introduce $h$-polynomials and local $h$-polynomials. 
When $r_\Gamma$ is the natural weak rank function and $\kappa_\Gamma$ is the Eulerian kernel, 
%In the case when $r_\Gamma$ and $\kappa_\Gamma$ are given by Example~\ref{ex:t-1case}, 
these reduce to the usual definitions (see \cite{KatzStapledon16}*{Definition~4.1}).

\begin{definition}\label{def:hellpolynomial}
	With the notation above, define elements $h_{\sigma},\ell_{\sigma} \in \II(\Gamma)$ by
	\[
	(t - 1) \cdot h_\sigma = g_\Gamma \cdot \kappa_\Gamma|_{(X/Y)^\circ} = g_\Gamma|_X \cdot \kappa_\Gamma|_{(X/Y)^\circ}, 
	\]
	\[
	\ell_\sigma = h_\sigma \cdot g_\Gamma^{-1}. 
	\]
	For any   $x \in X$ and $y \in Y$ such that $\sigma(x) \le y$, 
		the polynomials $h_{\sigma}(x,y)$ and $\ell_{\sigma}(x,y)$ are called the 
	\emph{$h$-polynomial} and \emph{local $h$-polynomial} respectively associated to the interval $[x,y]$  in $\Gamma$. When $\Gamma$ is Eulerian, we call  $h_{\sigma}(\Gamma) = h_{\sigma}(\hat{0}_X,\hat{1}_Y)$ and $\ell_{\sigma}(\Gamma) = \ell_{\sigma}(\hat{0}_X,\hat{1}_Y)$ the 
	$h$-polynomial and local $h$-polynomial respectively associated to $\sigma$.

%	For any interval $[x,y]$ in $\Gamma$ with $x \in X$ and $y \in Y$
%	
%	With the notation above, define elements $h_{\sigma},\ell_{\sigma} \in \II(\Gamma)$ as follows: let
%	$(h_{\sigma})_X = (h_{\sigma})_Y = (\ell_{\sigma})_X = (\ell_{\sigma})_Y = 0$, and for any   $x \in X$ and $y \in Y$ such that $\sigma(x) \le y$. 
%	let
%	\[
%	(t - 1)h_{\sigma}(x,y) = \sum_{ \substack{x \le x' \in X \\ \sigma(x') = y} } g_\Gamma(x,x')  \kappa_\Gamma(x',y), 
%	\]
%	\[
%	\ell_{\sigma}(x,y) = \sum_{ \sigma(x) \le y' \le y } h_{\sigma}(x,y') g_\Gamma^{-1}(y',y). 
%	\]
%	The polynomials $h_{\sigma}(x,y)$ and $\ell_{\sigma}(x,y)$ are called the 
%	\emph{$h$-polynomial} and \emph{local $h$-polynomial} respectively associated to the interval $[x,y]$  in $\Gamma$. 

\end{definition}

In Definition~\ref{def:hellpolynomial}, we could have equivalently defined $h_\sigma$ and $\ell_\sigma$ by setting $h_{\sigma}(z,z') = 0$ unless $z \in X$ and $z' \in Y$, and, for  any   $x \in X$ and $y \in Y$ such that $\sigma(x) \le y$, setting
\[
(t - 1)h_{\sigma}(x,y) = \sum_{ \substack{x \le x' \in X \\ \sigma(x') = y} } g_X(x,x')  \kappa_\Gamma(x',y), 
\]
\[
\ell_{\sigma}(x,y) = \sum_{ \sigma(x) \le y' \le y } h_{\sigma}(x,y') g_Y^{-1}(y',y). 
\]
In particular, 
%In Definition~\ref{def:hellpolynomial} above, observe that 
$h_{\sigma}(x,y)$ is a well-defined polynomial by Lemma~\ref{lem:poleatone}. The fact that $h_{\sigma},\ell_{\sigma} \in \II(\Gamma)$ follows from Remark~\ref{rem:degreeh} below. 
%Observe that $\ell_{\sigma} = h_{\sigma} \cdot g_\Gamma^{-1}$ and $h_{\sigma} = \ell_{\sigma} \cdot g_\Gamma$. 
Also, observe that $h_{\sigma}(x,y)$ and $\ell_{\sigma}(x,y)$ only depend on the restriction of the relevant KLS invariants to the interval $[x,y]$ in $\Gamma$. % (and the restriction of the relevant invariants to $[x,y]$). 

\begin{remark}\label{rem:degreeh}
	With the setup of Definition~\ref{def:hellpolynomial},  the  degree of both $h_{\sigma}(x,y)$ and $\ell_{\sigma}(x,y)$ is bounded by $r_\Gamma(x,y) - 1$. Indeed,  the bound on the degree of  $h_{\sigma}(x,y)$ follows since 
 $g_\Gamma \in \II_{1/2}(\Gamma)$ and $\kappa_\Gamma \in \II(\Gamma)$. The bound on the degree of $\ell_{\sigma}(x,y)$ then follows, using the fact that $g_\Gamma^{-1} \in \II_{1/2}(\Gamma)$ by Lemma~\ref{lem:inverse}. 
 %See also Proposition~\ref{prop:symmetry} below. 
 In particular, we deduce that 
 $(t - 1) \cdot h_{\sigma}, (t - 1) \cdot \ell_{\sigma} \in \II(\Gamma)$. 
\end{remark}

\begin{example}\label{ex:sigmaxequalsy}
	For any $x \in X$, we have $h_{\sigma}(x,\sigma(x)) = \ell_{\sigma}(x,\sigma(x))$. 
	
\end{example}

\begin{example}\label{ex:simple2}
	Consider Example~\ref{ex:simple} with $\kappa_\Gamma =  \kappa_\Gamma^{\rev} = f_\Gamma = g_\Gamma = \delta_\Gamma$.  With the notation above, $g_\Gamma|_{X/Y} =  h_{\sigma} = \ell_{\sigma} = 0$. % \Delta \ell_{\sigma} = 0$. 
\end{example}

\begin{example}\label{ex:simpleB1v2}
	Continuing with Example~\ref{ex:simpleB1}, 	
	let  $\Gamma = B_1 = \{ \hat{0}, \hat{1} \}$,  let $r_{\Gamma} = \rho_{B_1}$ be the natural weak rank function,
	%for a rank function $\rho_{B_1}$ for $B_1$, 
	and let $q = \hat{1}_\Gamma$. In this case, $X = Y = B_0$ and $\sigma$ is the identity function. Recall that there exists $\lambda \in \Z$ such that $\kappa_{B_1}(B_1) = \lambda(t - 1)$. Then $h_{\sigma}(B_1) = \ell_{\sigma}(B_1) = \lambda$. 
\end{example}

\begin{example}\label{ex:t-1casedetail}
	Consider the setup of Example~\ref{ex:t-1case}, where  $r_\Gamma$ is the natural weak rank function and $\kappa_\Gamma$ is the Eulerian kernel. Consider $x \in X$ and $y \in Y$ such that $\sigma(x) \le y$. 
%	Assume that $r_\Gamma$ and $\kappa_\Gamma$ are given by Example~\ref{ex:t-1case}.
%	That is, 
%	  $r_\Gamma = \rho_\Gamma$ and
%	$\kappa_\Gamma(z,z') = (t - 1)^{\rho_\Gamma(z,z')}$ for all $z,z' \in \Gamma$.
	Then $h_{\sigma}(x,y) = h(B)$, where $B$ is the lower Eulerian poset 
	$\{ x' \in X : x \le x', \sigma(x') \le y \}$
%	$X_{\ge x} \cap \sigma^{-1}([\hat{0}_Y,y])$ 
%	Then $h_{\sigma}(x,y)$ is the usual $h$-polynomial of the lower Eulerian poset $X_{\ge x} \cap \sigma^{-1}([\hat{0}_Y,y])$ 
	(see, for example, \cite{KatzStapledon16}*{Definition~3.12, Proposition 3.29}). 
	In particular, if $\Gamma$ is Eulerian, then 
	the 
	$h$-polynomial  associated to $\sigma$ is 	$h_{\sigma}(\Gamma) = h(X)$.

%	$h_{\sigma}(\Gamma) = h(X)$,
%	 %is the usual $h$-polynomial of $X$, 
%	 and $\ell_{\sigma}(\Gamma)$ is the usual local $h$-polynomial associated to $\sigma$ \cite{KatzStapledon16}*{Definition~4.1}. 
\end{example}

In the case when   $r_\Gamma$ is the natural weak rank function and $\kappa_\Gamma$ is the Eulerian kernel,
%$r_\Gamma$ and $\kappa_\Gamma$ are given by Example~\ref{ex:t-1case}, 
the following proposition is 
 \cite{KatzStapledon16}*{Proposition~3.29}.
%  \cite{KatzStapledon16}*{Proposition~3.29, Theorem~5.5}.
 The proof will be given in Section~\ref{ss:proofsmain}. Recall from Remark~\ref{rem:degreeh} that $(t - 1) \cdot \ell_{\sigma} \in \II(\Gamma)$. 
 
\begin{proposition}\label{prop:symmetry}
	The local $h$-polynomials are symmetric in the sense that for all $x \in X$ and $y \in Y$ such that $\sigma(x) \le y$,
	\[
	\ell_{\sigma}(x,y;t) = t^{r_\Gamma(x,y) - 1}\ell_{\sigma}(x,y;t^{-1}).
	\]
	Equivalently,  $(t - 1) \cdot \ell_{\sigma} \in \II(\Gamma)$ is antisymmetric.

\end{proposition}

Recall from Section~\ref{ss:KLSbackground} that we may consider the element $\Delta \ell_{\sigma} \in \II_{1/2}(\Gamma)$. It follows from Proposition~\ref{prop:symmetry} and \eqref{eq:tildeell} that 
$\Delta \ell_{\sigma}$ is an alternative encoding of $\ell_{\sigma}$.

\begin{example}\label{ex:lowdegreehell}
%	Using Example~\ref{ex:lowdegreeterms}, 
%Definition~\ref{def:hellpolynomial}, and Proposition~\ref{prop:symmetry}, one may verify 

We claim that for any   $x \in X$ and $y \in Y$ such that $\sigma(x) \le y$, 
	\[
	h_{\sigma}(x,y)_0 = \kappa_\Gamma(x,y)_{r_\Gamma(x,y)},
	\]
	\[
	h_{\sigma}(x,y)_{r_\Gamma(x,y) - 1} = \ell_{\sigma}(x,y)_{r_\Gamma(x,y) - 1} = \ell_{\sigma}(x,y)_0 = (\Delta \ell_{\sigma}(x,y))_0  =  \begin{cases}
		\kappa_\Gamma(x,y)_{r_\Gamma(x,y)} &\textrm{if } \sigma(x) = y, \\
		0    &\textrm{otherwise.} 
	\end{cases}
	\]
	In the case when  $r_\Gamma$ is the natural weak rank function and $\kappa_\Gamma$ is the Eulerian kernel, 
	%$r_\Gamma$ and $\kappa_\Gamma$ are given by Example~\ref{ex:t-1case}, 
	%$r_\Gamma(x,y) = \rho_\Gamma(x,y)$ and
	 $\kappa_\Gamma(x,y)_{r_\Gamma(x,y)} = 1$ above, and  the linear coefficients of $h_{\sigma}(x,y)$ and $\ell_{\sigma}(x,y)$
	 are described  in %this case see 
	 \cite{KatzStapledon16}*{Example~3.13, Example~4.10}. 
	 
	To prove the first equation, we compute, using Example~\ref{ex:lowdegreeterms} and 
	Definition~\ref{def:hellpolynomial}, as well as the assumption that 
	$\kappa_\Gamma$ is multiplicative and rank alternating, and \eqref{eq:strongsubdivisionequality} and \eqref{eq:rhoCyl}, 
	\begin{align*}
	 h_{\sigma}(x,y)_0  &= - \sum_{ \substack{x \le x' \in X \\ \sigma(x') 
	 		= y} } g_\Gamma(x,x')_0 \kappa_\Gamma(x',y)_0 
 		\\&= - \sum_{ \substack{x \le x' \in X \\ \sigma(x') = y} } \kappa_\Gamma(x,x')_{r_\Gamma(z,z')} \kappa_\Gamma(x',y)_0
 		 		\\&= - \sum_{ \substack{x \le x' \in X \\ \sigma(x') = y} } (-1)^{\rho_{\Gamma}(x',y)} \kappa_\Gamma(x,x')_{r_\Gamma(x,x')} \kappa_\Gamma(x',y)_{r_\Gamma(x',y)} \\
 		 		&= \kappa_\Gamma(x,y)_{r_\Gamma(x,y)} \sum_{ \substack{x \le x' \in X \\ \sigma(x') = y} } (-1)^{\rho_{Y}(y) - \rho_X(x)} 
 		 		\\ &=  	\kappa_\Gamma(x,y)_{r_\Gamma(x,y)}.
	\end{align*}
	The second equation follows by comparing highest degree terms in Definition~\ref{def:hellpolynomial}, and then using Proposition~\ref{prop:symmetry}. 
	
	Later we will need the following proposition. 
	
	\begin{proposition}\label{prop:composition}\cite{KatzStapledon16}*{Corollary~4.7}
		Let 
		$\sigma: X \to Y$ and $\tau: Y \to Z$ be strong formal subdivisions between lower Eulerian posets $X,Y,Z$ with rank functions $\rho_X, \rho_Y, \rho_Z$ respectively. 
		Consider each of $\Cyl(\sigma)$, $\Cyl(\tau)$, and $\Cyl(\tau \circ \sigma)$ with the corresponding natural weak rank function and Eulerian kernel. 
%		
%			Fix  a compatible choice of a weak rank function $r_{\Cyl(\sigma,\tau)} \in I(\Cyl(\sigma,\tau))$, and 
%		an element $\kappa_{\Cyl(\sigma,\tau)}\in \II(\Cyl(\sigma,\tau)) \cap U(\Cyl(\sigma,\tau))$ that is multiplicative and rank alternating.
		Then for any $x \in X$ and $z \in Z$ such that $(\tau \circ \sigma)(x) \le z$, 
		\begin{equation}\label{eq:composition}
			\ell_{\tau \circ \sigma}(x,z) = \sum_y  \ell_{\sigma}(x,y) \ell_{\tau}(y,z), 
		\end{equation}
		%	\[
		%
		%	%\sum_{\substack{y \in Y \\ \sigma(x) \le y, \tau(y) \le z}}
		%	\]
		where the above sum varies over all $y \in Y$ such that $\sigma(x) \le y$ and $\tau(y) \le z$. 
	\end{proposition}
	
%	Comparing highest degree terms in Definition~\ref{def:hellpolynomial}, we have 
%%	The second example follows by comparing highest degree terms in 
%	 $$h_{\sigma}(x,y)_{r_\Gamma(x,y) - 1}  =  \sum_{ \substack{x \le x' \in X \\ \sigma(x') 
%			= y} } g_\Gamma(x,x')_{r_\Gamma(x,x')} \kappa_\Gamma(x',y)_{r_\Gamma(x',y)} =  \kappa_\Gamma(x,y)_{r_\Gamma(x,y)}.$$
\end{example}

We are now ready to state our main theorem.  The proof  will be given in Section~\ref{ss:proofsmain}. 
The theorem implies that the left Kazhdan-Lusztig-Stanley function $g_\Gamma$ associated to $\Gamma$ is determined by the left Kazhdan-Lusztig-Stanley functions $g_X$ and $g_Y$ associated to $X$ and $Y$ respectively, together with $\ell_{\sigma}$. Conversely, 
 Remark~\ref{rem:altdeltaell} below implies that $g_X$, $g_Y$, and $\ell_{\sigma}$, are determined by $g_\Gamma$.

\begin{theorem}\label{thm:maingell}
		Let  $\sigma: X \to Y$ be a strong formal subdivision
	between lower Eulerian posets $X$ and $Y$ with rank functions $\rho_X$ and $\rho_Y$ respectively, corresponding to a triple $(\Gamma,\rho_\Gamma,q)$  under Theorem~\ref{thm:introbijection}. 
	Fix a weak rank function $r_\Gamma \in I(\Gamma)$ and a multiplicative and rank alternating element $\kappa_\Gamma \in \II(\Gamma) \cap U(\Gamma)$. 
	Then
	\[
	     g_\Gamma|_{X/Y} = \Delta \ell_{\sigma} \cdot  g_\Gamma =  \Delta \ell_{\sigma} \cdot  g_\Gamma|_Y. 
	\]
That is, for any $x \in X$ and $y \in Y$ such that $\sigma(x) \le y$, 
\[
g_\Gamma(x,y) = \sum_{ \sigma(x) \le y' \le y} \Delta \ell_{\sigma}(x,y') g_Y(y',y). 
\] 
\end{theorem}

Recall from Section~\ref{ss:KLSbackground} that we have a ring involution $p \mapsto \widehat{p}$ of $\II(\Gamma)$,
where $\widehat{p}(z, z') = (-1)^{\rho_\Gamma(z,z')} p(z, z')$ for $z \le z'$ in $\Gamma$.  

\begin{remark}\label{rem:altdeltaell}
	By Lemma~\ref{lem:inverse} and Theorem~\ref{thm:maingell}, we have  $\Delta \ell_{\sigma} = g_\Gamma|_{X/Y} \cdot g_\Gamma^{-1} = g_\Gamma|_{X/Y} \cdot \widehat{f_\Gamma}$. 
	That is, for any $x \in X$ and $y \in Y$ such that $\sigma(x) \le y$, 
	\[
	\Delta \ell_{\sigma}(x,y) = \sum_{ \sigma(x) \le y' \le y} (-1)^{\rho_Y(y',y)} g_\Gamma(x,y') f_Y(y',y). 
	\] 
\end{remark}

\begin{example}\label{ex:sigmaxequalsyv3}
			Recall from Example~\ref{ex:sigmaxequalsy} that for any $x \in X$, 
	$h_{\sigma}(x,\sigma(x)) = \ell_{\sigma}(x,\sigma(x))$. 
	By Theorem~\ref{thm:maingell}, $g_\Gamma(x,\sigma(x)) = \Delta \ell_{\sigma}(x,\sigma(x))$.
	
	The latter statement can also be deduced by comparing 
	\eqref{eq:tildeell} and the equality 
		$(t - 1)h_{\sigma}(x,\sigma(x)) = g_\Gamma^{\rev}(x,\sigma(x)) - g_\Gamma(x,\sigma(x))$ (which can be deduced from Definition~\ref{def:hellpolynomial} using the equality $g_\Gamma^{\rev} = g_\Gamma \cdot \kappa_\Gamma$). 
	  This special case is well-known in the case when 
	   $r_\Gamma$ is the natural weak rank function and $\kappa_\Gamma$ is the Eulerian kernel
	 % $r_\Gamma$ and $\kappa_\Gamma$ are given by Example~\ref{ex:t-1case} 
	 (see, for example, \cite{KatzStapledon16}*{Example~3.14}).
	
%	When $y = \sigma(x)$ in Theorem~\ref{thm:maingell}, the statement of the theorem says that $g_\Gamma(x,y) = \Delta \ell_{\sigma}(x,y)$. 
%	 This reproduces Example~\ref{ex:sigmaxequalsyv2}.
%	 \alan{maybe merge Example~\ref{ex:sigmaxequalsyv2} here and forgo the computations; just one line explanation}
%		
%	\begin{example}\label{ex:sigmaxequalsyv2}
%		Recall from Example~\ref{ex:sigmaxequalsy} that for any $x \in X$, 
%		$h_{\sigma}(x,\sigma(x)) = \ell_{\sigma}(x,\sigma(x))$. By \eqref{eq:tildeell}, $\Delta \ell_{\sigma} (x,y)$ is the unique polynomial of degree strictly less than $r_\Gamma(x,y)/2$ such that 
%		\begin{equation*}%\label{eq:tildeell}
%			t^{r_\Gamma(x,y)}\Delta \ell_{\sigma} (x,y;t^{-1}) - \Delta \ell_{\sigma} (x,y;t) = (t - 1) h_{\sigma}(x,y). 
%		\end{equation*}
%		On the other hand, using Definition~\ref{def:hellpolynomial} and the fact that $g_\Gamma^{\rev} = g_\Gamma \cdot \kappa_\Gamma$, we have
%		\[
%		(t - 1)h_{\sigma}(x,\sigma(x)) = \sum_{ \substack{x \le x' \in X \\ \sigma(x') = \sigma(x)} } g_\Gamma(x,x')  \kappa_\Gamma(x',\sigma(x)) = g_\Gamma^{\rev}(x,\sigma(x)) - g_\Gamma(x,\sigma(x)).  
%		\]
%		We conclude that  $\Delta \ell_{\sigma}(x,\sigma(x)) = g_\Gamma(x,\sigma(x))$. This is well-known in the case when $r_\Gamma$ and $\kappa_\Gamma$ are given by Example~\ref{ex:t-1case} (see, for example, \cite{KatzStapledon16}*{Example~3.14}). We will reproduce this result in Example~\ref{ex:sigmaxequalsyv3} below.
%		
%		
%	\end{example}	
		
\end{example}

We have the following analogue of Theorem~\ref{thm:maingell} for the right Kazhdan-Lusztig-Stanley function. The proof will be given in Section~\ref{ss:proofsmain}.

\begin{corollary}\label{cor:rightKLSfunction}
			Let  $\sigma: X \to Y$ be a strong formal subdivision
	between lower Eulerian posets $X$ and $Y$ with rank functions $\rho_X$ and $\rho_Y$ respectively, corresponding to a triple $(\Gamma,\rho_\Gamma,q)$  under Theorem~\ref{thm:introbijection}. 
	Fix a weak rank function $r_\Gamma \in I(\Gamma)$ and a multiplicative and rank alternating element $\kappa_\Gamma \in \II(\Gamma) \cap U(\Gamma)$. 
%	Let $\sigma: X \to Y$ be a strong formal subdivision between lower Eulerian posets with rank functions $\rho_X$ and $\rho_Y$ respectively,
%	corresponding under Theorem~\ref{thm:mainsimplified} to a triple $(\Gamma, \rho_\Gamma, q) \in \JoinIdealLW^\circ$ with $\Gamma = \Cyl(\sigma)$, $\rho_\Gamma$ determined by \eqref{eq:rhoCyl}, and $q = \hat{0}_Y$. 
	Then
	\[
	f_\Gamma|_{X/Y} = - f_\Gamma \cdot \Delta \widehat{\ell_{\sigma}} =  - f_\Gamma|_X \cdot \Delta \widehat{\ell_{\sigma}}. 
	\]
	That is, for any $x \in X$ and $y \in Y$ such that $\sigma(x) \le y$, 
	\[
	f_\Gamma(x,y) 
	=  \sum_{ \substack{x \le x' \in X \\ \sigma(x') \le y} } (-1)^{\rho_Y(y) - \rho_X(x')} f_X(x,x') \Delta \ell_{\sigma}(x',y). 
	\] 
\end{corollary}

%To deduce the second statement above from the first we used \eqref{eq:rhoCyl}. 
We also have an analogue of Remark~\ref{rem:altdeltaell}. 

\begin{remark}
	It follows from Lemma~\ref{lem:inverse} and Corollary~\ref{cor:rightKLSfunction} that $\Delta \ell_{\sigma} 
	 = - g_\Gamma \cdot \widehat{f_\Gamma}|_{X/Y}$. 	That is, for any $x \in X$ and $y \in Y$ such that $\sigma(x) \le y$, 
	 \[
	 \Delta \ell_{\sigma}(x,y) = \sum_{ \substack{x \le x' \in X \\ \sigma(x') \le y} } (-1)^{\rho_Y(y) - \rho_X(x')} g_X(x,x') f_\Gamma(x',y). 
	 \] 
\end{remark}

Recall from Definition~\ref{def:Zfunction} that we may consider the $Z$-function 
$Z_\Gamma = g_\Gamma \cdot \kappa_\Gamma \cdot f_\Gamma \in \II(\Gamma)$. 
The following corollary implies that $Z_\Gamma$  is determined by the $Z$-functions $Z_X$ and $Z_Y$ associated to $X$ and $Y$ respectively, together with $\ell_{\sigma}$. The proof will be given in Section~\ref{ss:proofsmain}. 

		\begin{corollary}\label{cor:Zmappingformula}
						Let  $\sigma: X \to Y$ be a strong formal subdivision
			between lower Eulerian posets $X$ and $Y$ with rank functions $\rho_X$ and $\rho_Y$ respectively, corresponding to a triple $(\Gamma,\rho_\Gamma,q)$  under Theorem~\ref{thm:introbijection}. 
			Fix a weak rank function $r_\Gamma \in I(\Gamma)$ and a multiplicative and rank alternating element $\kappa_\Gamma \in \II(\Gamma) \cap U(\Gamma)$. 
			Then
	\begin{align*}
		Z_\Gamma|_{X/Y} = - Z_\Gamma|_X \cdot \Delta \widehat{\ell_{\sigma}} + (\Delta \ell_{\sigma})^{\rev} \cdot Z_\Gamma|_Y. 
	\end{align*}
		That is, for any $x \in X$ and $y \in Y$ such that $\sigma(x) \le y$, 
	\[
	Z_\Gamma(x,y) = \sum_{ \substack{x \le x' \in X \\ \sigma(x') \le y} } (-1)^{\rho_Y(y) - \rho_X(x')} Z_X(x,x') \Delta \ell_{\sigma}(x',y) + \sum_{ \sigma(x) \le y' \le y} (\Delta \ell_{\sigma})^{\rev}(x,y') Z_Y(y',y). 
	\] 
\end{corollary}

\subsection{Examples of main results}\label{ss:examplesmain}

In this section, we apply our main results to some examples that appeared in 
\cite{StapledonLWPosets}*{Section~4}. We continue with the notation of Section~\ref{ss:statements}.

\begin{example}\label{ex:polytope}
	We recall the following example from \cite{StapledonLWPosets}*{Section~4.1}. 
	Let $Q$ be a full-dimensional polytope in a real vector space $V$. 
	Let $F$ be a nonempty face of $Q$.
	Consider the triple $(\Gamma,\rho_\Gamma,q)$,
	where  
	$\Gamma = \face(Q)$ is the face lattice of $Q$, $\rho_\Gamma$ is the natural rank function, and $q = F \in \Gamma$. 
	Let $\sigma: X \to Y$ be the corresponding strong formal subdivision under Theorem~\ref{thm:introbijection}, where $X = \Gamma \smallsetminus [F,Q]$ and $Y = [F,Q]$. Here $Y$ is the  face lattice of the quotient polytope $Q/F$ \cite{BMIntersectionHomologyKalai}.  Let
	 $r_\Gamma$ be the natural weak rank function and let $\kappa_\Gamma$ be the Eulerian kernel. 
%		$r_\Gamma = \rho_\Gamma$ and $\kappa_\Gamma(z,z') = (t - 1)^{\rho_\Gamma(z,z')}$ as in Example~\ref{ex:t-1case}. 
%%	 Recall that in this case we write $\rho = r_\Gamma = \rho_\Gamma$, $\kappa = \kappa_\Gamma$, $f = f_\Gamma$, $g = g_\Gamma$, and $Z = Z_\Gamma$. 
	The polynomial $g(\Gamma) = g_\Gamma(\hat{0}_\Gamma,\hat{1}_\Gamma)$ is the \emph{$g$-polynomial} of $Q$ and is also denoted $g(Q)$. It was first considered by Stanley in \cite{StanleyGeneralized87}. 
	%By Example~\ref{ex:t-1casedetail}, $h_\sigma(\Gamma) = h(X)$. 
	
	We briefly recall the following geometric description of $\sigma$, and refer the reader to \cite{StapledonLWPosets}*{Section~4.1} for details.  After possible translation, we may assume that the origin lies in the relative interior of $F$. For each face $G$ of $Q$, let 
	$C(G)$ denotes the smallest cone in $V$ containing $G$. 
	Consider the fan $\Sigma = \{ C(G) : G \in X \}$ with support $C(Q)$. 
	Let $L$ be the linear span of $F$.  Equivalently, $L$ is the  largest linear subspace of $V$ contained in $C(Q)$.   Let $\phi: V \to V/L$ denote the projection map, and let $C = \phi(C(Q))$.
	Then $\phi$ induces a projective, surjective morphism of fans from $\Sigma$ to $C$, and $\phi$ induces the strong formal subdivision $\sigma$. That is, $X$ and $Y$ are the poset of faces of $\Sigma$ and $C$ respectively, and $\sigma(G)$ corresponds to the smallest face of $C$ containing $\phi(C(G))$ for all $G \in X$. Conversely, every strong formal subdivision induced by a  projective, surjective morphism between a fan and a pointed cone appears in this way. 
	% (up to a shift in the rank function $\rho_\Gamma$). 

	Recall from the introduction that 
	Braden and MacPherson introduced the \emph{relative $g$-polynomial} $g(Q,F)$ associated to the pair $(Q,F)$  \cite{BMIntersectionHomologyKalai}*{Proposition~2}. They define  $g(Q,Q) = g(Q)$. and then recursively define $g(Q,F)$ by the relation
	\begin{equation}\label{eq:BMdefine}
		\sum_{F \subset E \subset Q}  g(E,F) g(Q/E) = g(Q). 
	\end{equation}
	For example, if $F$ is a facet of $Q$, then $g(Q,F) = g(Q) - g(F)$.
	Recall from Example~\ref{ex:t-1casedetail}, that $\ell_{\sigma}(\Gamma)$ is the usual local $h$-polynomial associated to $\sigma$. 
	We claim that Theorem~\ref{thm:maingell} implies that 
		$$g(Q,F) = \Delta \ell_{\sigma}(\Gamma).$$
	To prove the claim,  define an element $\bar{g} \in \II(\Gamma)$ as follows. For any $z \le z' \in \Gamma$, let
	$$\bar{g}(z,z') = \begin{cases}
		g(E/G,G \vee F/G) &\textrm{ if } z = G \in X, z' = E \in Y, \\
		0 &\textrm{ otherwise.}
	\end{cases}
	$$
	Consider any $x \in X$ and $y \in Y$ such that $\sigma(x) \le y$. Suppose that $x$ corresponds to a face $G$ of $Q$, and $y$ corresponds to a face $E$ of $Q$. Then $\sigma(x)$ corresponds to the face $G \vee F$ of $Q$.
	The defining equation \eqref{eq:BMdefine} applied to the pair $(E/G,G \vee F/G)$ gives
	\[
		(\bar{g} \cdot g_\Gamma)(x,y) = 	\sum_{y \le y' \le \hat{1}_Y}  \bar{g}(x,y') g_\Gamma(y',y) = g_\Gamma(x,y).
	\]
	We deduce that $\bar{g} \cdot g_\Gamma = g_\Gamma|_{X/Y}$. 
	Then Theorem~\ref{thm:maingell} implies that $\bar{g} = g_\Gamma|_{X/Y} \cdot g_\Gamma^{-1} = \Delta \ell_{\sigma}$. We conclude that
	$g(Q,F) = \bar{g}(\hat{0}_X,\hat{1}_Y) = \Delta \ell_{\sigma}(\hat{0}_X,\hat{1}_Y) = \Delta \ell_{\sigma}(\Gamma)$, as desired.	
	
\end{example}

\begin{example}
	We have the following concrete example of Example~\ref{ex:polytope} that appeared in \cite{StapledonLWPosets}*{Example~4.1}. 
	Let $Q$ be a polygon with $s + 3$ vertices, and let $F$ be a vertex of $Q$. 	Consider the triple $(\Gamma,\rho_\Gamma,q)$,
	where  
	$\Gamma = \face(Q)$ is the face lattice of $Q$, $\rho_\Gamma$ is the natural rank function, and $q = F \in \Gamma$.  
	Let $P = [0,1]$. Then $\face(P) = B_2$. 
	Let $\cS$ be the regular polyhedral subdivision of $P$ with $s$ interior vertices.  	 Let $\sigma: \face(\cS) \to \face(P)$ be the corresponding strong formal subdivision, where $\face(\cS)$ and $\face(P)$ are equipped with the natural rank functions (see Example~\ref{ex:intropolytope}). 	Then $(\Gamma,\rho_\Gamma,q)$ corresponds to $\sigma$ under Theorem~\ref{thm:introbijection}. Let
	$r_\Gamma$ be the natural weak rank function and let $\kappa_\Gamma$ be the Eulerian kernel.  Using the results of the previous section, one may compute that
 $h_\sigma(\Gamma) = g(\Gamma) = f(\Gamma) = 1 + st$,  $\ell_\sigma(\Gamma) = \Delta \ell_\sigma(\Gamma) = g(Q,F) = st$, and 
$Z(\Gamma) = 1 + (2s + 3)t + (2s + 3)t^2 + t^3$.  

% (c.f. Example~\ref{ex:lowdegreehell}). 

\end{example}

\begin{example}\label{ex:B0}
	We recall the following example from \cite{StapledonLWPosets}*{Example~4.9}. Every strong formal subdivision to $B_0$ has the following form. 
		Let $B$ be an Eulerian poset of positive rank with rank function $\rho_B$. 		Consider the triple $(\Gamma,\rho_\Gamma,q)$, where $\Gamma = B$, $\rho_\Gamma = \rho_B$, and  $q = \hat{1}_\Gamma$. Let $\sigma: X \to Y$ be the corresponding strong formal subdivision under Theorem~\ref{thm:introbijection}.  Then $X = \partial B$, $Y = B_0$ is identified with $\{ q \}$, and the rank functions for $X$ and $Y$ are determined by \eqref{eq:rhoCyl}.  
		For example, consider Example~\ref{ex:polytope} with $F = Q$. 
		
	 By Example~\ref{ex:sigmaxequalsy} and Example~\ref{ex:sigmaxequalsyv3},
	 $h_{\sigma}(B) = \ell_{\sigma}(B)$, and $\Delta \ell_{\sigma}(B) = g_B(B)$. In the case 
	 when 
	 $r_\Gamma$ is the natural weak rank function and $\kappa_\Gamma$ is the Eulerian kernel, $h_{\sigma}(B) = \ell_{\sigma}(B) = h(\partial B)$ by Example~\ref{ex:t-1casedetail}, and $\Delta \ell_{\sigma}(B) = g(B)$.

%	 $h_{\sigma}(B)$ is the usual $h$-polynomial of $\partial B$, and $g_B(B)$ is the usual $g$-polynomial of $B$. 
%	 If we specialize further to the case of Example~\ref{ex:t-1caseBn} when $B = B_n$ for some $n > 0$, then 
%	 $h_{\sigma}(B_n) = \ell_{\sigma}(B_n) = 1 + t + \cdots + t^{n - 1}$ and $\Delta \ell_{\sigma}(B_n) = g_{B_n}(B_n) = 1$. 

\end{example}

\begin{example}\label{ex:B1}
	We recall the following example from \cite{StapledonLWPosets}*{Example~4.11}. 
	Recall from \cite[p.485]{StanleyFlagfVectors} that a  nonempty poset $B$ is \emph{near-Eulerian} if  there exists an Eulerian poset of positive rank $\tilde{\Sigma} B$, called the \emph{semisuspension} of $B$, and a maximal element $\hat{z} \in \partial (\tilde{\Sigma} B)$  such that $B = \partial (\tilde{\Sigma} B) \smallsetminus \{ \hat{z} \}$. The \emph{boundary} $\partial B$ is the lower order ideal of $B$ generated by all elements $z \in B$ %(necessarily of corank $1$)
	such that  precisely one element of $B$ is strictly greater than $z$.

	Every strong formal subdivision to $B_1$ has the following form. 
	Let $B$ be a near-Eulerian poset with rank function $\rho_B$. 
	Then $\tilde{\Sigma} B$ and $\partial B$ have induced rank functions 
	  $\rho_{\tilde{\Sigma} B}$ and $\rho_{\partial B}$ respectively.
%	  Recall from Section~\ref{ss:posets} the corresponding semisuspension $\tilde{\Sigma} B = B \cup   \{ \hat{z}, \hat{1}_{\Gamma}\}$ and boundary $\partial B = \{ z \in \tilde{\Sigma} B : z < \hat{z} \}$ with induced rank functions  $\rho_{\tilde{\Sigma} B}$ and $\rho_{\partial B}$ respectively. 
	Here 
	$\rho_{\tilde{\Sigma} B}(\hat{1}_{\Gamma}) = \rho_{\tilde{\Sigma} B}(\hat{z}) + 1 = \rho_{B}(\hat{0}_B) + \rank(B) + 1$.  
    Consider the triple $(\Gamma,\rho_\Gamma,q)$, where
		$\Gamma = \tilde{\Sigma} B$, $\rho_\Gamma = \rho_{\tilde{\Sigma} B}$, and $q = \hat{z}$. Let $\sigma: X \to Y$ be the corresponding strong formal subdivision under Theorem~\ref{thm:introbijection}, where $X = B$ and $Y = B_1$ is identified with $\{ \hat{z}, \hat{1}_{\Gamma}\}$. That is, 
		\[
		\sigma: B \to B_1,
		\]
		\[
		\sigma(z) = \begin{cases}
			\hat{0} &\textrm{if } z \in \partial B, \\
			\hat{1} &\textrm{otherwise. }
		\end{cases}
		\]
		 For example, consider Example~\ref{ex:polytope} with $F$ a facet of $Q$.
		  
		 By Remark~\ref{rem:altdeltaell},
		 $\Delta \ell_{\sigma}(\tilde{\Sigma} B) = g_\Gamma(\tilde{\Sigma} B) - g_\Gamma(\overline{\partial B})f_\Gamma(\hat{z},\hat{1}_\Gamma)$, where $\overline{\partial B} = [\hat{0}_{\tilde{\Sigma} B},\hat{z}] \subset \tilde{\Sigma} B$ is obtained from $\partial B$ by adjoining a maximal element.  In the case
		 	 when 
		 $r_\Gamma$ is the natural weak rank function and $\kappa_\Gamma$ is the Eulerian kernel,
%		  when $r_\Gamma$ and $\kappa_\Gamma$ are given by Example~\ref{ex:t-1case},   
		  $h_{\sigma}(\tilde{\Sigma} B) = h(B)$ by Example~\ref{ex:t-1casedetail}, and 	  $\Delta \ell_{\sigma}(\tilde{\Sigma} B) = g(\tilde{\Sigma} B) - g(\overline{\partial B})$.
%		  where $g_\Gamma(\tilde{\Sigma} B)$ and $g_\Gamma(\overline{\partial B})$ are the usual 
%		 $g$-polynomials of $\tilde{\Sigma} B$ and $\overline{\partial B}$ respectively (see Example~\ref{ex:t-1casedetail}). 
		 
\end{example}

%Let  $P \subset V$ and $P' \subset V'$ be  a polytopes in  a real vector spaces $V$ and  $V'$ respectively. 
%Let $0_V$ and $0_{V'}$ denote the origin in $V'$ and $V$ respectively.
%The \emph{free join} $P \star P'$ is the convex hull of 
%$P \times \{ 0_{V'} \} \times \{ 0 \}$ and 
%$\{ 0_V \} \times P'  \times \{ 1 \}$ in $V \oplus V' \oplus \R$.
%%The \emph{pyramid} $\Pyr(P)$ is the free join of $P$ and $P' = \{ 0 \} \subset \R$.  
%Then
%$\face(P \star P') = \face(P) \times \face(P')$, i.e., faces, including the empty face, of  $P \star P'$ have the form $F \star F'$, where $F$ and $F'$ are possibly empty faces of $P$ and $P'$ respectively. 

% and $\face(\Pyr(P)) = \Pyr(\face(P))$. The \emph{apex} of $\Pyr(P)$ is the vertex 
%$(0_V,0,1)$. 

\begin{example}\label{ex:productsv2}
let $B$ be a lower Eulerian poset with rank function $\rho_B$,  weak rank function $r_B$, and 
multiplicative and rank alternating element $\kappa_B \in \II(B) \cap U(B)$. 
%multiplicative and rank alternating $B$-kernel $\kappa_B$.
% and strictly order-preserving function $r_B$, right and left Kazhdan-Lusztig-Stanley functions $f_B$ and $g_B$ respectively, and $Z$-function $Z_B$. 
Consider a triple $(\Gamma,\rho_\Gamma,q)$ corresponding to a strong formal subdivision $\sigma: X \to Y$ under Theorem~\ref{thm:introbijection}. 
Consider a weak rank function $r_\Gamma$ and a
multiplicative and rank alternating element $\kappa_\Gamma \in \II(\Gamma) \cap U(\Gamma)$. 
%Consider a multiplicative and rank alternating $\Gamma$-kernel $\kappa_{\Gamma}$ and weak rank function $r_\Gamma \in I(\Gamma)$. 
%, right and left Kazhdan-Lusztig-Stanley functions $f_\Gamma$ and $g_\Gamma$ respectively, and $Z$-function $Z_\Gamma$. 
Recall from Example~\ref{ex:productsprepre} that $B \times \Gamma$ is lower Eulerian with rank function $\rho_{B \times \Gamma}$. 
Recall from Example~\ref{ex:products} that if we fix the corresponding
weak rank function $r_{B \times \Gamma}$, then  
$\kappa_{B} \times \kappa_{\Gamma}$ is a $(B \times \Gamma)$-kernel, 
$f_{B \times \Gamma} = f_B \times f_{\Gamma}$, $g_{B \times \Gamma} = g_B \times g_{\Gamma}$, and $Z_{B \times \Gamma} = Z_{B} \times Z_{\Gamma}$.
Moreover, in this case, $\kappa_{B} \times \kappa_{\Gamma}$ is multiplicative and rank alternating.

%
%Recall that the corresponding
%right and left Kazhdan-Lusztig-Stanley functions are $f_{B} \times f_\Gamma$ and $g_B \times g_{\Gamma}$ respectively, and the corresponding $Z$-function is  $Z_B \times Z_\Gamma$.

With the notation of Example~\ref{ex:productsprepre} and
by \cite{StapledonLWPosets}*{Example~4.14}, $\id_B \times \sigma: B \times X \to B \times Y$ is a strong formal subdivision corresponding to the triple $(B \times \Gamma, \rho_{B \times \Gamma},(\hat{0}_B,q))$ under Theorem~\ref{thm:introbijection}. 
We claim that
\begin{equation}\label{eq:productidentity}
	h_{\id_B \times \sigma}  =  g_B \times h_{\sigma}, \: \: \: 
	\ell_{\id_B \times \sigma}  = 
	\delta_B \times \ell_{\sigma}, \: \: \: 
	\Delta \ell_{\id_B \times \sigma} =  
	\delta_B \times \Delta \ell_{\sigma}. 
\end{equation}
%\[
%\ell_{\id_B \times \sigma}  = 
%\delta_B \cdot \ell_{\sigma},
%\]
%\[
%\Delta \ell_{\id_B \times \sigma} =  
%\delta_B \cdot \Delta \ell_{\sigma}. 
%\]
%
%We claim that  for any 
%$z \le z' \in B$, and $x \in X$, $y \in Y$ with $\sigma(x) \le y$, 
%\[
%h_{\id_B \times \sigma}( (z,x),(z',y))  =  g_B(z,z') h_{\sigma}(x,y),
%\]
%\[
%\ell_{\id_B \times \sigma}( (z,x),(z',y))  = 
%	\delta_B(z,z')\ell_{\sigma}(x,y),
%\]
%\[
%\Delta \ell_{\id_B \times \sigma}( (z,x),(z',y))  =  
%	\delta_B(z,z')\Delta \ell_{\sigma}(x,y). 
%\]
%%\[
%%\ell_{\id_B \times \sigma}( (z,x),(z',y))  =  \begin{cases}
%%	\ell_{\sigma}(x,y) &\textrm{ if } z  = z', \\
%%	0 &\textrm{ otherwise. }
%%\end{cases}
%%\]
%%\[
%%\Delta \ell_{\id_B \times \sigma}( (z,x),(z',y))  =  \begin{cases}
%%	\Delta \ell_{\sigma}(x,y) &\textrm{ if } z  = z', \\
%%	0 &\textrm{ otherwise. }
%%\end{cases}
%%\]
These equations can be computed directly from  Definition~\ref{def:hellpolynomial}.
Alternatively, by Theorem~\ref{thm:maingell},
\[
\Delta \ell_{\id_B \times \sigma} = (g_B \times g_\Gamma)|_{(B \times X)/(B \times Y)} \cdot (g_B \times g_\Gamma)^{-1} = 
(g_B \times (g_\Gamma|_{X/Y})) \cdot (g_B^{-1} \times g_\Gamma^{-1}) = \delta_B \times \Delta \ell_{\sigma}.
%(g_B \cdot g_B^{-1}) \times \Delta \ell_{\sigma}. 
\]
This establishes the third equation, and the second equation follows after expanding definitions. The first equation also follows using $h_{\id_B \times \sigma} = \ell_{\id_B \times \sigma} \cdot (g_B \times g_\Gamma)$. 

For example, taking  $X = Y = B_0$ with the natural rank function and letting $\sigma = \id_{B_0}$,  we deduce that $\ell_{\id_B} = \Delta \ell_{\id_B} = \delta_B$.
%As another example, consider Example~\ref{ex:polytope} with $\Gamma = \face(P')$ for some polytope $P'$, $\rho_{\Gamma'}$ the natural rank function, and $q = F'$ for some nonempty face $F'$ of $P'$. 
%%Let $\sigma: X \to Y$ be the corresponding strong formal subdivision. Let $B = \face(P)$ for some polytope $P$. 
%Then $B \times \Gamma = \face(P \star P')$, $\rho_{B \times \Gamma}$ is the natural rank function, and $(\hat{0}_B,q)$ corresponds to the nonempty face $\emptyset \star F'$  of $P \star P'$. 

%As another example, suppose that we fix strictly order-preserving functions and kernels for $B$, $\Gamma$, and $B \times \Gamma$ as in Example~\ref{ex:t-1case}. Using Example~\ref{ex:t-1casedetail}, \eqref{eq:productidentity} implies that 
%$h(B \times X) = g(B)h(X)$ and $\ell_{\id_B \times \sigma}(B \times \Gamma)  = 0$.  
%\[
%h(B \times X) = h_{\id_B \times \sigma}(B \times \Gamma)  =  g(B) h_{\sigma}(\Gamma) = g(B)h(X), \: \: \: 
%\ell_{\id_B \times \sigma}(B \times \Gamma)  = 0 
%\delta_B \times \ell_{\sigma}, \: \: \: 
%\Delta \ell_{\id_B \times \sigma} =  
%\delta_B \times \Delta \ell_{\sigma}. 
%\]

%Also, $B \times Y = \face(P \star P'/\emptyset \star F')$. 
%Then \eqref{eq:productidentity} implies that 
%$h_{\id_B \times \sigma}(\face(P \star P'))  =  g(\face(P))h_{\sigma}(\face(P'))$ and 
%$\ell_{\id_B \times \sigma}(\face(P \star P'))  = 0$. 
\end{example}

\begin{example}\label{ex:productsformulas}
	Consider strong formal subdivisions $\sigma: X \to Y$  and $\sigma': X' \to Y'$.   With the notation of Example~\ref{ex:productsprepre} and
	by \cite{StapledonLWPosets}*{Example~4.15},  $\sigma \times \sigma': X \times X' \to Y \times Y'$ is a strong formal subdivision.
	Consider the elements 
	$(\Cyl(\sigma),\rho_{\Cyl(\sigma)},q)$,  $(\Cyl(\sigma'),\rho_{\Cyl(\sigma')},q')$, and $(\Cyl(\sigma \times \sigma'),\rho_{\Cyl(\sigma \times \sigma')},(q,q'))$  corresponding to $\sigma$, $\sigma'$ and $\sigma \times \sigma'$ respectively under Theorem~\ref{thm:introbijection}. 
%	Then  corresponds to $\sigma \times \sigma'$. 
	We warn that  $\rho_{\Cyl(\sigma \times \sigma')}$ is not the restriction of $\rho_{\Cyl(\sigma) \times \Cyl(\sigma')}$ to $\Cyl(\sigma \times \sigma') \subsetneq \Cyl(\sigma) \times \Cyl(\sigma')$. Indeed, while $\rho_{\Cyl(\sigma \times \sigma')}$ and $\rho_{\Cyl(\sigma) \times \Cyl(\sigma')}$ restrict to $\rho_{X \times X'}$ on $X \times X'$, they restrict to $\rho_{Y \times Y'}[1]$ and $\rho_{Y \times Y'}[2]$ respectively on $Y \times Y'$. 
	
%    We have $\sigma \times \sigma' = (\sigma \times \id_{Y'}) \circ (\id_X \times \sigma')$. 	With the notation of Proposition~\ref{prop:composition}, 
%	fix  a compatible choice of a strictly order-preserving function $r_{\Cyl(\id_X \times \sigma',\sigma \times \id_{Y'})} : \Cyl(\id_X \times \sigma',\sigma \times \id_{Y'}) \to \Z$, and 
%	an element $\kappa_{\Cyl(\id_X \times \sigma',\sigma \times \id_{Y'})}\in \II(\Cyl(\id_X \times \sigma',\sigma \times \id_{Y'})) \cap U(\Cyl(\id_X \times \sigma',\sigma \times \id_{Y'}))$ that is multiplicative and rank alternating. 
%	
%	$X \times X'  ---> X \times Y' ----> X' \times Y'$
%	
%	For example, 
%
%	WORKING ABOVE

	Suppose that we fix the natural weak rank function and Eulerian kernel
	for $\Cyl(\sigma)$, $\Cyl(\sigma')$, and $\Cyl(\sigma \times \sigma')$. 
%	Suppose that we fix weak rank functions and multiplicative and rank alternating kernels for $\Cyl(\sigma)$, $\Cyl(\sigma')$, and $\Cyl(\sigma \times \sigma')$ as in Example~\ref{ex:t-1case}. 
We will prove the following formulas. 
	For any $(x,x') \in X \times X'$ and $(y,y') \in Y \times Y'$ such that $(\sigma(x),\sigma(x')) \le (y,y')  \in Y \times Y'$, 
	\begin{equation}\label{eq:productell}
				\ell_{\sigma \times \sigma'}((x,x'),(y,y'))  = \ell_{\sigma}(x,y) \ell_{\sigma'}(x',y'),
	\end{equation}
	\begin{equation}\label{eq:productgCyl}
			g_{\Cyl(\sigma \times \sigma')}((x,x'),(y,y')) = 
					\sum_{\tilde{y}} \sum_{\tilde{y}'} 
			\Delta \ell_{\sigma \times \sigma'}( (x,x'), (\tilde{y},\tilde{y}')) g_Y(\tilde{y},y) g_{Y'}(\tilde{y}',y'),
		%	\sum_{\sigma(x) \le \tilde{y} \le y} \sum_{\sigma'(x') \le \tilde{y}' \le y'} 
%		\sum_{\tilde{y}} \sum_{\tilde{y}'} 
%		\Delta \left( \ell_{\sigma}(x,\tilde{y}) \ell_{\sigma'}(x',\tilde{y}') \right) g_Y(\tilde{y},y) g_{Y'}(\tilde{y}',y'),
	\end{equation}
		where the above sum varies over all $\tilde{y} \in Y$ and $\tilde{y}' \in Y'$ such that  $\sigma(x) \le \tilde{y} \le y$ and $\sigma'(x') \le \tilde{y}' \le y'$. For example, if $\Cyl(\sigma)$ and $\Cyl(\sigma')$ are Eulerian, or, equivalently, if $Y$ and $Y'$ are Eulerian, then setting $(x,x') = (\hat{0}_X,\hat{0}_{X'})$ and $(y,y') = (\hat{1}_Y,\hat{1}_{Y'})$ in \eqref{eq:productell} gives
		$\ell_{\sigma \times \sigma'}(\Cyl(\sigma \times \sigma')) = \ell_{\sigma}(\Cyl(\sigma))\ell_{\sigma'}(\Cyl(\sigma'))$, i.e., the local $h$-polynomial is multiplicative. 
%		\[
%		g(\Cyl(\sigma \times \sigma')) = 				\sum_{\tilde{y}} \sum_{\tilde{y}'} 
%		\Delta \ell_{\sigma \times \sigma'}( (x,x'), (\tilde{y},\tilde{y}')) g([\tilde{y},\hat{1}_Y]) g([\tilde{y}',\hat{1}_{Y'}]),.
%		\]

		Observe that \eqref{eq:productell} and \eqref{eq:productgCyl} allow us to compute $g_{\Cyl(\sigma \times \sigma')}$ in terms of $g_{\Cyl(\sigma)}$ and $g_{\Cyl(\sigma')}$. Indeed, 
		 $g_{\Cyl(\sigma \times \sigma')}|_{X \times X'} = g_{\Cyl(\sigma)} \times g_{\Cyl(\sigma')}|_{X \times X'}$ restricts to $g_X \times g_{X'}$, and  $g_{\Cyl(\sigma \times \sigma')}|_{Y \times Y'} = g_{\Cyl(\sigma)} \times g_{\Cyl(\sigma')}|_{Y \times Y'}$ restricts to $g_Y \times g_{Y'}$. 		By Theorem~\ref{thm:maingell}, $\Delta \ell_\sigma$ and $\Delta \ell_{\sigma'}$  and hence $\ell_\sigma$ and $\ell_{\sigma'}$ are determined by $g_{\Cyl(\sigma)}$ and $g_{\Cyl(\sigma')}$ respectively. Finally, \eqref{eq:productell} and \eqref{eq:productgCyl} determine 
		 $g_{\Cyl(\sigma \times \sigma')}|_{(X \times X')/(Y \times Y')}$.

		We now prove the formulas. First observe that \eqref{eq:productgCyl} follows from Theorem~\ref{thm:maingell} and \eqref{eq:productell}, so we need to establish \eqref{eq:productell}. 
		 We have $\sigma \times \sigma' = (\sigma \times \id_{Y'}) \circ (\id_X \times \sigma')$. 	
 Fix 
		  the natural weak rank function and Eulerian kernel
		 for
		  $\Cyl(\id_X \times \sigma')$ and $\Cyl(\sigma \times \id_{Y'})$. 
		  		 Applying Proposition~\ref{prop:composition} with $\sigma$ replaced by $\id_X \times \sigma'$, and $\tau$ replaced by $\sigma \times \id_{Y'}$, shows that		  
%		  In this case, Proposition~\ref{prop:composition} reduces to	\cite{KatzStapledon16}*{Corollary~4.7} and states that 
		 for any $(x,x') \in X \times X'$ and $(y,y') \in Y \times Y'$ such that $(\sigma(x),\sigma(x')) \le (y,y') \in Y \times Y'$, 
	\begin{align*}
					\ell_{\sigma \times \sigma'}((x,x'),(y,y'))  = \sum_{(\tilde{x},\tilde{y}')} \ell_{\id_X \times \sigma'}((x,x'),(\tilde{x},\tilde{y}')) \ell_{\sigma \times \id_{Y'}}((\tilde{x},\tilde{y}'),(y,y')), 
	\end{align*}
			where the above sum varies over all $(\tilde{x},\tilde{y}') \in X \times Y'$ such that $(x,\sigma'(x')) \le (\tilde{x},\tilde{y}') \in X \times Y'$ and $(\sigma(\tilde{x}),\tilde{y}') \le (y,y') \in Y \times Y'$.
					By \eqref{eq:productidentity},
			$\ell_{\id_X \times \sigma'} = \delta_X \times \ell_{\sigma'}$ and $\ell_{\sigma \times \id_{Y'}} = \ell_\sigma \times \delta_{X'}$. Then the above expression simplifies to
				$\ell_{\sigma \times \sigma'}((x,x'),(y,y'))  = \ell_{\sigma}(x,y) \ell_{\sigma'}(x',y'),$ as desired.  
\end{example}

\subsection{Proofs of main results}\label{ss:proofsmain}

In this section, we prove Proposition~\ref{prop:symmetry}, Theorem~\ref{thm:maingell}, Corollary~\ref{cor:rightKLSfunction}, and Corollary~\ref{cor:Zmappingformula}. 
  %and Proposition~\ref{prop:composition}. 
  We continue with the notation of Section~\ref{ss:statements}.

	We need the following lemma which follows from the assumption that $\kappa_\Gamma$ is multiplicative and rank alternating.

\begin{lemma}\label{lem:alternatingmultiplicative}
	With the setup of Section~\ref{ss:statements},
	\[
	\kappa_{\Gamma}|_{X/Y} = 
	\kappa_\Gamma|_{(X/Y)^\circ} \cdot \kappa_\Gamma = 
	- \kappa_\Gamma \cdot (\kappa_\Gamma|_{(X/Y)^\circ})^{\rev}.
	\]
\end{lemma}
\begin{proof}
	Consider $z \le z'$ in $\Gamma$. We need to show that 
	$\kappa_{\Gamma}|_{X/Y}(z,z') = (\kappa_\Gamma|_{(X/Y)^\circ} \cdot \kappa_\Gamma)(z,z') = - (\kappa_\Gamma \cdot (\kappa_\Gamma|_{(X/Y)^\circ})^{\rev})(z,z').$
	All terms are zero unless $z \in X$ and $z' \in Y$. 
	Assume that $z \in X$ and $z' \in Y$. 
	Using the fact that $\kappa_\Gamma$ is multiplicative, we compute
	\begin{align*}
		(\kappa_\Gamma|_{(X/Y)^\circ} \cdot \kappa_\Gamma)(z,z') &=  \kappa_\Gamma(z,\sigma(z)) \kappa_\Gamma(\sigma(z),z') = \kappa_\Gamma(z,z') = \kappa_{\Gamma}|_{X/Y}(z,z').
	\end{align*}
	Also,
	using the fact that $\kappa_\Gamma$ is multiplicative and rank alternating, as well as \eqref{eq:strongsubdivisionequality}, we compute
	\begin{align*}
		(\kappa_\Gamma \cdot (\kappa_\Gamma|_{(X/Y)^\circ})^{\rev})(z,z') &= \sum_{ \substack{z \le x' \in X \\ \sigma(x') = z'} } \kappa_\Gamma(z,x') \kappa_\Gamma(x',z')^{\rev} \\
		&= \sum_{ \substack{z \le x' \in X \\ \sigma(x') = z'} } \kappa_\Gamma(z,x') (-1)^{\rho_\Gamma(x',z')}\kappa_\Gamma(x',z') \\
		&= - \kappa_\Gamma(z,z') \sum_{ \substack{z \le x' \in X \\ \sigma(x') = z'} } (-1)^{\rho_\Gamma(x',z') - 1} \\
		&=  - \kappa_\Gamma(z,z'). 
	\end{align*}
\end{proof}

	\begin{proof}[Proof of Proposition~\ref{prop:symmetry}]
%		The equivalence of the two statements follows from the observation that 
%%		$t(t^{-1} - 1) = -(t - 1)$. 
%%	We need to show that for all $x \in X$ and $y \in Y$ such that $\sigma(x) \le y$, 
%		\[
%	\ell_{\sigma}(x,y;t) = t^{r_\Gamma(x,y) - 1}\ell_{\sigma}(x,y;t^{-1}).
%	\]	
%	if and only if 
%%	Equivalently, 
%			\[
%	(t - 1)\ell_{\sigma}(x,y;t) = - t^{r_\Gamma(x,y)}(t^{-1} - 1)\ell_{\sigma}(x,y;t^{-1}).
%	\]	
%	That is, 
	We need to show that $(t - 1) \cdot \ell_{\sigma}$ is antisymmetric. Recall from Definition~\ref{def:hellpolynomial} that $(t - 1) \cdot \ell_{\sigma} = g_\Gamma \cdot \kappa_{\Gamma}|_{(X/Y)^\circ} \cdot g_\Gamma^{-1}$.
	Recall that  $g_\Gamma^{\rev} = g_\Gamma \cdot \kappa_\Gamma$.  Using Lemma~\ref{lem:alternatingmultiplicative}, we compute
	\begin{align*}
		((t - 1) \cdot \ell_{\sigma})^{\rev} &= (g_\Gamma \cdot \kappa_{\Gamma}|_{(X/Y)^\circ} \cdot g_\Gamma^{-1})^{\rev}  \\
		&= g_\Gamma \cdot \kappa_\Gamma \cdot (\kappa_\Gamma|_{(X/Y)^\circ})^{\rev} \cdot (g_\Gamma^{-1})^{\rev} \\
				&= - g_\Gamma \cdot \kappa_\Gamma|_{(X/Y)^\circ} \cdot \kappa_\Gamma \cdot (g_\Gamma^{-1})^{\rev} \\
				&= - (t - 1) \cdot \ell_{\sigma}. \\
	\end{align*}

\end{proof}

Before proving Theorem~\ref{thm:maingell}, we need the following lemma.
It follows from the definition of the convolution product on 
$\II(\Gamma)$, together with the fact that $X$ is a lower order ideal of $\Gamma$, and $Y$ is an upper order ideal of $\Gamma$. 

\begin{lemma}\label{lem:restrict}
		With the setup of Section~\ref{ss:statements}, if 
	$p,p' \in \II(\Gamma)$, then
	\[
	(p \cdot p')|_X = p|_X \cdot p'|_X = p \cdot p'|_X,
	\]
	\[
	(p \cdot p')|_Y = p|_Y \cdot p'|_Y = p|_Y \cdot p',
	\]
	\[
	(p \cdot p')|_{X/Y} = p \cdot p'|_{X/Y} + p|_{X/Y} \cdot p'. 
	\]
\end{lemma}
\begin{proof}
	Consider $z \le z' \in \Gamma$. 
	If $z' \notin X$, then 	$(p \cdot p')|_X(z,z') = (p|_X \cdot p'|_X)(z,z') = (p \cdot p'|_X)(z,z') = 0$. If $z' \in X$, then $z \in X$, and $(p \cdot p')|_X(z,z') = (p|_X \cdot p'|_X)(z,z') = (p \cdot p'|_X)(z,z') = (p \cdot p')(z,z')$. This establishes the first set of equations. Similarly, if $z \notin Y$, then $(p \cdot p')|_Y(z,z') = (p|_Y \cdot p'|_Y)(z,z') = (p|_Y \cdot p)(z,z') = 0$. If $z \in Y$, then $z' \in Y$, and $(p \cdot p')|_Y(z,z') = (p|_Y \cdot p'|_Y)(z,z') = (p|_Y \cdot p)(z,z')  = (p \cdot p')(z,z')$. This establishes the  second  set of equations.
	If $z \in X$ and $z' \in Y$, then we compute
	\begin{align*}
		(p \cdot  p')|_{X/Y}(z,z') &= (p \cdot p')(z,z') \\
		&= 	 \sum_{\substack{z \le z'' \le z' \\ z'' \in X}} p(z,z'') p'(z'',z') + \sum_{\substack{z \le z'' \le z' \\ z'' \in Y}} p(z,z'') p'(z'',z') \\
		&= \sum_{z \le z'' \le z'} p(z,z'') p'|_{X/Y}(z'',z') + \sum_{z \le z'' \le z'} p|_{X/Y}(z,z'') p'(z'',z') \\
		&=  (p \cdot p'|_{X/Y})(z,z') + (p|_{X/Y} \cdot p')(z,z').
	\end{align*}
	Suppose that $z \notin X$ or $z' \notin Y$. Then $(p \cdot p')|_{X/Y}(z,z') = 0$ by definition.  
	If $(p \cdot p'|_{X/Y})(z,z')$ is nonzero,  then $z' \in Y$, and $z'' \in X$ for some $z \le z''$, and hence $z \in X$, a contradiction. Hence $(p \cdot p'|_{X/Y})(z,z') = 0$. Similarly, if  $(p|_{X/Y} \cdot p')(z,z')$ is nonzero, then $z \in X$ and 
	$z'' \in Y$ for some $z'' \le z'$, a contradiction. Hence $(p|_{X/Y} \cdot p')(z,z') = 0$. We conclude that the third  set of equations hold.
\end{proof}

\begin{proof}[Proof of Theorem~\ref{thm:maingell}]
	By Lemma~\ref{lem:restrict} applied to $g_\Gamma^{\rev}  = g_\Gamma \cdot \kappa_\Gamma$, we have
	\[
	(g_\Gamma|_{X/Y})^{\rev} = g_\Gamma \cdot \kappa_{\Gamma}|_{X/Y} + g_\Gamma|_{X/Y} \cdot \kappa_\Gamma. 
	\]
	By Lemma~\ref{lem:alternatingmultiplicative}, this is equivalent to 
	\[
	(g_\Gamma|_{X/Y})^{\rev} = (g_\Gamma \cdot \kappa_\Gamma|_{(X/Y)^\circ} + g_\Gamma|_{X/Y}) \cdot \kappa_\Gamma. 
	\]
	Multiplying both sides on the right by $(g_\Gamma^{-1})^{\rev} = \kappa_\Gamma^{-1} \cdot g_\Gamma^{-1}$ gives
	\[
	(g_\Gamma|_{X/Y} \cdot g_\Gamma^{-1})^{\rev} = (g_\Gamma \cdot \kappa_\Gamma|_{(X/Y)^\circ} + g_\Gamma|_{X/Y}) \cdot g_\Gamma^{-1}. 
	\]
	Using  Definition~\ref{def:hellpolynomial}, this is equivalent to
	\[
	(g_\Gamma|_{X/Y} \cdot g_\Gamma^{-1})^{\rev} - g_\Gamma|_{X/Y} \cdot g_\Gamma^{-1} = g_\Gamma \cdot \kappa_\Gamma|_{(X/Y)^\circ} \cdot g_\Gamma^{-1} = (t - 1) \cdot \ell_{\sigma}. 
	\]
	Since $g_\Gamma \in \II_{1/2}(\Gamma)$, we have $g_\Gamma|_{X/Y}  \in \II_{1/2}(\Gamma)$. It follows from Lemma~\ref{lem:inverse} that $g_\Gamma^{-1} \in \II_{1/2}(\Gamma)$. Since $\II_{1/2}(\Gamma)$ is a ring, we deduce that $g_\Gamma|_{X/Y} \cdot g_\Gamma^{-1} \in \II_{1/2}(\Gamma)$. 
	On the other hand, recall from \eqref{eq:tildeell} that 
		 $\Delta \ell_{\sigma}$ is the unique element in  $\II_{1/2}(\Gamma)$ satisfying
		$(\Delta \ell_{\sigma})^{\rev} - \Delta \ell_{\sigma}  = (t - 1) \cdot \ell_{\sigma}$. 
		We conclude that $g_\Gamma|_{X/Y} \cdot g_\Gamma^{-1} = \Delta \ell_{\sigma}$. Equivalently, 
		$g_\Gamma|_{X/Y} = \Delta \ell_{\sigma} \cdot g_\Gamma$. Finally, observe that $\Delta \ell_{\sigma} \cdot g_\Gamma = \Delta \ell_{\sigma} \cdot g_\Gamma|_Y$, since $\Delta \ell_{\sigma} = (\Delta \ell_{\sigma})|_{X/Y}$.

\end{proof}

			\begin{proof}[Proof of Corollary~\ref{cor:rightKLSfunction}]
	By Lemma~\ref{lem:inverse}, 
	$f_\Gamma \cdot \widehat{g_\Gamma} = \delta_\Gamma$. 
	By Lemma~\ref{lem:restrict}, 
	\[
	f_\Gamma|_{X/Y} \cdot \widehat{g_\Gamma} + f_\Gamma \cdot \widehat{g_\Gamma}|_{X/Y} = \delta_\Gamma|_{X/Y} = 0. 
	\]
	Multiplying both sides on the right by $(\widehat{g_\Gamma})^{-1}$, and using Theorem~\ref{thm:maingell}, we compute
	\begin{align*}
		f_\Gamma|_{X/Y} = - f_\Gamma \cdot \widehat{g_\Gamma}|_{X/Y} \cdot (\widehat{g_\Gamma})^{-1} 
		= - f_\Gamma \cdot \Delta \widehat{\ell_{\sigma}} \cdot \widehat{g_\Gamma} \cdot (\widehat{g_\Gamma})^{-1} 
		= - f_\Gamma \cdot \Delta \widehat{\ell_{\sigma}}. 
	\end{align*}
	 Finally, observe that $f_\Gamma \cdot \Delta \widehat{\ell_{\sigma}} =  - f_\Gamma|_X \cdot \Delta \widehat{\ell_{\sigma}}$, since $\Delta \widehat{\ell_{\sigma}} = (\Delta \widehat{\ell_{\sigma}})|_{X/Y}$.  When expanding terms we also use that $-\Delta \widehat{\ell_{\sigma}}(x',y) =   -(-1)^{\rho_\Gamma(x',y)} \Delta \ell_{\sigma}(x',y)
	 =   (-1)^{\rho_Y(y) - \rho_X(x')} \Delta \ell_{\sigma}(x',y)$. 
\end{proof}

			\begin{proof}[Proof of Corollary~\ref{cor:Zmappingformula}]
					Recall from Theorem~\ref{thm:maingell} that $g_\Gamma|_{X/Y} = \Delta \ell_{\sigma} \cdot  g_\Gamma$. Recall from Corollary~\ref{cor:rightKLSfunction} that 
				$f_\Gamma|_{X/Y} = - f_\Gamma \cdot \Delta \widehat{\ell_{\sigma}}$. 
	Applying Lemma~\ref{lem:restrict} to the equality $Z_\Gamma = g_\Gamma^{\rev} \cdot f_\Gamma$,  we compute
	\begin{align*}
		Z_\Gamma|_{X/Y} &= g_\Gamma^{\rev} \cdot f_\Gamma|_{X/Y} + (g_\Gamma|_{X/Y})^{\rev} \cdot f_\Gamma \\ 
		&= Z_\Gamma \cdot f_\Gamma^{-1} \cdot f_\Gamma|_{X/Y} + (g_\Gamma|_{X/Y})^{\rev} \cdot (g_\Gamma^{-1})^{\rev} \cdot Z_\Gamma \\
		&= - Z_\Gamma \cdot \Delta \widehat{\ell_{\sigma}} + (\Delta \ell_{\sigma})^{\rev} \cdot Z_\Gamma
		\\
				&= - Z_\Gamma|_X \cdot \Delta \widehat{\ell_{\sigma}} + (\Delta \ell_{\sigma})^{\rev} \cdot Z_\Gamma|_Y. 
	\end{align*}
	When expanding terms we also use that $$-\Delta \widehat{\ell_{\sigma}}(x',y) =   -(-1)^{\rho_\Gamma(x',y)} \Delta \ell_{\sigma}(x',y)
	=   (-1)^{\rho_Y(y) - \rho_X(x')} \Delta \ell_{\sigma}(x',y).$$ 
	
\end{proof}

\section{Applications to equivariant Kazhdan-Lusztig-Stanley theory}\label{sec:equivariant}

The main goal of this section is to state and prove equivariant versions of the results of Section~\ref{ss:statements}. After recalling and then further developing equivariant KLS theory, we will, in fact, deduce these results as corollaries of the non-equivariant results.
In Section~\ref{ss:representationtheory} we recall some basic representation theory. In Section~\ref{ss:backgroundequivariantKLS} we recall background on equivariant KLS theory. In Section~\ref{ss:classfunctions}, we 
 consider the effect on equivariant KLS theory of evaluating the corresponding class functions.
In Section~\ref{ss:equivariantKLSlowerEulerian}, we develop equivariant generalizations of results and examples in Section~\ref{sec:lowerEulerian}. 
%In Section~\ref{ss:examplesmainequiv}, we give some examples. 
Finally, in Section~\ref{ss:equivEhrhart}, we give some applications to equivariant Ehrhart theory. 

Recall that all groups are finite and all vector spaces are finite-dimensional. 
If a group $W$ acts on a set $S$, then for any $w \in W$, let 
$S^w = \{ s \in S : w \cdot s = s \}$ denote the corresponding fixed point set. 	Recall that given a field $k$ and a $\Z$-module $A$, we write $A_k = A \otimes_\Z k$.

\subsection{Background on representation theory}\label{ss:representationtheory}

We briefly recall some basic facts about representation theory of finite groups. We refer the reader to 
\cite{SerreLinearRepresentations} for more details.

Let $W$ be a finite group with identity element $\id \in W$ and complex group algebra $\C[W]$.  
All $\C[W]$-modules in this paper are finite dimensional.  Given a $\C W$-module $V$ with corresponding complex representation $\psi: W \to \GL(V)$, its associated \emph{character} is the function 
$\chi_V: W \to \C$,  $\chi_V(w) = \tr(\psi(w))$ for all $w \in W$, where $\tr$ denotes the trace function. 
%A \emph{complex representation} of $W$ is a group homomorphism $\psi: W \to \GL(V)$ for some complex vector space $V$. Throughout this paper, we will always assume that $V$ is finite-dimensional. 
%Equivalently, $V$ has the structure of a $\C[W]$-module, where $\C[W]$ is the complex group algebra of $W$. 
The complex \emph{representation ring} $R(W)$ 
is the free abelian group generated by isomorphism classes of finite-dimensional $\C W$-modules, 
modulo the relation $[V \oplus V'] = [V] + [V']$ for $\C W$-modules $V,V'$. Multiplication is defined by $[V][V'] := [V \otimes_\C V']$ and the identity element $1 \in R(W)$ is the class %$[\1]$ 
of the trivial representation. 
Below, we will sometimes write $[\psi] = [V] \in R(W)$.  
As a $\Z$-module, $R(W)$ is a free module with basis given by the classes of the irreducible complex representations of $W$. In particular, the natural map of rings $\Z \to R(W)$ is an inclusion, and induces an inclusion $\Z[t] \to R(W)[t]$. 

A \emph{class function} is a function $\phi: W \to \C$ which is constant on conjugacy classes of $W$. Let $\cf(W)$ be the set of all class functions. Then $\cf(W)$ has the structure of a $\C$-algebra; given $\phi,\phi' \in \cf(W)$, $\lambda \in \C$, and $w \in W$, $(\lambda \cdot \phi)(w) = \lambda \phi(w)$, $(\phi + \phi')(w) = \phi(w) + \phi'(w)$, $(\phi \cdot \phi')(w) = \phi(w)\phi'(w)$. 
%Given a polynomial $\phi = \sum_i \phi_i t^i \in \cf(W)[t]$, we let $\phi(w) = \sum_i \phi_i(w) t^i   \in \C[t]$. 
The assignment 
$V \mapsto \chi_V$ induces a $\C$-algebra isomorphism between $R(W)_\C$ and $\cf(W)$. For any $w \in W$, evaluation of the corresponding class functions at $w \in W$ induces a map
$\ev_w : R(W)_\C \to \C$,
such that $\ev_w([V]) = \chi_V(w)$ for all $\C W$-modules $V$. For example, evaluation at the identity element is determined by $\ev_{\id}([V]) = \dim V$ for all $\C W$-modules $V$. Moreover, by applying $\ev_w$ coefficientwise, we may extend $\ev_w$ to a map $\ev_w: R(W)[[t]] \to \C[[t]]$. 

Given a group homomorphism $\eta: W' \to W$, pullback along $\eta$ induces a 
representation $\Lambda_{\eta}(V)$, i.e., $\Lambda_{\eta}(V)$ is $V$ regarded as a $\C[W']$-module via the induced map $\C[W'] \to \C[W]$. This extends to a ring homomorphism $\Lambda_{\eta}: R(W) \to R(W')$. 
%We also have a corresponding map on class functions such that $\Lambda_{\eta}(\phi)(w') = \phi(\eta(w'))$ for all $\phi \in \cf(W)$.  
By acting coefficientwise, we will also consider the induced map $\Lambda_{\eta}: R(W)[[t]] \to R(W')[[t]]$. This is compatible with composition in the sense that if $\nu: W'' \to W'$, then $\Lambda_{\eta \circ \nu} = \Lambda_\nu \circ \Lambda_\eta$.

Let $W'$ be a subgroup of $W$ with inclusion map $\iota: W' \to W$. Let $V$ be a $\C[W]$-module and let $V'$ be a $\C[W']$-module. The \emph{restricted representation} $\Res_{W'}^W V$ of $V$ is $\Lambda_{\iota}(V)$. 
%, where $\eta : W' \to W$ is the inclusion map. 
%$V$ regarded as a $\C[W']$-module via the inclusion $\C[W'] \to \C[W]$. 
The  \emph{induced representation} $\Ind_{W'}^W V'$ of $V'$ is $\C[W] \otimes_{\C[W']} V'$. These operations extend to additive maps
$\Res_{W'}^W: R(W) \to R(W')$ and 
$\Ind_{W'}^W: R(W') \to R(W)$. 
 By acting coefficientwise, we have induced maps
$\Res_{W'}^W: R(W)[[t]] \to R(W')[[t]]$ and 
$\Ind_{W'}^W: R(W')[[t]] \to R(W)[[t]]$.
%We also have a compatible notion of restriction and induction for class functions. Given a class function $\phi: W \to \C$, $\Res_{W'}^W \phi = \Lambda_{\iota}(\phi)$. We will give an example of induction of class functions below.
%% is the composition of the inclusion $W' \to W$ with $\phi$. 

Suppose that $W$ acts on a set $X$. For any $x \in X$, let $W_x$ denote the stabilizer of $x$. For any $w \in W$, conjugation by $w$ induces a group isomorphism
$$\eta^w_x : W_x \to W_{w \cdot x},$$ 
$$\eta^w_x(w') =  w w' w^{-1}.$$ 
Then pullback induces a ring isomorphism
 $$\Lambda_{\eta^w_x}: R(W_{w \cdot x})[[t]] \to R(W_x)[[t]].$$
	   Suppose that we have  a collection $\{ p_x \in R(W_x)[[t]] : x \in X \}$ %of class functions
satisfying $p_x = \Lambda_{\eta^w_x}(p_{w \cdot x})$ for all  $x \in X$ and $w \in W$. 
Recall that if $W$ acts on a set $S$ and $w \in W$, then  
$S^w = \{ s \in S : w \cdot s = s \}$ denotes the corresponding fixed point set. Fix an element $x_0 \in X$. 
Then for all $w \in W$, 
\begin{equation}\label{eq:inducedcharacter}
	\ev_w(\Ind^{W}_{W_{x_0}} p_{x_0}) = 
	\sum_{x \in (W \cdot x_0)^w} \ev_w(p_x).
	%\sum_{ \substack{ x \in X \\ w \cdot x = x }} p_x(w).
\end{equation} 
%Summing over all choices of $x_0$ in the orbit $W \cdot x_0$ gives 
%\begin{equation}\label{eq:inducedcharacterbasepointfree}
%	\sum_{x \in W \cdot x_0} \ev_w(\Ind^{W}_{W_{x}} p_{x}) = |W \cdot x_0| 
%	\sum_{x \in (W \cdot x_0)^w} \ev_w(p_x).
%	%\sum_{ \substack{ x \in X \\ w \cdot x = x }} p_x(w).
%\end{equation} 
For example, if $p_x = 1$ for all $x \in X$, then $\Ind^{W}_{W_{x_0}} p_{x_0}$ is the permutation representation of $W$ acting on $W \cdot x_0$. If $\pi$ denotes the permutation representation of $W$ acting on $X$, then   \eqref{eq:inducedcharacter} implies that $\ev_w(\pi) = |X^w|$.

Consider the representation $\psi: W \to \GL(V)$. 
For $m \in \Z_{\ge 0}$, we may consider the corresponding representation $\bigwedge^m \psi: W \to \GL(\bigwedge^m V).$ 
%For example, when $m = 0$, $\bigwedge^m \psi$ is the trivial representation, when $m = 1$, $\bigwedge^m \psi = \psi$, and when $m > \dim V$, 
We define $\det(I - \psi t ) \in R(W)[t]$ by 
\begin{equation}\label{eq:detI-rhot}
	\det(I - \psi t ) := \sum_{m \ge 0} (-1)^m [ \bigwedge^m V ]	t^m \in R(W)[t]. 
\end{equation}
Also, let $\det(\psi) = [\bigwedge^{\dim V} \psi] = [ \bigwedge^{\dim V} V ] \in R(W)$. 
%If $\psi: W \to \GL(V)$ is the corresponding representation and 
Consider an element $w \in W$. If $\psi(w)$ has eigenvalues $\lambda_1,\ldots,\lambda_{\dim V}$, then $\ev_w(\det(I - \psi t)) = \det(I - \psi(w)t) = \prod_{i = 1}^{\dim V} (1 - \lambda_i t)$ (see, for example, \cite[Lemma~3.1]{StapledonEquivariant}). Similarly, we will also consider $\det(tI - \psi) := t^{\dim V} \det(I - \psi t^{-1}) \in R(W)[t]$. Then 
$\ev_w(\det(tI - \psi))  = \det(tI - \psi(w)) = \prod_{i = 1}^{\dim V} (t - \lambda_i)$. 
%The constant coefficients of the polynomials $\det(I - \psi t)$ and $\det(tI - \psi)$ in $R(W)[t]$ are $1$ and $(-1)^{\dim V} \det(\psi)$ respectively in $R(W)$. 
%%= [\1_W] 
%In particular, $\det(I - \psi t)$ and  $\det(tI - \psi)$ are invertible in $R(W)[[t]]$. 
We will use the following simple lemma (see, for example, \cite{StapledonEquivariant}*{Remark~5.5} and the proof of \cite{StapledonCalabi12}*{Lemma~2.3}). We reproduce the proof for the benefit of the reader. Note that for any $w \in W$, the fixed locus $V^w$ is a subspace of $V$. 

% The constant coefficient of $\det(tI - \psi)$ is $(-1)^{\dim V} \det(\psi)
% \in R(W)$. In particular, $\det(tI - \psi)$ is invertible in $R(W)[[t]]$.

\begin{lemma}\label{lem:detswitcht}
	Consider a real representation $\psi: W \to \GL(V)$. Then 
	\[
	\det(\psi) \det(tI - \psi) = (-1)^{\dim V}  \det(I - \psi t ) \in R(W)[t].
	\]  
	Moreover, for all $w \in W$, $\ev_w(\det(\psi)) = (-1)^{\dim V - \dim V^w}$. 
\end{lemma}
\begin{proof}
	Consider an element $w \in W$. We need to show that both sides agree after applying   $\ev_w$. Suppose that $\psi(w)$ has eigenvalues $\lambda_1,\ldots,\lambda_{\dim V}$. Since $w$ has finite order, each $\lambda_i$ is a root of unity. Hence the complex conjugate of $\lambda_i$ is 	$\lambda_i^{-1}$. Since $V$ is a real vector space, the multiset
	$\{ \lambda_1,\ldots,\lambda_{\dim V} \}$ is closed under complex conjugation. We compute
	\begin{align*}
		\prod_{i = 1}^{\dim V} \lambda_i	\prod_{i = 1}^{\dim V} (t - \lambda_i) = \prod_{i = 1}^{\dim V} \lambda_i \prod_{i = 1}^{\dim V} (t - \lambda_i^{-1}) 
		= (-1)^{\dim V}   \prod_{i = 1}^{\dim V} (1 - \lambda_i t). \\
	\end{align*} 
	This establishes the equation. For the final statement, observe that 
	$\dim V^w = |\{ i : \lambda_i = 1 \}|$. Since the $\lambda_i$ occur in conjugate pairs with $\lambda_i \leftrightarrow \lambda_i^{-1}$, 
		$|\{ i : \lambda_i = 1 \}| +  |\{ i : \lambda_i = -1 \}| = \dim V$ mod $2$,  and 
	$$\ev_w(\det(\psi)) = \prod_{i = 1}^{\dim V} \lambda_i = (-1)^{|\{ i : \lambda_i = -1 \}|} = (-1)^{\dim V - \dim V^w}.$$
\end{proof}

\subsection{Background on equivariant Kazhdan-Lusztig-Stanley theory}\label{ss:backgroundequivariantKLS}

In this section, we recall some basic facts on equivariant KLS theory, which is a generalization of the theory discussed in Section~\ref{ss:KLSbackground}. We refer the reader to 
 \cite{ProudfootEquivariantKLS} for more details.   We continue with the notation of Section~\ref{ss:representationtheory}. 
 
Recall that all posets in this paper are finite. 
Let $B$ be a poset. 
In what follows, we will often denote an interval $[z,z'] \in \Int(B)$ simply by $z,z'$. 
Let $W$ be a finite group acting on $B$. This means that $W$ acts on $B$ as a set, and for all $w \in W$ and $z \le z' \in B$, $w \cdot z \le w \cdot z'$.
In particular, we have an induced action of $W$ on $\Int(B)$.  When $W$ acts trivially on $B$, the theory below reduces to Kazhdan-Lusztig-Stanley theory as described in Section~\ref{ss:KLSbackground}.

Given elements $z \le z'' \le z' \in B$, let $W_z  = \{ w \in W : w \cdot z = z \} \subset W$ be the stabilizer of $z$, and let 
$W_{z,z'} = W_z \cap W_{z'}$ and $W_{z,z'',z'} = W_z \cap W_{z''} \cap W_{z'}$.  
%For any subgroup $W' \subset W$, recall that $R(W')$ denotes the complex representation ring of $W'$.
Consider the action of $W$ on $\Int(B)$. 
Recall from Section~\ref{ss:representationtheory} %(setting $X = \Int(B)$) 
that 
%Then 
for any $w \in W$, conjugation by $w$ and pullback induce a ring homomorphism
\[
\Lambda_{\eta^w_{z,z'}}: R(W_{w \cdot z, w \cdot z'})[t] \to R(W_{z,z'})[t].
\]
For $z \le z'' \le z'$, conjugation by $w$ and pullback also induce a ring homomorphism
\[
\Lambda_{\eta^w_{z,z'',z'}}: R(W_{w \cdot z, w \cdot z'', w \cdot z'})[t] \to R(W_{z,z'',z'})[t].
\]
% Given $w \in W$, let $\eta^w: W \to W'$, $\eta^w(w') = w w' w^{-1}$, denote conjugation by $w$. 
%Given $z \le z' \in B$ and $w \in W$, $\eta^w$ restricts to a group isomorphism
%\[
%\eta^w_{z,z'} : W_{z,z'} \to W_{w \cdot z, w \cdot z'}.
%\]
%Then pullback of representations induces a ring isomorphism 
%\[
%\Lambda_{\eta^w_{z,z'}}: R(W_{w \cdot z, w \cdot z'}) \to R(W_{z,z'}).
%\]
%The corresponding map on class functions satisfies 
%$\Lambda_{\eta^w_{z,z'}}(\phi)(w') = \phi(w w' w^{-1})$ ....
%As in Section~\ref{ss:KLSbackground}, given a function $p$ on $\Int(B)$, we often write
%$p(z,z')$ to denote $p([z,z'])$.
\begin{definition}\label{def:equivariantincidencealgebra}\cite{ProudfootEquivariantKLS}*{Section~2}
	The  \emph{equivariant incidence algebra} $I^W(B)$ is the set of functions  $p: \Int(B) \to \cup_{W' \subset W} R(W')[t]$ satisfying
	\begin{enumerate}
		\item $p(z,z') \in R(W_{z,z'})[t]$ for any $z \le z' \in B$,
		\item\label{i:equivincidence} $p(z,z') = \Lambda_{\eta^w_{z,z'}}(p(w \cdot z, w \cdot z'))$  for any $z \le z' \in B$ and $w \in W$.
	\end{enumerate}
\end{definition}
The equivariant incidence algebra has the structure of 
%a ring. 
an associative $\Z[t]$-algebra.
Given elements $p,p' \in I^W(B)$, addition is defined by 
$(p + p')(z,z') = p(z,z') + p'(z,z')$.
Multiplication 
is given by (see \cite{ProudfootEquivariantKLS}*{Proposition~2.3})
\begin{equation}\label{eq:equivariantmultiplication}
(p \cdot p')(z,z') = \sum_{z \le z'' \le z'} \frac{|W_{z,z'',z'}|}{|W_{z,z'}|}
\Ind^{W_{z,z'}}_{W_{z,z'',z'}} \left( \Res^{W_{z,z''}}_{W_{z,z'',z'}} p(z,z'') \cdot  \Res^{W_{z'',z'}}_{W_{z,z'',z'}} p'(z'',z')  \right). 
\end{equation}
As explained in 
\cite{ProudfootEquivariantKLS}*{Remark~2.4}, although the expression on the right hand side of \eqref{eq:equivariantmultiplication} contains fractions, it is in fact a well-defined element of $R(W_{z,z'})[t]$; an expression involving only integers can be obtained by allowing   $z''$ to vary over a choice of representatives for the $W_{z,z'}$-orbits of $[z,z']$. 
As in Section~\ref{ss:KLSbackground}, the identity element $\delta^W_B \in I^W(B)$ is defined by $\delta^W_B(z,z') = 1$ if $z = z'$, and 
$\delta^W_B(z,z') = 0$ if $z \neq z'$.
%\[
%\delta^W_B(z,z') = \begin{cases}
%	1 &\textrm{if } z = z', \\
%	0 &\textrm{otherwise. }
%\end{cases}
%\]
As in Section~\ref{ss:KLSbackground}, the $\Z[t]$-module structure is defined as follows:
given $a \in \Z[t]$ and $p \in I^W(B)$, $(a \cdot p)(z,z') = a p(z,z')$ for all $z \le z' \in B$. 
A weak rank function $r_B$ is $W$-invariant if $r_B(z,z') = r_B(w \cdot z, w \cdot z')$ for all $z \le z' \in B$ and $w \in W$. 

%Let $r_B \in I(B)$ be a weak rank function that is $W$-invariant in the sense that $r_B(z,z') = r_B(w \cdot z, w \cdot z')$ for all $z \le z' \in B$ and $w \in W$. 

\begin{example}\label{ex:Winvariantrank}
	Suppose that $B$ is a lower Eulerian poset with rank function $\rho_B$, and $W$ acts on $B$.  
	%Suppose that $B$ is a ranked poset with rank function $\rho_B$ and unique minimal element $\hat{0}_B$. 
	Then $\rho_B$ is $W$-invariant in the sense that $\rho_B(z) = \rho_B(w \cdot z)$ for all $z \in B$ and $w \in W$. This follows, for example, by induction on $\rho_B(\hat{0}_B,z)$ for $z \in B$. 
	We deduce that  
	the natural weak rank function is $W$-invariant. 
%	$r_B = \rho_B$ is a $W$-invariant weak rank function.
\end{example}

Fix a $W$-invariant weak rank function $r_B$. 
Consider the following subring of $I^W(B)$
\[
\II^W(B) = \{ p \in I^W(B) : \deg(p(z, z')) \le r_B(z, z') \textrm{ for all } z \le z' \in B \}  \subset I^W(B).
\]
Then $\II^W(B)$
admits a ring involution 
$p \mapsto p^{\rev}$, where
$(p^{\rev})(z, z';t) := t^{r_B(z, z')} p(z, z';t^{-1})$ for any $z \le z'$ in $B$. An element $p \in \II^W(B)$ is \emph{symmetric} if $p = p^{\rev}$, and  
\emph{antisymmetric} if $p = -p^{\rev}$. 
Let 
$U^W(B) = \{ p \in I^W(B) : p(z,z) = 1 \textrm{ for all } z \in B \}$.
Observe that $U^W(B)$ is closed under multiplication. All elements of $U^W(B)$ are invertible by \cite{ProudfootEquivariantKLS}*{Proposition~2.8}. 
%, and $U^W(B)$ is a subring of $I(B)$. 
%We need the following notation.   
Consider the subring 
\[
\II_{1/2}^W(B) = \{ p \in \II^W(B) : \deg(p(z, z')) < r_B(z, z')/2 \textrm{ for all } z < z' \in B \} \subset \II^W(B).
\]
%Observe that $\II_{1/2}^W(B)$ is a subring of $\II^W(B)$.
We may define 
$\Delta: \II^W(B) \to \II_{1/2}^W(B), p \mapsto \Delta p$ by equation  \eqref{eq:Delta}. 
An element $\kappa_B \in \II^W(B) \cap U^W(B)$ is an \emph{equivariant $B$-kernel} if 
%$\kappa(x, x) = 1$ for all $x \in B$ and 
$\kappa_B^{-1} =  \kappa_B^{\rev}$. 
By \cite{ProudfootEquivariantKLS}*{Theorem~3.1}, there exist
 unique elements $f_B,g_B \in \II_{1/2}^W(B) \cap U^W(B)$ such that $f_B^{\rev} = \kappa_B \cdot f_B$  and 
$g_B^{\rev} = g_B \cdot \kappa_B$. The elements $f_B$ and $g_B$  are called the 
\emph{right equivariant Kazhdan-Lusztig-Stanley function} and \emph{left equivariant Kazhdan-Lusztig-Stanley function} respectively. 
The \emph{equivariant $Z$-function}  is $Z_B^W = g_B \cdot \kappa_B \cdot f_B \in \II^W(B) \cap U^W(B)$. Proudfoot showed that $Z_B = g_B^{\rev} \cdot f_B = g_B \cdot f_B^{\rev}$ and hence $Z_B$ is symmetric, i.e., $Z_B = Z_B^{\rev}$ \cite{ProudfootEquivariantKLS}*{Proposition~3.3}.

\begin{remark}\cite{ProudfootAGofKLSpolynomials}*{Remark~3.4}\label{rem:pullback}
	All the definitions above are compatible with pullback in the following sense. 
	Given a group homomorphism $\eta: W' \to W$, recall from Section~\ref{ss:representationtheory} that pullback induces a ring homomorphism $\Lambda_{\eta}: R(W)[t] \to R(W')[t]$. Then  $\Lambda_{\eta}$ induces a ring homomorphism from $I^W(B)$ to $I^{W'}(B)$ that takes $\II^W(B)$ to  $\II^{W'}(B)$ and takes  $\II_{1/2}^W(B)$ to  $\II_{1/2}^{W'}(B)$. Moreover, $\Lambda_{\eta}(\kappa_B) \in \II^{W'}(B) \cap U^{W'}(B)$ is an 
	equivariant $B$-kernel with corresponding right and left equivariant Kazhdan-Lusztig-Stanley functions
	$\Lambda_{\eta}(f_B)$ and $\Lambda_{\eta}(g_B)$ respectively, and equivariant $Z$-function
	$\Lambda_{\eta}(Z_B)$. 
\end{remark}

\begin{example}\label{ex:lowdegreetermsequiv}
	We have the following equivariant generalization of Example~\ref{ex:lowdegreeterms}. 	Recall that given a polynomial $a \in \Z[t]$,  $a_i$ denotes the coefficient of $t^i$ in $a$. Then for any $z \le z'$ in $B$, comparing highest degree terms in the expressions $g_B^{\rev} = g_B \cdot \kappa_B$, $f_B^{\rev} = \kappa_B \cdot f_B$, and $Z_B = g_B \cdot \kappa_B \cdot f_B$ evaluated at $[z,z']$ gives
	\[
	g_B(z,z')_0 = f_B(z,z')_0 = \kappa_B(z,z')_{r_B(z,z')} = Z_B(z,z')_0 = Z_B(z,z')_{r_B(z,z')} \in R(W_{z,z'}). 
	\]

	In particular, if $r_B(z,z') \le 2$, then $g_B(z,z') = f_B(z,z') = \kappa_B(z,z')_{r_B(z,z')}$. Assume that $r_B(z,z') > 2$. Then comparing second highest degree terms in the expressions $g_B^{\rev} = g_B \cdot \kappa_B$ and  $f_B^{\rev} = \kappa_B \cdot f_B$ evaluated at $[z,z']$ gives the following equalities in $R(W_{z,z'})$,
	\[
	g_B(z,z')_1 = \kappa_B(z,z')_{r_B(z,z') - 1} + \sum_{\substack{z \le z'' \le z' \\ r_B(z,z'') = 1 }} \frac{|W_{z,z'',z'}|}{|W_{z,z'}|} \Ind^{W_{z,z'}}_{W_{z,z'',z'}} \Res^{W_{z'',z'}}_{W_{z,z'',z'}} \kappa_B(z'',z')_{r_B(z'',z')},
	\]	
	\[
	f_B(z,z')_1 = \kappa_B(z,z')_{r_B(z,z') - 1} + \sum_{\substack{z \le z'' \le z' \\ r_B(z'',z') = 1 }} \frac{|W_{z,z'',z'}|}{|W_{z,z'}|} \Ind^{W_{z,z'}}_{W_{z,z'',z'}} \Res^{W_{z,z''}}_{W_{z,z'',z'}} \kappa_B(z,z'')_{r_B(z,z'')}.
	\]	
%	Equivalently, these can be deduced from Example~\ref{ex:lowdegreeterms} by applying $\ev_w$ for all $w \in W$ and using Lemma~\ref{lem:Ztalgebrahom}. 
\end{example}

\begin{example}\label{ex:fanequiv}
	Suppose that $\Sigma$ is a %full-dimensional 
	fan in a real vector space $V$, and $\psi: W \to \GL(V)$ is a real representation of a finite group $W$ such that $\Sigma$ is $W$-invariant. %Assume that the support $|\Sigma|$ of $\Sigma$ is convex.  
	Consider the lower Eulerian poset $B = \face(\Sigma)$ with 
		the natural weak rank function $\rho_B$, which is $W$-invariant by
		Example~\ref{ex:Winvariantrank}.
		  %and $\kappa_\Gamma$ is the Eulerian kernel, 
	%, and let $\rho_B$ be the natural rank function.  Then $\rho_B$ is $W$-invariant by 
%	Example~\ref{ex:Winvariantrank}. Let $r_B = \rho_B$ and 
    Define $\kappa_B \in \II^W(B)$ as follows. 
	For any $z \in B$, let $C_z$ be the corresponding cone in $\Sigma$, and let $V_z$ be the linear span of $C_z$. 
	For any $z \le z' \in B$, $\psi$ induces a representation  $\psi_{z,z'}: W_{z,z'} \to \GL(V_{z'}/V_z)$. Define
	\[
			\kappa_B(z,z') = 
		\det(tI - \psi_{z,z'}). 
	\]
	Then $\kappa_B$ is an equivariant $B$-kernel by Example~\ref{ex:fanequivv2} below. By Example~\ref{ex:lowdegreetermsequiv}, 
	$g_B(z,z')_0 = f_B(z,z')_0 = Z_B(z,z')_0 = Z_B(z,z')_{\rho_B(z,z')} = 1 \in R(W_{z,z'})$. For $\rho_B(z,z') > 2$, 
	$g_B(z,z')_1 = [\pi] - [\psi_{z,z'}]$, where $\pi$ is the permutation representation of $W_{z,z'}$ acting on $\{ z'' \in [z,z'] : \rho_B(z,z'') = 1\}$, and $f_B(z,z')_1 = [\pi'] - [\psi_{z,z'}]$, where $\pi'$ is the permutation representation of $W_{z,z'}$ acting on $\{ z'' \in [z,z'] : \rho_B(z'',z') = 1\}$.

	For example, suppose that $z \le z' \in B$ and $C_{z'}$ is a simplicial cone. Then $\psi_{z,z}$ is the permutation representation of $W_{z,z'}$ acting on the set $S$ of rays of $C_{z'} \smallsetminus C_{z}$.  Suppose that $w \in W_{z,z'}$ acts on $S$ with orbits of size $m_1,\ldots,m_s$. Then 
	$\ev_w(\kappa_B(z,z')) = \prod_{i = 1}^s (t^{m_i} - 1)$. 
	We claim that  $f_B(z,z') = g_B(z,z') = 1$. 
	Moreover, for $0 \le i \le \rho_B(z,z')$, let $\pi_i$ be the permutation representation of $W_{z,z'}$ acting on subsets of $S$ of size $i$. 
	%, and let $[\pi_i]$ be the corresponding element in $R(W_{z,z'})$. 
	We claim that 
	$Z_B(z,z') = \sum_{i = 0}^{\rho_B(z,z')} [\pi_i] t^i$.  When $W$ acts trivially, these claims reduce to Example~\ref{ex:t-1caseBn}.

	The first claim follows by induction and the computation
%	Using induction and the identity $t^{\sum_{i \in [s]} m_i} = \sum_{I \subset [s]} \prod_{i \in I} (t^{m_i} - 1)$, one may directly compute that $f_B(z,z') = g_B(z,z') = 1$. 
	\[
	\ev_w(f_B(z,z')^{\rev}) - \ev_w(f_B(z,z')) = \ev_w(g_B(z,z')^{\rev}) - \ev_w(g_B(z,z')) =  \sum_{\emptyset \nsubseteq I \subset [s]} \prod_{i \in I} (t^{m_i} - 1) = t^{\rho_B(z,z')} - 1.  
	\]
	For the second claim, we compute 
	\[
	\ev_w(Z_B(z,z')) = \ev_w((g_B^{\rev} \cdot f_B)(z,z')) = \sum_{I \subset [s]} t^{\sum_{i \in I} m_i} =  \prod_{i = 1}^s (1 + t^{m_i}). 
	\]
	Observe that the coefficient of $t^i$ in $\prod_{i = 1}^s (1 + t^{m_i})$ is the number of subsets of $S$ of size $i$ that are $w$-invariant. The latter is $\ev_w([\pi_i])$.

%	$\Sigma = C$ is a simplicial cone and let $n = \dim V$. Then $B = B_n$ and $\psi$ is the permutation representation of $W$ acting on the rays of $C$. For $z \le  z'$, suppose that $w \in W_{z,z'}$ acts on the rays of $C_{z'} \smallsetminus C_{z}$ with orbits of size $m_1,\ldots,m_s$. Then 
%	$\ev_w(\kappa_B(z,z')) = \prod_{i = 1}^s (t^{m_i} - 1)$. Using induction and the identity $t^{\sum_{i \in [s]} m_i} = \sum_{I \subset [s]} \prod_{i \in I} (t^{m_i} - 1)$, one may directly compute that $f_B(z,z') = g_B(z,z') = 1$. 
%	\alan{TODO}
%	Compare with Example~\ref{ex:t-1caseBn}. 
%	What about $Z$??
%	Alternative proof: once you know that $g_B$ and $f_B$ (do we know this?? yes; by looking at the dual) are effective, then reduce to nonequivariant calculation...
\end{example}

\subsection{Evaluation of class functions}\label{ss:classfunctions}

We now consider the effect on equivariant Kazhdan-Lusztig-Stanley theory of evaluating the corresponding class functions. We continue with the notation of Section~\ref{ss:KLSbackground} and  Section~\ref{ss:backgroundequivariantKLS}. 

Let $W$ be a finite group acting on a poset $B$ with $W$-invariant weak rank function $r_B$ and equivariant $B$-kernel $\kappa_B$. Recall that we have corresponding equivariant KLS invariants $f_B$,$g_B$, and $Z_B$. 
%Recall that the finite group $W$ acts on a poset $B$ with $W$-invariant weak rank function $r_B \in I(B)$ and equivariant $B$-kernel $\kappa_B \in \II^W(B) \cap U^W(B)$, with corresponding
%equivariant Kazhdan-Lusztig-Stanley functions
%$f_B,g_B \in \II_{1/2}^W(B) \cap U^W(B)$, and equivariant $Z$-function
%$\Lambda_{\eta}(Z_B) \in \II_{1/2}^W(B) \cap U^W(B)$. 
Fix an element $w \in W$. Consider the fixed locus $B^w = \{ z \in B : w \cdot z = z \} \subset B$ with its induced poset structure. 
By definition, $z \in B^w$ if and only if $w \in W_z$. 
Let $r_B^w$ be the restriction of $r_B$ to $B^w$. Then 
$r_B^w$ is a weak rank function. We fix this choice of weak rank function, and may consider $\II(B^w)$ and $\II_{1/2}(B^w)$. 

\begin{remark}\label{rem:rankwarning}
	Suppose that $r_B$ is the natural weak rank function. 
   In general, 
	one should not expect $r_B^w$ to be the natural weak rank function. 
%	induced by a rank function for $B^w$, even if $B^w$ is a ranked poset (which is not guaranteed; see Example~\ref{ex:notEulerian} below). 
	For example, if $B$ and $B^w$ are Eulerian then  $\rho_{B^w}(B^w) = \rank(B^w) \le \rank(B) =  r_B^w(B^w)$ but the inequality is strict in general.
\end{remark}

For any subgroup $W'$ of $W$ containing $w$, recall that evaluation of corresponding class functions at $w$ induces a ring homomorphism $\ev_w: R(W')[t] \to \C[t]$. These are compatible with restriction  in the sense that for any $W' \subset W'' \subset W$ and $p \in R(W'')[t]$, we have  $\ev_w(\Res_{W'}^{W''} p ) = \ev_w(p)$. We have an induced map
\[
\ev_w : I^W(B) \to I(B^w)_\C,
\]
\[
\ev_w(p)(z,z') = \ev_w(p(z,z')) \textrm{ for any } z \le z' \in B^w. 
\]
Above, observe that $p(z,z') \in R(W_{z,z'})[t]$, and $z,z' \in B^w$ implies that $w \in W_{z,z'}$, so $\ev_w(p(z,z'))  \in \C[t]$ is well-defined. 
Also, observe that $\ev_w(\II^W(B)) \subset \II(B^w)_\C$, 
$\ev_w(\II_{1/2}^W(B)) \subset \II_{1/2}(B^w)_\C$, 
and 
$\ev_w(U^W(B)) \subset U(B^w)_\C$.
Moreover, $\ev_w$ commutes with the involution $p \mapsto p^{\rev}$, i.e.,
$\ev_w(p^{\rev}) = \ev_w(p)^{\rev}$ for all $p \in \II^W(B)$. 
 For example, $\ev_{\id}: I^W(B) \to I(B)_\C$ corresponds to pullback to the trivial subgroup in Remark~\ref{rem:pullback}. 
 We have the following key lemma.

\begin{lemma}\label{lem:Ztalgebrahom}
	Let $W$ be a finite group acting on a poset $B$ with $W$-invariant weak rank function $r_B$. Fix an element $w \in W$, and consider $B^w$ with the restricted weak rank function $r_B^w$. 
	Then $\ev_w : I^W(B) \to I(B^w)_\C$ is a $\Z[t]$-algebra homomorphism. 
\end{lemma}
\begin{proof}
	The fact that $\ev_w : I^W(B) \to I(B^w)_\C$ is a $\Z[t]$-module homomorphism follows since  $\ev_w: R(W')[t] \to \C[t]$ is a ring homomorphism for any subgroup $W'$ of $W$. We need to show that multiplication is preserved. Using \eqref{eq:equivariantmultiplication}, we compute for $p,p' \in I^W(B)$ and $z \le z' \in B^w$,
	\begin{align*}
	(\ev_w(p \cdot p'))(z,z') &= \ev_w((p \cdot p')(z,z')) \\
	&= \ev_w \left( \sum_{z \le z'' \le z'} \frac{|W_{z,z'',z'}|}{|W_{z,z'}|}
	\Ind^{W_{z,z'}}_{W_{z,z'',z'}} \left( \Res^{W_{z,z''}}_{W_{z,z'',z'}} p(z,z'') \cdot  \Res^{W_{z'',z'}}_{W_{z,z'',z'}} p'(z'',z')  \right) \right) \\
	&= \sum_{z \le z'' \le z'} \frac{|W_{z,z'',z'}|}{|W_{z,z'}|}
	\ev_w \left( \Ind^{W_{z,z'}}_{W_{z,z'',z'}} \left( \Res^{W_{z,z''}}_{W_{z,z'',z'}} p(z,z'') \cdot  \Res^{W_{z'',z'}}_{W_{z,z'',z'}} p'(z'',z')  \right) \right). 
	\end{align*}
	For any $z'' \in [z,z']$, let 
	\[
	p_{z''} =  \Res^{W_{z,z''}}_{W_{z,z'',z'}} p(z,z'') \cdot  \Res^{W_{z'',z'}}_{W_{z,z'',z'}} p'(z'',z') \in R(W_{z, z'',z})[t]. 
	\]
%	Consider an element 
	Observe that $w \in W_{z,z'}$. 
	Since applying the inclusion $W_{z,z'',z'} \hookrightarrow W_{z,z''}$ followed by conjugation by $w$ is the same as conjugating by $w$ and then applying the inclusion $W_{z,w \cdot z'',z'} \hookrightarrow W_{z,w \cdot z''}$, we have 
	\[
	\Res_{W_{z,z'',z'}}^{W_{z,z''}} \circ \Lambda_{\eta^w_{z,z''}} = \Lambda_{\eta^w_{z,z'',z'}}	\circ \Res_{W_{z,w \cdot z'',z'}}^{W_{z,w \cdot z''}}.  
	\]
	Similarly, applying the inclusion $W_{z,z'',z'} \hookrightarrow W_{z'',z'}$ followed by conjugation by $w$ is the same as conjugating by $w$ and then applying the inclusion $W_{z,w \cdot z'',z'} \hookrightarrow W_{w \cdot z'',z'}$, and we have 
	\[
		\Res_{W_{z,z'',z'}}^{W_{z'',z'}} \circ \Lambda_{\eta^w_{z'',z'}} = \Lambda_{\eta^w_{z,z'',z'}}	\circ \Res_{W_{z,w \cdot z'',z'}}^{W_{w \cdot z'',z'}}.  
	\]
	Using the fact that $\Lambda_{\eta^w_{z,z'',z'}}$ is a ring homomorphism as well as condition \eqref{i:equivincidence} in Definition~\ref{def:equivariantincidencealgebra}, we compute
	\begin{align*}
			\Lambda_{\eta^w_{z,z'',z'}}(p_{w \cdot z''}) &= \Lambda_{\eta^w_{z,z'',z'}} \left( \Res^{W_{z,w \cdot z''}}_{W_{z,w \cdot z'',z'}} p(z,w \cdot z'') \cdot  \Res^{W_{w \cdot z'',z'}}_{W_{z,w \cdot z'',z'}} p'(w \cdot z'',z')  \right) \\
			&= \Lambda_{\eta^w_{z,z'',z'}} \left( \Res^{W_{z,w \cdot z''}}_{W_{z,w \cdot z'',z'}} p(z,w \cdot z'') \right) \cdot \Lambda_{\eta^w_{z,z'',z'}} \left( \Res^{W_{w \cdot z'',z'}}_{W_{z,w \cdot z'',z'}} p'(w \cdot z'',z')  \right) \\
			&= 	\Res_{W_{z,z'',z'}}^{W_{z,z''}} \left( \Lambda_{\eta^w_{z,z''}} p(z,w \cdot z'') \right) \cdot 
				\Res_{W_{z,z'',z'}}^{W_{z'',z'}} \left( 
				\Lambda_{\eta^w_{z'',z'}} p'(w \cdot z'',z') \right) 
				\\
						&= 	 \Res_{W_{z,z'',z'}}^{W_{z,z''}}   p(z,z'')  \cdot 
			 \Res_{W_{z,z'',z'}}^{W_{z'',z'}}  
			p'(z'',z') \\
			&= p_{z''}.
	\end{align*}
	Applying \eqref{eq:inducedcharacter} to the action of $W_{z,z'}$ on $X = [z,z']$ with $x_0 = z''$ gives
	\[
		\ev_w(\Ind^{W_{z,z'}}_{W_{z,z'',z'}} p_{z''}) =  	\sum_{x \in (W_{z,z'} \cdot z'')^w} \ev_w(p_x).
	\]
	Putting this together, we compute
		\begin{align*}
		(\ev_w(p \cdot p'))(z,z') 
		&= \sum_{z \le z'' \le z'} \frac{|W_{z,z'',z'}|}{|W_{z,z'}|}
		\ev_w ( \Ind^{W_{z,z'}}_{W_{z,z'',z'}} 	p_{z''} ) \\
		&= \sum_{z \le z'' \le z'} \frac{|W_{z,z'',z'}|}{|W_{z,z'}|}
 \sum_{x \in (W_{z,z'} \cdot z'')^w} \ev_w(p_x) \\
 &= \sum_{z \le z'' \le z'} \frac{|W_{z,z'',z'}|}{|W_{z,z'}|}
 \sum_{x \in (W_{z,z'} \cdot z'')^w} \ev_w(p)(z,x) \ev_w(p)(x,z') \\
 &= \sum_{x \in  [z,z']^w} \ev_w(p)(z,x) \ev_w(p)(x,z')  \sum_{z'' \in W_{z,z'} \cdot x} \frac{|W_{z,z'',z'}|}{|W_{z,z'}|} \\
  &= \sum_{x \in  [z,z']^w} \ev_w(p)(z,x) \ev_w(p)(x,z')  \\
  &= (\ev_w(p) \cdot \ev_w(p'))(z,z'). 
	\end{align*}	
\end{proof}

With Lemma~\ref{lem:Ztalgebrahom} in hand, we have the following consequences which will allow us to relate the equivariant and non-equivariant settings. 

\begin{lemma}\label{lem:reducenonequivariant}
	Let $W$ be a finite group acting on a poset $B$ 
	with $W$-invariant weak rank function $r_B$. 
	For each $w \in W$, 	let $r_B^w$ be the restriction of $r_B$ to $B^w$.
	Consider an element $\kappa_B \in \II^W(B) \cap U^W(B)$. Assume that $\ev_w(\kappa_B) \in I(B^w) \subset I(B^w)_\C$ for all $w \in W$.  Then $\kappa_B$ is an equivariant $B$-kernel if and only if $\ev_w(\kappa_B) \in \II(B^w)_\C$ is a $B^w$-kernel for all $w \in W$. 
	
	Moreover, assume that $\kappa_B$ is an equivariant $B$-kernel. Consider 
	$f_B, g_B \in \II_{1/2}^W(B)  \cap U^W(B)$ and $Z_B \in  \II_{1/2}^W(B)  \cap U^W(B)$. Then $f_B$ and $g_B$ are the right and left equivariant Kazhdan-Lusztig-Stanley functions respectively corresponding to $\kappa_B$ if and only if $\ev_w(f_B)$ and $\ev_w(g_B)$ are the right and left Kazhdan-Lusztig-Stanley functions respectively corresponding to $\ev_w(\kappa_B)$ for all $w \in W$. 
	If $Z_B$ is the equivariant $Z$-function corresponding to $\kappa_B$, then $\ev_w(Z_B)$ is the $Z$-function corresponding to $\ev_w(\kappa_B)$  for all $w \in W$. 	
\end{lemma}
\begin{proof}
	Recall from Section~\ref{ss:representationtheory} that there is an  isomorphism between $R(W')_\C[t]$ and $\cf(W')[t]$ for all finite groups $W'$. In particular, $p \in I^W(B)$ is determined by $\{ \ev_w(p) : w \in W \}$. The first statement now follows from Lemma~\ref{lem:Ztalgebrahom}. Explicitly, $\kappa_B \cdot \kappa_B^{\rev} = \delta^W_B$ if and only if 
 	$\ev_w(\kappa_B \cdot \kappa_B^{\rev}) = \ev_w(\kappa_B) \cdot \ev_w(\kappa_B^{\rev}) = \ev_w(\kappa_B) \cdot \ev_w(\kappa_B)^{\rev} = \ev_w(\delta^W_B) =  \delta_B$ for all $w \in W$. Also, $\ev_w(\kappa_B)  \in \II(B^w) \cap U(B^w)$ by assumption. 
 	
 	For the second statement, Lemma~\ref{lem:Ztalgebrahom} implies that
 	$f_B^{\rev} = \kappa_B \cdot f_B$ if and only if 
 	$\ev_w(f_B)^{\rev} = \ev_w(\kappa_B) \cdot \ev_w(f_B)$ for all $w \in W$, and
 	 $g_B^{\rev} = g_B \cdot \kappa_B$ if and only if 
 	$\ev_w(g_B)^{\rev} = \ev_w(g_B) \cdot \ev_w(\kappa_B)$ for all $w \in W$.
 	The statement now follows from Remark~\ref{rem:existenceofgoverk}. 
 	Explicitly, Remark~\ref{rem:existenceofgoverk} implies that $\ev_w(f_B)^{\rev} = \ev_w(\kappa_B) \cdot \ev_w(f_B)$ if and only if $\ev_w(f_B)$ lies in $I(B^w) \subset I(B^w)_\C$ and is the right Kazhdan-Lusztig-Stanley function. Similarly, $\ev_w(g_B)^{\rev} = \ev_w(g_B) \cdot \ev_w(\kappa_B)$ if and only if $\ev_w(g_B)$ lies in $I(B^w) \subset I(B^w)_\C$ and is the left Kazhdan-Lusztig-Stanley function.
 	
 	Finally, suppose that $Z_B$ is the equivariant $Z$-function  corresponding to $\kappa_B$. By Lemma~\ref{lem:Ztalgebrahom}, for any $w \in W$, $\ev_w(Z_B) = \ev_w( g_B \cdot \kappa_B \cdot f_B) = \ev_w(g_B) \cdot \ev(\kappa_B) \cdot \ev_w(f_B)$ is the $Z$-function corresponding to $\ev_w(\kappa_B)$. 
\end{proof}

For the remainder of the section, assume that $B$ is a lower Eulerian poset.  %with rank function $\rho_B: B \to \Z$. 
We introduce the following definition.
%Example~\ref{ex:notEulerian} provides an example of an action not satisfying the conditions below.

\begin{definition}\label{def:Eulerianmap}
	Consider a finite group $W$  acting on a lower Eulerian poset $B$. The action is \emph{Eulerian} if $B^w$ is a lower Eulerian poset for all $w \in W$.
\end{definition}

The example below shows that the action in Example~\ref{ex:fanequiv} is Eulerian. 
%For example, we will see that the action in Example~\ref{ex:fanequiv} is Eulerian (see Example~\ref{ex:convexsupportequiv} below).

\begin{example}\label{ex:fanactionisEulerian}
	
		Recall the setup of Example~\ref{ex:fanequiv}. 
	That is,  $\Sigma$ is a %full-dimensional 
	fan in a real vector space $V$, and $\psi: W \to \GL(V)$ is a real representation of a finite group $W$ such that $\Sigma$ is $W$-invariant. %Assume that the support $|\Sigma|$ of $\Sigma$ is convex. 
	We claim that the induced action of $W$ on $\face(\Sigma)$ is Eulerian. 
	
	Fix an element $w \in W$, and consider the fan $\Sigma^w = \{ C^w : C \in \Sigma \}$ in $V^w$. We claim that $\face(\Sigma^w) = \face(\Sigma)^w$. In particular, $\face(\Sigma)^w$ is lower Eulerian. 
	The claim follows by a standard averaging argument. Any $v \in |\Sigma|^w$ lies in the relative interior of a 
	unique cone $C$ in $\Sigma$, and $w \cdot v = v$ implies that $w \cdot C = C$.  Hence $\Sigma^w = \{ C^w : C \in \Sigma , w \cdot C = C \}$. 
	Conservely, 
	consider $C \in \Sigma$ such that $w \cdot C = C$.  Given an element $v$ in the relative interior of $C$, $\frac{1}{|W|} \sum_{w \in W} w \cdot v \in C^w$ lies in the relative interior of $C$. % \subset V^w$.

%	Let $\Sigma^w = \{ C^w : C \in \Sigma, w \cdot C = C \}$. By a similar averaging argument, $\Sigma^w$ is a %full-dimensional 
%	fan in $V^w$ with face poset $\face(\Sigma^w) = \face(\Sigma)^w$. Indeed, any $v \in |\Sigma|^w$ \alan{fix/check} lies in the relative interior of a unique cone $C$ in $\Sigma$, and $w \cdot v = v$ implies that $w \cdot C = C$. Conversely, given $C \in \Sigma$ such that $w \cdot C = C$ and an element $v$ in the relative interior of $C$, $\frac{1}{|W|} \sum_{w \in W} w \cdot v \in C^w \subset V^w$.

\end{example}

	Below is an example of an action that is not Eulerian. 

%
%The example below shows that the fixed point set of a poset may not be well-behaved. 

\begin{example}\label{ex:notEulerian}
	If $B$ is Eulerian, then the fixed poset $B^w$ is not necessarily a ranked poset. 
	For example, with the notation of Example~\ref{ex:B1}, consider the semisuspension $\tilde{\Sigma} B_2 = B_2 \cup   \{ \hat{z}, \hat{1}_{\tilde{\Sigma} B_2}\}$  of $B_2$. That is, $\tilde{\Sigma} B_2$ is obtained from $B_2$ by adjoining an element $\hat{z}$ that is greater than all elements in $B_2$ except $\hat{1}_{B_2}$, and then adjoining a maximal element $\hat{1}_{\tilde{\Sigma} B_2}$. 
%	For example, let $B'$ be the rank $3$ Eulerian poset with two 
%	For example, consider the semisuspension $\tilde{\Sigma} B_2 = B_2 \cup   \{ \hat{z}, \hat{1}_{\tilde{\Sigma} B_2}\}$  of $B_2$. 
	Let $B$ be obtained from two copies  of 
	$\tilde{\Sigma} B_2$ by identifying the unique minimal elements, and also identifying the two maximal elements. Then $B$ is Eulerian of rank $3$. See \cite{StanleySurveyEulerian}*{Figure~2}.  Let $W = \Z/2\Z = \langle w \rangle$. Let $W$ act trivially on the first copy of $\tilde{\Sigma} B_2$, and act on the second copy of $\tilde{\Sigma} B_2$ by fixing all elements except $w \cdot \hat{z} = \hat{1}_{B_2}$ and $w \cdot \hat{1}_{B_2} = \hat{z}$. Then $B^w = [\hat{0}_{B^w},\hat{1}_{B^w}]$ has maximal chains of length both $2$ and $3$, and hence $B^w$ is not a ranked poset.
\end{example}

Assume that the action of $W$ on the lower Eulerian poset $B$ is Eulerian. 
%Assume that $B^w$ is a ranked poset for all $w \in W$. 
For any $w \in W$, recall the ring involution $p \mapsto \widehat{p}$ of $\II(B^w)$ defined by $\widehat{p}(z, z') = (-1)^{\rho_{B^w}(z,z')} p(z, z')$ for $z \le z'$ in $B^w$, where $\rho_{B^w}$ is the natural weak rank function for $B^w$. 
%$\rho_{B^w}: B^w \to \Z$ is any choice of rank function for $B^w$. 
%Note that this is independent of the choice of $\rho_{B^w}$.  
To help with notation, let us temporarily write 
$\iota_w(p) = \hat{p}$ for $p \in \II(B^w)$. 
Recall that an element $p \in I^W(B)_\C$ determines and is determined by $\{ \ev_w(p) : w \in W \}$. 
We may define a corresponding involution 
$p \mapsto \widehat{p}$ on $\II^W(B)_\C$, where $\widehat{p}$  is determined by
$\ev_w(\widehat{p}) = \iota_w(\ev_w(p))$ for all $w \in W$. 
%setting
%$\ev_w(\widehat{p}) = \widehat{\ev_w(p)}$ for all $w \in W$. 

\begin{lemma}
	Consider an Eulerian action of a finite group $W$ on the lower Eulerian poset $B$. 
	%If $B^w$ is a ranked poset for all $w \in W$, 
	Then
	the involution $p \mapsto \widehat{p}$ on $\II^W(B)_\C$ is a ring involution. 
\end{lemma}
\begin{proof}
	This follows from Lemma~\ref{lem:Ztalgebrahom} and the fact that $\iota_w$ is a ring involution is for all $w \in W$. 
	%We need to check that multiplication is preserved by the involution.
	For example, we explicitly check that multiplication is preserved. For any $p,p' \in \II^W(B)_\C$ and $w \in W$,
%	This follows from Lemma~\ref{lem:Ztalgebrahom} and the fact that $\iota_w$ is a ring involution for all $w \in W$. 
%	%To see this, we introduce some temporarily notation. Let $\iota_w(\lambda) = \hat{\lambda}$ for $\lambda \in I(B^w)$ so that $\ev_w(\hat{p}) = \iota_w(\ev_w(p))$ for all $p \in \II^W(B)_\C$ and $w \in W$. 
%	Explicitly,
%	for all $p,p' \in \II^W(B)_\C$ and $w \in W$, 
	\begin{align*}
		\ev_w(\widehat{p \cdot p'}) &= \iota_w(\ev_w(p \cdot p')) = \iota_w(\ev_w(p) \cdot \ev_w(p')) \\ &= \iota_w(\ev_w(p)) \cdot \iota_w(\ev_w(p')) 
		= \ev_w(\widehat{p}) \cdot \ev_w(\widehat{p'}) = 
		\ev_w(\widehat{p} \cdot \widehat{p'}). 
	\end{align*}

\end{proof}

We say that $p \in \II^W(B)$ is \emph{rank alternating} if $p^{\rev} = \widehat{p}$. We say that  an element $p \in \II^W(B)$ is \emph{multiplicative} if
$\ev_w(p(z,z')) = \ev_w(p(z,z''))\ev_w(p(z'',z'))$ for all $w \in W$ and $z \le z'' \le z'$ in $B^w$.  

\begin{lemma}\label{lem:equivariantmultiplicativealternating}
	Consider an Eulerian action of a finite group $W$ on the lower Eulerian poset $B$. 
	%If $B^w$ is a ranked poset for all $w \in W$, and  
	Suppose that $\kappa_B \in \II^W(B) \cap U^W(B)$ is multiplicative and rank alternating.  Then 
	$\kappa_B$ is an equivariant $B$-kernel, and 
	 the corresponding equivariant Kazhdan-Lusztig-Stanley functions satisfy $\widehat{f_B} = g_B^{-1}$ and $\widehat{g_B} = f_B^{-1}$.
\end{lemma}
\begin{proof}
	It follows from the definitions that for any $w \in W$, $\ev_w(\kappa_B) \in \II(B^w) \cap U(B^w)$ is multiplicative and rank alternating. 
	By Lemma~\ref{lem:multaltiskernel}, $\ev_w(\kappa_B)$ is a $B^w$-kernel.
	By Lemma~\ref{lem:reducenonequivariant}, $\kappa_B$ is an equivariant $B$-kernel with equivariant Kazhdan-Lusztig-Stanley functions $f_B,g_B$.
	By Lemma~\ref{lem:reducenonequivariant}, $\ev_w(f_B),\ev_w(g_B)$  are the 
	Kazhdan-Lusztig-Stanley functions corresponding to $\ev_w(\kappa_B)$. 
	With the notation above, by Lemma~\ref{lem:inverse} and Lemma~\ref{lem:Ztalgebrahom}, 
	$\ev_w(\widehat{f_B}) = \iota_w(\ev_w(f_B)) = \ev_w(g_B)^{-1} = \ev_w(g_B^{-1})$ and $\ev_w(\widehat{g_B}) = \iota_w(\ev_w(g_B)) = \ev_w(f_B)^{-1} = \ev_w(f_B^{-1})$ for all $w \in W$. We deduce that 
	$\widehat{f_B} = g_B^{-1}$ and $\widehat{g_B} = f_B^{-1}$.
\end{proof}

For example, we will see that 	the equivariant kernel in Example~\ref{ex:fanequiv} is multiplicative and rank alternating (see Example~\ref{ex:fanequivv2} below).

\subsection{Equivariant Kazhdan-Lusztig-Stanley theory for subdivisions of lower Eulerian posets}\label{ss:equivariantKLSlowerEulerian}

In this section, we develop equivariant generalizations of the results  of Section~\ref{ss:statements}. 
As in Section~\ref{ss:statements}, we fix the following setup. 
Let $\sigma: X \to Y$ be a strong formal subdivision between lower Eulerian posets with rank functions $\rho_X$ and $\rho_Y$ respectively,
corresponding under Theorem~\ref{thm:introbijection} to a triple $(\Gamma, \rho_\Gamma, q)$, where $\Gamma$ is a lower Eulerian poset
with 
rank function $\rho_\Gamma$, and $q$ is a join-admissible element of $\Gamma$ with $q \neq \hat{0}_\Gamma$. 

%See Section~\ref{ss:backgroundposet} for details. 
%As in Section~\ref{sec:lowerEulerian}, we fix the following setup throughout. 
%Let $\sigma: X \to Y$ be a strong formal subdivision between lower Eulerian posets with rank functions $\rho_X$ and $\rho_Y$ respectively,
%corresponding under Theorem~\ref{thm:mainsimplified} to a triple $(\Gamma, \rho_\Gamma, q) \in \JoinIdealLW^\circ$ with $\Gamma = \Cyl(\sigma)$, $\rho_\Gamma$ determined by \eqref{eq:rhoCyl}, and $q = \hat{0}_Y$. 
%Recall that $X = \Gamma \smallsetminus \Gamma_{\ge q}$, $Y = \Gamma_{\ge q}$, and $\sigma(x) = x \vee q$ for all $x \in X$. \alan{Thinking of phasing out notation $\Gamma_{\ge q}$}

%Let $W$ be a finite group acting on the poset $\Gamma$.
%
% We say that the action is \emph{Eulerian} if 

Let $W$ be a finite group. We define an action of $W$ on $(\Gamma, \rho_\Gamma, q)$ to be an action of $W$ on the poset $\Gamma$ such that $\rho_\Gamma$ is $W$-invariant and $q \in \Gamma$ is fixed by all elements of $W$. 
%, and $\rho_\Gamma$ is $W$-invariant in the sense that $\rho_\Gamma(w \cdot z) = \rho_\Gamma(z)$ for all $z \in \Gamma$ and $w \in W$. 
In that case, $X$ and $Y$ are $W$-invariant, and $\sigma$ is $W$-equivariant in the sense that for all $x \in X$ and $w \in W$, 
$\sigma(w \cdot x) = w \cdot \sigma(x)$. In particular, for any $w \in W$,
$\sigma$ restricts to an order-preserving function between posets $\sigma^w: X^w \to Y^w$.
We say that the action of $W$ on $(\Gamma, \rho_\Gamma, q)$ is \emph{Eulerian} if the action of $W$ on $\Gamma$ is Eulerian as defined in Definition~\ref{def:Eulerianmap}, i.e.,  $\Gamma^w$ is a lower Eulerian poset for all $w \in W$.
% Observe that for any $w \in W$, the induced map $w \cdot  : \Gamma \to \Gamma$, $z \mapsto w \cdot z$ is a strong formal subdivision. Moreover, under \alan{add a refernce}
% 			\[
% \CYLcat \left(\begin{tikzcd} X  \arrow[r, "\sigma"] \arrow[d, "w \cdot"] & Y  \arrow[d, "w \cdot"] \\ X \arrow[r, "\sigma"] &  Y \end{tikzcd}\right) = (w \cdot  : \Gamma \to \Gamma).
% \]		
%We introduce the following definition. 
%
%\begin{definition}\label{def:Eulerianmap}
%	Consider a finite group $W$  acting on a lower Eulerian poset $B$. We say that the action is \emph{Eulerian} if $B^w$ is a lower Eulerian poset for all $w \in W$.
%\end{definition}
%
%See Example~\ref{ex:notEulerian} for an example of a finite group action on an Eulerian poset that is not an Eulerian action. 
%We say that an action of $W$ on $(\Gamma, \rho_\Gamma, q)$ is Eulerian if the action of $W$ on $\Gamma$ is Eulerian. 
%
%Recall from Definition~\ref{def:Eulerianmap} that .....

Fix an Eulerian action of $W$ on $(\Gamma, \rho_\Gamma, q)$. 
%Assume that the action of $W$ on $\Gamma$ is Eulerian as defined in Definition~\ref{def:Eulerianmap}. 
Since $X$ is a lower order ideal of $\Gamma$, $Y$ is an upper order ideal of $\Gamma$, and $q$ is $W$-fixed, it follows that the action of $W$ restricts to Eulerian actions on $X$ and $Y$. 
Consider any $w \in W$. Observe that $q$ is a join-admissible element of $\Gamma^w$. Indeed,  for any $z \in \Gamma^w$, we have $w \cdot (z \vee q) = (w \cdot z) \vee (w \cdot q) = z \vee q$, and hence $z \vee q \in \Gamma^w$.
  Let $\rho_{\Gamma^w}: \Gamma^w \to \Z$ be a choice of rank function for $\Gamma^w$. 
  % uniquely determined by $\rho_{\Gamma^w}(q) = \rho_\Gamma(q)$.  
  Then $\rho_{X^w} = \rho_{\Gamma^w}|_{X^w}$ and $\rho_{Y^w} = \rho_{\Gamma^w}|_{Y^w}[-1]$ are rank functions for $X^w$ and $Y^w$ respectively. 
We deduce that the triple $(\Gamma^w, \rho_{\Gamma^w}, q)$ corresponds to  $\sigma^w : X^w \to Y^w$  under Theorem~\ref{thm:introbijection}. In particular, $\sigma^w : X^w \to Y^w$ is a strong formal subdivision. 

%Let $W$ be a finite group acting on the poset $\Gamma$ and assume that $q \in \Gamma$ is fixed by all elements of $W$. 
%Then $X$ and $Y$ are $W$-invariant, and $\sigma$ is $W$-equivariant in the sense that for all $x \in X$ and $w \in W$, 
%$\sigma(w \cdot x) = w \cdot \sigma(x)$. 
%We further assume that $\Gamma^w$ is a lower Eulerian poset for all $w \in W$. Example~\ref{ex:notEulerian} shows that this assumption does not hold in general.  Let $\rho_{\Gamma^w}: \Gamma^w \to \Z$ be the rank function  uniquely determined by $\rho_{\Gamma^w}(q) = \rho_\Gamma(q)$. 
%Observe that $q \in \Gamma^w$ is join-admissible since for any $z \in \Gamma^w$, we have $w \cdot (z \vee q) = (w \cdot z) \vee (w \cdot q) = z \vee q$, and hence $z \vee q \in \Gamma^w$. Then $(\Gamma^w, \rho_{\Gamma^w}, q) \in \JoinIdealLW^\circ$ and the corresponding strong formal subdivision is the restriction of $\sigma$ to a function $\sigma^w: X^w \to Y^w$, with rank functions $\rho_{X^w} = \rho_{\Gamma^w}|_{X^w}$ and $\rho_{Y^w} = \rho_{\Gamma^w}|_{Y^w}[-1]$ for $X^w$ and $Y^w$ respectively. 

Fix  a  $W$-invariant weak rank function $r_\Gamma$ for $\Gamma$. 
Then
$r_\Gamma$ restricts to $W$-invariant weak rank functions on $X$ and $Y$. We may then define $\II^W(X), \II^W(Y), \II^W(\Gamma)$ and related invariants as in Section~\ref{ss:backgroundequivariantKLS}.
For any $w \in W$, fix the weak rank function   $r_\Gamma^w$ for $\Gamma^w$ given by restriction of $r_\Gamma$. We may then consider $\II(\Gamma^w)$ and related invariants as in Section~\ref{ss:classfunctions}. 
The most important example is when $r_\Gamma$ is the natural weak rank function, which is $W$-invariant by  Example~\ref{ex:Winvariantrank}.  We warn that even in this case, one should not expect $r_\Gamma^w$ to be the natural weak rank function  (see Remark~\ref{rem:rankwarning}).

Fix an element $\kappa_\Gamma \in \II^W(\Gamma) \cap U^W(\Gamma)$ that is multiplicative and rank alternating. By Lemma~\ref{lem:equivariantmultiplicativealternating}, 
	$\kappa_\Gamma$ is an equivariant $\Gamma$-kernel. Moreover, the restrictions $\kappa_X \in \II^W(X)$ and $\kappa_Y \in \II^W(Y)$ of $\kappa_\Gamma$ to $X$ and $Y$ respectively are multiplicative and rank alternating, and are
	an equivariant $X$-kernel  and an equivariant $Y$-kernel respectively.
  Then $f_\Gamma,g_\Gamma, Z_\Gamma \in \II^W(\Gamma)$ restrict to $f_X,g_X, Z_X \in \II^W(X)$ respectively, and restrict to $f_Y,g_Y, Z_Y \in \II^W(Y)$ respectively.

As in Section~\ref{sec:lowerEulerian}, we will use the following notation. 
Given $p \in \II^W(\Gamma)$ and a $W$-invariant subset $S \subset \Int(\Gamma)$, define 
$p|_S \in \II^W(\Gamma)$ by 
\[
p|_{S}(z,z') = \begin{cases}
	p(z,z') &\textrm{ if } [z,z'] \in S, \\
	0 &\textrm{ otherwise.}
\end{cases}
\]
It follows from Definition~\ref{def:equivariantincidencealgebra} that $p|_S$ is a well-defined element of $\II^W(\Gamma)$. Observe that with the notation of Section~\ref{sec:lowerEulerian}, $\ev_w(p|_S) = \ev_w(p)|_S$. 
%We use the following shorthand notations
For convenience, we also set
\[
p|_X := p|_{\{ [z,z'] \in \Int(\Gamma) : z,z' \in X \} },
\]
\[
p|_Y := p|_{\{ [z,z'] \in \Int(\Gamma) : z,z' \in Y \} },
\]
\[
p|_{X/Y} := p|_{\{ [z,z'] \in \Int(\Gamma) : z \in X, z' \in Y \} },
\]
\[
p|_{(X/Y)^\circ} := p|_{\{ [z,z'] \in \Int(\Gamma) : z \in X, z' = \sigma(z) \in Y\} }.
\]

We have the following equivariant analogue of Definition~\ref{def:hellpolynomial}. 

\begin{definition}\label{def:equivarianthellpolynomial}
	With the notation above, define elements $h_{\sigma},\ell_{\sigma} \in \II^W(\Gamma)$ by
	\[
	(t - 1) \cdot h_\sigma = g_\Gamma \cdot \kappa_\Gamma|_{(X/Y)^\circ} =  g_\Gamma|_X \cdot \kappa_\Gamma|_{(X/Y)^\circ}, 
	\]
	\[
	\ell_\sigma = h_\sigma \cdot g_\Gamma^{-1}. 
	\]
	For any   $x \in X$ and $y \in Y$ such that $\sigma(x) \le y$, 
	the polynomials $h_{\sigma}(x,y)$ and $\ell_{\sigma}(x,y)$ in $R(W_{x,y})[t]$ are called the 
	\emph{equivariant $h$-polynomial} and \emph{equivariant local $h$-polynomial} respectively associated to the interval $[x,y]$  in $\Gamma$. 	When $\Gamma$ is Eulerian, we call  $h_{\sigma}(\Gamma) = h_{\sigma}(\hat{0}_X,\hat{1}_Y)$ and $\ell_{\sigma}(\Gamma) = \ell_{\sigma}(\hat{0}_X,\hat{1}_Y)$ the 
	equivariant $h$-polynomial and equivariant local $h$-polynomial respectively associated to $\sigma$. 
\end{definition}

Equivalently, using Lemma~\ref{lem:Ztalgebrahom} and Lemma~\ref{lem:reducenonequivariant}, and comparing with Definition~\ref{def:hellpolynomial}, $h_\sigma$ and $\ell_\sigma$ are determined by the condition that $\ev_w(h_\sigma) = h_{\sigma^w} \in \II(B^w)$ and  $\ev_w(\ell_\sigma) = \ell_{\sigma^w} \in \II(B^w)$ for all $w \in W$.  In particular, $h_\sigma$ is a well-defined element of $\II^W(B)_\C$. We need to show that $h_\sigma$ is a well-defined element of $\II^W(B)$. Then $\ell_\sigma = h_\sigma \cdot g_\Gamma^{-1}$ is also a well-defined element of $\II^W(B)$. For any $x \in X$ and $y \in Y$, since 
$t - 1$ is invertible in $R(W_{x,y})[[t]]$ and $(t - 1)h_\sigma(x,y) \in R(W_{x,y})[t]$ by Definition~\ref{def:equivarianthellpolynomial}, we have
$h_\sigma(x,y) \in R(W_{x,y})[[t]] \cap R(W_{x,y})_\C[t] = R(W_{x,y})[t]$, as desired.

With this setup, we can now deduce equivariant analogues of the results of Section~\ref{ss:statements}.   

\begin{corollary}\label{cor:equivariantsymmetry}
	The equivariant local $h$-polynomials are symmetric in the sense that for all $x \in X$ and $y \in Y$ such that $\sigma(x) \le y$,
\[
\ell_{\sigma}(x,y;t) = t^{r_\Gamma(x,y) - 1}\ell_{\sigma}(x,y;t^{-1}).
\]
Equivalently,  $(t - 1) \cdot \ell_{\sigma} \in \II^W(\Gamma)$ is antisymmetric.
\end{corollary}
\begin{proof}
	This follows from Proposition~\ref{prop:symmetry} using Lemma~\ref{lem:Ztalgebrahom} and the fact that for all $w \in W$, $\ev_w$ commutes with the involution $p \mapsto p^{\rev}$ and 
	$\ev_w(\ell_{\sigma}) = \ell_{\sigma^w}$. 
\end{proof}

 It follows from Corollary~\ref{cor:equivariantsymmetry} and \eqref{eq:tildeell} that 
$\Delta \ell_{\sigma}  \in \II_{1/2}^W(\Gamma)$ is an alternative encoding of $\ell_{\sigma}$. Observe that $\ev_w(\Delta \ell_{\sigma}) = \Delta \ell_{\sigma^w}$ for all $w \in W$.

\begin{example}\label{ex:lowdegreehellequiv}
	We have the following equivariant generalization of Example~\ref{ex:lowdegreehell}. For any   $x \in X$ and $y \in Y$ such that $\sigma(x) \le y$, 
		\[
	h_{\sigma}(x,y)_0 = \kappa_\Gamma(x,y)_{r_\Gamma(x,y)},
	\]
	\[
	h_{\sigma}(x,y)_{r_\Gamma(x,y) - 1} = \ell_{\sigma}(x,y)_{r_\Gamma(x,y) - 1} = \ell_{\sigma}(x,y)_0 = (\Delta \ell_{\sigma}(x,y))_0  =  \begin{cases}
		\kappa_\Gamma(x,y)_{r_\Gamma(x,y)} &\textrm{if } \sigma(x) = y, \\
		0    &\textrm{otherwise.} 
	\end{cases}
	\]
	In fact, these equations follow from Example~\ref{ex:lowdegreehell}, by applying $\ev_w$ to both sides for all $w \in W_{x,y}$, and using Lemma~\ref{lem:reducenonequivariant}, together with the fact that 
	$\ev_w(h_\sigma) = h_{\sigma^w}$ and  $\ev_w(\ell_\sigma) = \ell_{\sigma^w}$.
\end{example}

We deduce the following equivariant generalization of the main results of Section~\ref{ss:statements}. % Theorem~\ref{thm:maingell}. 
In fact, using the formalism of Section~\ref{ss:classfunctions}, we deduce this result as a corollary of the main results of Section~\ref{ss:statements}.

\begin{theorem}\label{thm:equivariantmaingell}
			Let  $\sigma: X \to Y$ be a strong formal subdivision
	between lower Eulerian posets $X$ and $Y$ with rank functions $\rho_X$ and $\rho_Y$ respectively, corresponding to a triple $(\Gamma,\rho_\Gamma,q)$  under Theorem~\ref{thm:introbijection}. 
	Consider an Eulerian action of a finite group $W$ on $(\Gamma,\rho_\Gamma,q)$. 
	Fix a $W$-invariant weak rank function $r_\Gamma \in I(\Gamma)$ and a multiplicative and rank alternating element $\kappa_\Gamma \in \II^W(\Gamma) \cap U^W(\Gamma)$. 
	Then
	\[
	g_\Gamma|_{X/Y} = \Delta \ell_{\sigma} \cdot  g_\Gamma =  \Delta \ell_{\sigma} \cdot  g_\Gamma|_Y,
	\]
%	That is, for any $x \in X$ and $y \in Y$ such that $\sigma(x) \le y$, 
%	\[
%	g_\Gamma(x,y) = \sum_{ \sigma(x) \le y' \le y} \Delta \ell_{\sigma}(x,y') g_Y(y',y). 
%	\] 
	\[
f_\Gamma|_{X/Y} = - f_\Gamma \cdot \Delta \widehat{\ell_{\sigma}} =  - f_\Gamma|_X \cdot \Delta \widehat{\ell_{\sigma}},
\]
	\begin{align*}
	Z_\Gamma|_{X/Y} = - Z_\Gamma|_X \cdot \Delta \widehat{\ell_{\sigma}} + (\Delta \ell_{\sigma})^{\rev} \cdot Z_\Gamma|_Y. 
\end{align*}
\end{theorem}
\begin{proof}
	Consider the first equation. 
	For all $w \in W$, we need to show that $\ev_w(g_\Gamma|_{X/Y})   = \ev_w(\Delta \ell_{\sigma} \cdot  g_\Gamma) = \ev_w(\Delta \ell_{\sigma} \cdot  g_\Gamma|_Y)$. This follows from Theorem~\ref{thm:maingell}, Lemma~\ref{lem:Ztalgebrahom}, Lemma~\ref{lem:reducenonequivariant}, the fact that $\ev_w(\Delta \ell_{\sigma}) = \Delta \ell_{\sigma^w}$, and the fact that $\ev_w(p|_S) = \ev_w(p)|_S$ for all $W$-invariant sets $S$ and $p \in \II^W(\Gamma)$.	
%	it follows from Theorem~\ref{thm:maingell}, Lemma~\ref{lem:Ztalgebrahom}, Lemma~\ref{lem:reducenonequivariant}, the fact that $\ev_w(\Delta \ell_{\sigma}) = \Delta \ell_{\sigma^w}$, and the fact that $\ev_w(p|_S) = \ev_w(p)|_S$ for all $W$-invariant sets $S$ and $p \in I^W(\Gamma)$, that This establishes the first equation. 
	The second equation follows in the same way using Corollary~\ref{cor:rightKLSfunction} and the fact that $\ev_w$ commutes with the involution $p \mapsto \widehat{p}$  by construction. 
	Similarly, the final equation follows from 
	Corollary~\ref{cor:Zmappingformula} using the fact that $\ev_w$ commutes with the involution $p \mapsto p^{\rev}$. 
\end{proof}

\begin{remark}\label{rem:altdeltaellequiv}
	%Consider the setup of Theorem~\ref{thm:equivariantmaingell}. 
		As in Remark~\ref{rem:altdeltaell}, by Lemma~\ref{lem:equivariantmultiplicativealternating} and Theorem~\ref{thm:equivariantmaingell}, we have  $\Delta \ell_{\sigma} = g_\Gamma|_{X/Y} \cdot g_\Gamma^{-1}  = g_\Gamma|_{X/Y} \cdot \widehat{f_\Gamma}$. 
\end{remark}

\begin{example}
	We have the following equivariant generalization of Example~\ref{ex:sigmaxequalsyv3}. 
	For any $x \in X$, it follows from Definition~\ref{def:equivarianthellpolynomial} that $h_{\sigma}(x,\sigma(x)) = \ell_{\sigma}(x,\sigma(x)) \in R(W_{x,y})$. 
	By Theorem~\ref{thm:equivariantmaingell},  $g_\Gamma(x,\sigma(x)) = \Delta \ell_{\sigma}(x,\sigma(x))  \in R(W_{x,y})$. 		
\end{example}

We end the section with some examples where there is a natural choice of equivariant kernel.

\begin{example}\label{ex:propermapequiv}
Consider the setup of Example~\ref{ex:introproperequivariant}. 	
That is, consider a linear map $\phi: V' \to V$ inducing a proper, surjective morphism between %full-dimensional 
fans $\Sigma'$ and $\Sigma$ in real vector spaces $V'$ and $V$ respectively, with induced strong formal subdivision  
$\sigma: X = \face(\Sigma') \to Y = \face(\Sigma)$. Let  $(\Gamma,\rho_\Gamma,q)$  be the corresponding triple under Theorem~\ref{thm:introbijection}. 
By \cite{StapledonLWPosets}*{Example~7.26},  $\Gamma$ is a $CW$-poset, and, if  the morphism of fans is projective, then $\Gamma$ is the face poset of a fan. If the morphism of fans is projective and $\Sigma$ is a pointed cone, then $\Gamma$ is the face lattice of a polytope (see Example~\ref{ex:polytope}). 
Let  $\psi': W \to \GL(V')$ and $\psi: W \to \GL(V)$ be real representations of $W$ such that we have an induced action of $W$ on 
$\Sigma'$ and $\Sigma$, and 
$\phi: V' \to V$ is $W$-equivariant. Then we have a corresponding action of $W$ on $(\Gamma,\rho_\Gamma,q)$.  We claim that this action is Eulerian.

  For any $w \in W$, we have a corresponding linear map 
$\phi^w: (V')^w \to V^w$. Observe that $\phi^w$ is surjective by a standard averaging argument. Explicitly, consider any $v \in V^w$.  Since $\phi$ is surjective, there exists $v' \in V'$ such that $\phi(v') = v$. Then $\frac{1}{|W|} \sum_{w \in W} w \cdot v' \in (V')^w$ maps to $v$ under $\phi$. 
By Example~\ref{ex:fanactionisEulerian}, we have fans 
$(\Sigma')^w = \{ (C')^w : C' \in \Sigma', w \cdot C' = C' \}$ and 
$\Sigma^w = \{ C^w : C \in \Sigma, w \cdot C = C \}$ with face posets $\face((\Sigma')^w) = \face(\Sigma')^w$ and $\face(\Sigma^w) = \face(\Sigma)^w$ respectively. 
%
%We have seen in Example~\ref{ex:fanactionisEulerian} that 
% $\Sigma^w$ is a %full-dimensional 
%fan in $V^w$ with face poset $\face(\Sigma^w) = \face(\Sigma)^w$.
%%By a similar averaging argument, $\Sigma^w$ is a %full-dimensional 
%%fan in $V^w$ with face poset $\face(\Sigma^w) = \face(\Sigma)^w$. Indeed, any $v \in |\Sigma|^w$ \alan{fix/check} lies in the relative interior of a unique cone $C$ in $\Sigma$, and $w \cdot v = v$ implies that $w \cdot C = C$. Conversely, given $C \in \Sigma$ such that $w \cdot C = C$ and an element $v$ in the relative interior of $C$, $\frac{1}{|W|} \sum_{w \in W} w \cdot v \in C^w \subset V^w$. 
It follows that $\phi^w$ induces a proper surjective map of fans from  $(\Sigma')^w$ to  $\Sigma^w$ with induced strong formal subdivision $\sigma^w:  \face(\Sigma')^w \to \face(\Sigma)^w$. In particular, $\Gamma^w = \Cyl(\sigma^w)$ is lower Eulerian, and the action of $W$ on $\Gamma$ is Eulerian, as desired.
 %in the sense of Definition~\ref{def:Eulerianmap}. 
 
% \alan{Do we need this?}
% Let $\rho_{\Gamma^w}$ be the natural rank function. 
%% and corresponding element  $(\Gamma^w,\rho_{\Gamma^w},q)$ in $\JoinIdealLW^\circ$. 
%Then $\rho_{\Gamma^w}(C') = \dim (C')^w$ for $C' \in \face(\Sigma')^w$, and 
% $\rho_{\Gamma^w}(C) = \dim (\phi^w)^{-1}(C^w) + 1 = \dim (\phi^{-1}(C))^w + 1$ for $C \in \face(\Sigma)^w$.

%For a (not necessarily pointed) cone $C'$ in $V'$, let $V_{C'}'$ be the linear span of $C'$ in $V'$. 

Given $z \in \Gamma$, define a (not necessarily pointed) cone $C_z'$ in $V'$ by 
\[
C_z' = \begin{cases}
	C' &\textrm{ if } z  = C' \in X = \face(\Sigma'), \\
	\phi^{-1}(C) &\textrm{ if } z  = C \in Y = \face(\Sigma). \\
\end{cases}
\]
Let $V_z'$ denote the linear span of $C_z'$ in $V'$. Observe that $C_z'$ and $V_z'$ are $W_z$-invariant. Also, if $w \in W_z$, then the linear span of $(C_z')^w$ in $(V')^w$ equals $(V_z')^w$. This follows by the same averaging argument as above;  given any $v' \in (V_z')^w$, we may write
$v' = \sum_i a_i v_i'$ for some $a_i \in \R$ and $v_i' \in C_z'$,
and then $v' =  \frac{1}{|W|}  \sum_{w \in W} w \cdot v' =  \sum_i a_i \frac{1}{|W|} \sum_{w \in W} w \cdot v_i'$ lies in the linear span of $(C_z')^w$. 
For $z \le z'$, observe that $C_z' \subset C_{z'}'$ and $V_z' \subset V_{z'}'$. 
Then $\psi$ induces a representation $\psi_{z,z'}: W_{z,z'} \to \GL(V_{z'}'/V_z')$. Observe that  for $z \le z' \in \Gamma$, the natural weak rank function for $\Gamma$ is given by 
\begin{equation}\label{eq:rhoGammanotw}
\rho_\Gamma(z,z') = \begin{cases}
	\dim (V_{z'}'/V_z')  + 1 &\textrm{ if } z \in X, z' \in Y, \\
	\dim (V_{z'}'/V_z')  &\textrm{ otherwise. }
\end{cases}
\end{equation}
Also, for $z \le z' \in \Gamma^w$, the natural weak rank function for $\Gamma^w$ is given by 
\begin{equation}\label{eq:rhoGammaw}
	\rho_{\Gamma^w}(z,z') = \begin{cases}
		\dim ((V_{z'}'/V_z')^w)  + 1 &\textrm{ if } z \in X, z' \in Y, \\
		\dim ((V_{z'}'/V_z')^w)  &\textrm{ otherwise. }
	\end{cases}
\end{equation}

Fix the $W$-invariant natural weak rank function $r_\Gamma$ for $\Gamma$. %This is $W$-invariant by Example~\ref{ex:Winvariantrank}. 
%Fix the $W$-invariant weak rank function $r_\Gamma = \rho_\Gamma$. 
Define an element $\kappa_\Gamma \in \II^W(\Gamma) \cap U^W(\Gamma)$ as follows. With the notation of Section~\ref{ss:representationtheory}, for any $z \le z' \in \Gamma$, define
\begin{equation}\label{eq:kappagammaequiv}
\kappa_\Gamma(z,z') = \begin{cases}
	(t - 1)\det(tI - \psi_{z,z'}) &\textrm{ if } z \in X, z' \in Y, \\
	\det(tI - \psi_{z,z'}) &\textrm{ otherwise. }
\end{cases}	
\end{equation}
One may verify that $\kappa_\Gamma$ satisfies the properties of Definition~\ref{def:equivariantincidencealgebra} and hence is a well-defined element of $\II^W(\Gamma)$. 
By Lemma~\ref{lem:kappamultalt} below, $\kappa_\Gamma$ is multiplicative and rank alternating. By Lemma~\ref{lem:equivariantmultiplicativealternating}, 
$\kappa_\Gamma$ is an equivariant $\Gamma$-kernel. 

Consider the restrictions $\kappa_{\face(\Sigma')}$ and $\kappa_{\face(\Sigma)}$ of $\kappa_\Gamma$ to $\face(\Sigma')$ and $\face(\Sigma)$ respectively. Observe that these are equivariant kernels that are
multiplicative and rank alternating. 
%$\kappa_\Gamma$ restricts to a multiplicative and rank alternating  equivariant $\face(\Sigma)$-kernel $\kappa_{\face(\Sigma)}$ that is independent of $\phi$. 
%See Example~\ref{ex:fanequiv}.  
%For example, 
%	if $\Sigma$ is a %full-dimensional 
%fan in a real vector space $V$, and $\psi: W \to \GL(V)$ is a real representation of a finite group $W$ such that $\Sigma$ is $W$-invariant, then the identity map on $V$ induces a morphism of fans from $\Sigma$ to itself. 
\end{example}

\begin{example}\label{ex:fanequivv2}
	Recall the setup of Example~\ref{ex:fanequiv}. 
	That is,  $\Sigma$ is a %full-dimensional 
	fan in a real vector space $V$, and $\psi: W \to \GL(V)$ is a real representation of a finite group $W$ such that $\Sigma$ is $W$-invariant. %Assume that the support $|\Sigma|$ of $\Sigma$ is convex. 
	Consider the lower Eulerian poset $B = \face(\Sigma)$ with 
	the natural weak rank function, and $\kappa_B$ defined by 
	\[
	\kappa_B(z,z') = 
	\det(tI - \psi_{z,z'}),
	\]
	for any $z \le z' \in B$. 
	Observe that the identity map on $V$ induces a morphism of fans from $\Sigma$ to itself. 
	By Example~\ref{ex:propermapequiv} applied to this morphism, $\kappa_B$ is a multiplicative and rank alternating equivariant $B$-kernel.
	
%	For example, if $P$ is a full-dimensional $W$-invariant polytope in a vector space $V$ \alan{TODO: is this already covered in the intro??}

\end{example}

%\begin{example}\label{ex:convexsupportequiv}
%	We have the following special case of Example~\ref{ex:propermapequiv}.  
%	Suppose that $\Sigma$ is a %full-dimensional 
%	fan in a real vector space $V$, and $\psi: W \to \GL(V)$ is a real representation of a finite group $W$ such that $\Sigma$ is $W$-invariant. Assume that the support $|\Sigma|$ of $\Sigma$ is convex. Let $L$ be the largest linear subspace of $V$ contained in $|\Sigma|$, let $\phi: V \to V/L$ be the projection map, and  let $C = \phi(|\Sigma|)$. Then $C$ is $W$-invariant, $\phi$ is $W$-equivariant, and we have an induced proper, surjective morphism of fans from $\Sigma$ to $C$. For example, if $\Sigma$ is complete, i.e., $|\Sigma| = V$, then $C = \{ 0 \}$ (c.f. Example~\ref{ex:B0}).
%\end{example}

\begin{example}\label{ex:polytopeequiv}
	Let $P$ be a full-dimensional polytope $P$ in a real vector space $V$ that contains the origin. 	Suppose that $\psi: W \to \GL(V)$ is a representation of a finite group $W$ such that $P$ is $W$-invariant. 
	Consider $\face(P)$ with the natural weak rank function. 
%	We explain that there is a natural choice of equivariant $\face(P)$-kernel. 
	
	On the one hand, recall that we may consider the cone $C(P \times \{1 \})$ in $V \oplus \R$ with $\face(C(P \times \{1 \})) = \face(P)$. We have an induced  representation
	$\widetilde{\psi}: W \to \GL(V \oplus \R)$, where $W$ acts trivially on the last coordinate of $V \oplus \R$. Then $C(P \times \{1 \})$ is $W$-invariant, and we have a natural choice of equivariant $\face(P)$-kernel given by Example~\ref{ex:fanequivv2}. 
	
	On the other hand, let $F$ be the unique face of $P$ containing the origin in its relative interior. Then $F$ is $W$-invariant. Let $\rho_{\face(P)}$ be the natural rank function for $\face(P)$.  Recall from Example~\ref{ex:polytope} that there is a projective, surjective morphism of 
	fans from a fan $\Sigma$ to a pointed cone $C$ such that the induced strong formal subdivision $\sigma : \face(\Sigma) \to \face(C)$ corresponds to
	the triple 
	$(\face(P),\rho_{\face(P)},F)$ under Theorem~\ref{thm:introbijection}.
	One may apply Example~\ref{ex:propermapequiv} to this projective, surjective morphism and \eqref{eq:kappagammaequiv} gives a natural choice of equivariant $\face(P)$-kernel. 
	
	Expanding out the definitions, these two choices of  equivariant $\face(P)$-kernel agree.

\end{example}

\begin{lemma}\label{lem:kappamultalt}
	With the notation of Example~\ref{ex:propermapequiv}, $\kappa_\Gamma$ is multiplicative and rank alternating.
\end{lemma}
\begin{proof}
	Consider $w \in W$ and $z \le z'' \le z'$ in $B^w$.
	Since $V_{z''}'/V_z'$ is isomorphic as a $W_{z,z'',z'}$-representation to the direct sum of $V_{z''}'/V_{z'}'$ and $V_{z'}'/V_z'$, we have 
	$\ev_w(\det(tI - \psi_{z,z''})) = \det(tI - \psi_{z,z''}(w)) = \det(tI - \psi_{z,z'}(w)) \det(tI - \psi_{z',z''}(w)) = \ev_w(\det(tI - \psi_{z,z'}))\ev_w(\det(tI - \psi_{z',z''})).$ It follows that $\kappa_\Gamma$ is multiplicative. 
	
	To show that $\kappa_\Gamma$ is rank alternating, we need to show that
	for any $w \in W$ and $z \le z'$ in $B^w$,
	\[
	\ev_w(\kappa_\Gamma(z,z')^{\rev}) = (-1)^{\rho_{\Gamma^w}(z,z')} \ev_w(\kappa_\Gamma(z, z')). 
	\]
	By Lemma~\ref{lem:detswitcht}, 
	\[
	(-1)^{\dim (V_{z'}'/V_{z}')^w}
	\det(tI - \psi_{z,z'}(w)) = t^{\dim (V_{z'}'/V_{z}')} \det(t^{-1}I - \psi_{z,z'}(w)). 
	\]  
	If $z \in X = \face(\Sigma'), z' \in Y = \face(\Sigma)$, we compute
	\begin{align*}
		\ev_w(\kappa_\Gamma(z,z')^{\rev}) &= t^{\rho_\Gamma(z,z')} (t^{-1} - 1)\det(t^{-1}I - \psi_{z,z'}(w)) \\
		&= -(-1)^{\dim (V_{z'}'/V_{z}')^w}
		(t - 1)\det(tI - \psi_{z,z'}(w)) \\
		&= (-1)^{\dim (V_{z'}'/V_{z}')^w + 1} \ev_w(\kappa_\Gamma(z, z')).
	\end{align*}
	Otherwise, we compute
		\begin{align*}
		\ev_w(\kappa_\Gamma(z,z')^{\rev}) &= t^{\rho_\Gamma(z,z')} \det(t^{-1}I - \psi_{z,z'}(w)) \\
		&= (-1)^{\dim (V_{z'}'/V_{z}')^w}
		\det(tI - \psi_{z,z'}(w)) \\
		&= (-1)^{\dim (V_{z'}'/V_{z}')^w} \ev_w(\kappa_\Gamma(z, z')).
	\end{align*}
		The result now follows from \eqref{eq:rhoGammaw}. 
%	It remains to show that 
%	\[
%	\rho_{\Gamma^w}(z,z') = \begin{cases}
%		\dim (V_{z'}/V_z)  + 1 &\textrm{ if } z \in X = \face(\Sigma'), z' \in Y = \face(\Sigma), \\
%		\dim (V_{z'}/V_z)  &\textrm{ otherwise. }
%	\end{cases}
%	\]

\end{proof}

%WORKING BELOW
%
%\begin{example}
%	We have the following equivariant analogue of Example~\ref{ex:productsformulas}. 
%\end{example}
%

%\begin{example}\label{ex:productsformulas}

%OLD BELOW
%
%
%
%
%
%
%
%\alan{TODO: give example of an equivariant unimodular triangulation of a lattice polytope (probably want to give the non-equivariant version earlier).} 
%
%
%\alan{expansion formulas for equivariant $h^*$-polynomial ??}
%
%\alan{revisit other examples in the equivariant case}

\subsection{Applications to equivariant Ehrhart theory}\label{ss:equivEhrhart}

In this final section, we relate invariants from equivariant Ehrhart theory to invariants from equivariant KLS theory. 
%These results are well-known in the non-equivariant case. 

We first recall some background on equivariant Ehrhart theory and refer the reader to \cite{StapledonEquivariant,StapledonEquivariantCommutativeTriangulations} for more details.
Let $N$ be a lattice and let  $\widetilde{N} = N \oplus \Z$.  Let $\pr: \widetilde{N} = N \oplus \Z \to \Z$ denote projection onto the last coordinate.
%Let $W$ be a finite group. 
	Suppose that $\psi: W \to \Aff(N)$ is an affine representation of a finite group $W$. That is, there is a  representation
$\widetilde{\psi}: W \to \GL(\widetilde{N})$ such that $W$ acts trivially on the last coordinate of $\widetilde{N} = N \oplus \Z$, and $\psi$ is the induced action on $N \times \{ 1 \}$.
%
%
%A representation $\widetilde{\psi}: W \to \GL(\widetilde{N})$ is called an \emph{affine representation} if 
%%$\widetilde{\psi}$ preserves 
%projection onto the last coordinate $\pr: \widetilde{N} = N \oplus \Z \to \Z$ is $W$-invariant, where $W$ acts trivially on the codomain $\Z$ 
%%i.e., $\widetilde{\psi}(g \cdot u) = \widetilde{\psi}(u)$ 
%(see \cite[Section~2.3]{StapledonEquivariantCommutativeTriangulations} for more details). Equivalently, if we identify $N$ with $N \times \{ 1 \} \subset \widetilde{N}$, then we may consider the induced representation $\psi: W \to \Aff(N)$ of $W$ acting on $N$ via affine transformations.  
%%For example, if $N = \{ 0 \}$, then $[\widetilde{\psi}] = 1 \in R(W)$. 
For example, 
%Observe that 
a (linear) representation $\psi: W \to \GL(N)$ is naturally an affine representation. 
% naturally induces an affine representation $\widetilde{\psi}: W \to \GL(\widetilde{N})$ by $W$ acting trivially on the last coordinate of $\widetilde{N} = N \oplus \Z$. 
Below, we will also consider the case when $N$ is empty, $\widetilde{N} = \{ 0 \}$, and  $\widetilde{\psi}: W \to \GL(\widetilde{N})$. 
% In this case, we consider the representation $\widetilde{\psi}: W \to \GL(\widetilde{N})$ to be an affine representation. 

%is the trivial representation. 

%Fix an affine representation $\widetilde{\psi}: W \to \GL(\widetilde{N})$. 
Let $P \subset N_\R$ be a full-dimensional $W$-invariant lattice polytope. %Assume that 
%$P \times \{ 1\}$ is $W$-invariant. 
%Equivalently, assume that  
%$P$ is $W$-invariant under the corresponding action via $\psi : W \to \Aff(N)$.
Recall that given a polyhedron $Q$ in $\widetilde{N}_\R$,   $C(Q)$ denotes the cone spanned by $Q$. 
 Then we have an induced action of $W$ on the semigroup  $C(P \times \{ 1\}) \cap \widetilde{N}$. 
	For any $m \in \Z_{\ge 0}$, let 
$L(P,\psi;m)$ in  $R(W)$  be the class of the permutation representation of $W$ acting on 
$\{ u \in C(P \times \{ 1\}) \cap \widetilde{N} : \pr(u) = m \}$. 
	The corresponding \emph{equivariant Ehrhart series} in $R(W)[[t]]$   is %defined by 
$$\Ehr(P,\psi;t) := \sum_{m \ge 0} L(P,\psi;m) t^m.$$
With the notation of \eqref{eq:detI-rhot}, %Section~\ref{ss:representationtheory}, 
the corresponding \emph{equivariant $h^*$-series} in $R(W)[[t]]$ is %defined by 
\begin{equation}\label{eq:equivhstar}
	h^*(P,\psi;t) := \Ehr(P,\psi;t)\det(I - \widetilde{\psi} t).
\end{equation}
%For example, if $W$ acts trivially, then $\Ehr(P,\psi;t)$ and $h^*(P,\psi;t)$ are the usual Ehrhart series and $h^*$-polynomial respectively associated to the lattice polytope $P$. 

For $w \in W$, 
$\ev_w(L(P,\psi;m))$ is the number of lattice points in the $m$th dilate of the  rational polytope $P^w$, and
%by applying $\ev_w$ coefficientwise, we may extend $\ev_w$ to a map $\ev_w: R(W)[[t]] \to \C[[t]]$. 
%Then 
$\ev_w(\Ehr(P,\psi;t)) \in \C[[t]]$ is the usual Ehrhart series of 
%the rational polytope 
$P^w$. In particular, $\ev_w(\Ehr(P,\psi;t))$ is a rational function in $t$, i.e., an element of $\C(t) \cap \C[[t]]$.
When $w = \id$, $\ev_w(h^*(P,\psi;t))$ is the usual $h^*$-polynomial $h^*(P;t) \in \Z[t]$; a polynomial of degree at most $\dim P$ with nonnegative coefficients \cite{StanleyDecompositions}.  

 Given an element $a(t) \in \C(t)$, we may consider $a(t^{-1}) \in \C(t)$. 
Let $\Int(C)$ denote the relative interior of a cone $C$. 	For any $m \in \Z_{> 0}$, let 
$L^\circ(P,\psi;m)$ in  $R(W)$  be the class of the permutation representation of $W$ acting on 
$\{ u \in \Int(C(P \times \{ 1\})) \cap \widetilde{N} : \pr(u) = m \}$, and let $\Ehr^\circ(P,\psi;t) := \sum_{m \ge 1} L^\circ(P,\psi;m) t^m \in R(W)[[t]].$
Then equivariant Ehrhart reciprocity \cite{StapledonEquivariant}*{Corollary~6.6} states that both $\ev_w(\Ehr^\circ(P,\psi;t))$ and $t^{\dim P + 1}\ev_w(h^*(P,\psi;t^{-1}))$ are well-defined elements of $\C[[t]] \cap \C(t)$, and we have the following equality in $\C[[t]] \cap \C(t)$, 
\begin{equation}\label{eq:reciprocity}
	\ev_w(\Ehr^\circ(P,\psi;t)) = \frac{ t^{\dim P + 1}\ev_w(h^*(P,\psi;t^{-1}))}{\det(I - \widetilde{\psi}(w)t)}.
\end{equation}
%where both sides are interpreted as rational functions in $\C(t)$ and $\ev_w(t^{\dim P + 1}h^*(P,\psi;t^{-1}))$ is a rational function obtained by setting $t \mapsto t^{-1}$ 
Moreover, if $h^*(P,\psi;t) \in R(W)[t]$  then its degree equals the degree of $h^*(P;t)$. 
% (which is at most $\dim P$). 
%is at most $\dim P + 1$ (in fact, its degree is the degree of the usual $h^*$-polynomial). 

%Recall that for any  $z \in \face(P)$, $W_z$ denotes the stabilizer of $z$, $F_z$ denotes the corresponding face of $P$, and $V_z$ denotes the linear span of $F_z \times \{ 1 \}$ in $\widetilde{N}_\R$. For example, if $z = \hat{1}_{\face(P)}$, then $W_z = W$, $F_z = P$ and $V_z = \widetilde{N}_\R$. Recall that for any $z \le z' \in \face(P)$, 
% $\psi$ induces a representation $\psi_{z,z'}: W_{z,z'}  \to \GL(V_{z'}/V_z)$, where $W_{z,z'} = W_{z} \cap W_{z'}$.
% We let $\psi_z = \psi_{\hat{0}_{\face(P)},z}$ for any $z \in \face(P)$. 
%%For any $z \in \face(P)$, 
%%recall that $\psi$ induces a representation $\psi_{z,\hat{1}_{\face(P)}}: W_{z,\hat{1}_P} = W_z \to \GL(V_{\hat{1}_{\face(P)}}/V_z)$, where $W_{z,z'} = W_{z} \cap W_{z'}$.
%The corresponding \emph{equivariant local $h^*$-series} in $R(W)[[t]]$ is %defined by 
%\[
%\ell^*(P,\psi;t) = \sum_{z \in \face(P)} h^*(F_z,\psi_{z};t)
%\]

Let $\cS$ be a $W$-invariant lattice polyhedral subdivision of $P$, and consider the corresponding face poset $\face(\cS)$.
For example, if $\cS$ is the trivial subdivision of $P$, then $\face(\cS)$ is the face lattice $\face(P)$ of $P$. 
Recall from Example~\ref{ex:introequivariantpolytope} that there is a natural choice of equivariant $\face(\cS)$-kernel. We recall the details.
For any $z \in \face(\cS)$, let $F_z$ denote the corresponding face in $\cS$, and let $C_z = C(F_z \times \{1\})$ with linear span $V_z$ in $\widetilde{N}_\R$.  
%Let $\widetilde{N}_z$ be the intersection of $\widetilde{N}$ with the linear span of  $C(F_z \times \{1\})$ in $\widetilde{N}_\R$. 
%For example, if $z = \hat{0}_{\face(\cS)}$, then $F_z$ is the empty face and  $C(F_z \times \{1\}) = \widetilde{N}_z$ is the origin in $\widetilde{N}_\R$. 
Recall that the fan $\Sigma_{\cS} =  \{ C_z : z \in \face(\cS) \}$ has support $C(P \times \{1 \})$. 
Recall that $W_z$ denotes the stabilizer of $z$, and $W_{z,z'} = W_z \cap W_{z'}$. 
For any $z \le z' \in \face(\cS)$, 
$\widetilde{\psi}$ induces a representation $\widetilde{\psi}_{z,z'}: W_{z,z'}  \to \GL(V_{z'}/V_z)$.
%, where $W_{z,z'} = W_{z} \cap W_{z'}$.
Consider $\face(\cS)$ with the $W$-invariant natural weak rank function, i.e., 
$\rho_{\face(\cS)}(z,z') = 
\dim (V_{z'}/V_z)$ for $z \le z' \in \face(\cS)$. 
By Example~\ref{ex:fanequivv2}, there is a natural choice of equivariant $\face(\cS)$-kernel $\kappa_{\face(\cS)}$. Explicitly, 
for $z \le z' \in \face(\cS)$, 
\[
\kappa_{\face(\cS)}(z,z') = 
\det(tI - \widetilde{\psi}_{z,z'}) \in R(W_{z,z'})[t].
\]
For $z \in \face(\cS)$, let $\widetilde{\psi}_z = \widetilde{\psi}_{\hat{0}_{\face(\cS)},z} : W_z \to \GL(V_z)$, with corresponding representation $\psi_z: W_z \to \Aff(N_z)$, where $N_z$ is the intersection of $N$ with the affine span of $F_z$ in $N_\R$.  
We have
%an induced affine representation 
%$\widetilde{\psi}_z : W_z \to \GL(V_z)$, and 
corresponding invariants $\Ehr(F_z,\psi_z;t)$ and $h^*(F_z,\psi_z;t)$ in $R(W_z)[[t]]$. 
For example, if $z = \hat{0}_{\face(\cS)}$, then $F_z$ is the empty face, and 
$\Ehr(F_z,\psi_z;t) = h^*(F_z,\psi_z;t) = 1$.

%By Example~\ref{ex:convexsupportequiv} (see \eqref{eq:rhoGammanotw} and \eqref{eq:kappagammaequiv}),  $r_{\face(\cS)} = \rho_{\face(\cS)}$ is a $W$-invariant rank function, and $\kappa_{\face(\cS)} \in \II^W(\face(\cS))$ is an  equivariant $W$-kernel. 
%%We may consider a $W$-invariant rank function  $r_{\face(\cS)} = \rho_{\face(\cS)}$ given by \eqref{eq:rhoGammanotw}, and equivariant $W$-kernel given by \eqref{eq:kappagammaequiv}.  
%%That is, 

For example, consider the trivial subdivision of $P$ with corresponding 
equivariant $W$-kernel $\kappa_{\face(P)}$ and left equivariant Kazhdan-Lusztig-Stanley function $g_{\face(P)}$. The \emph{equivariant local $h^*$-series} in $R(W)[[t]]$ is
\begin{equation}\label{eq:equivlocalhstar}
	\ell^*(P,\psi;t) = \sum_{z \in \face(P)} \frac{|W_z|}{|W|} \Ind_{W_z}^W \left( h^*(F_z,\psi_z;t) g^{-1}_{\face(P)}(z,\hat{1}_{\face(P)})\right).
\end{equation}
As in \eqref{eq:equivariantmultiplication}, the fractions appearing in the right hand side are for convenience, and can be removed by allowing the sum to vary over a choice of $W$-orbits of $\face(P)$.   The definition of 
$\ell^*(P,\psi;t)$ agrees with \cite{StapledonCalabi12}*{Definition~3.3}. 
When $w = \id$, $\ev_w(\ell^*(P,\psi;t))$ is the usual local $h^*$-polynomial $\ell^*(P;t) \in \Z[t]$; a polynomial first introduced by Stanley in \cite[Example~7.13]{Stanley92} with nonnegative coefficients and $\ell^*(P;t) = t^{\dim P + 1}  \ell^*(P;t^{-1})$.  

\begin{example}\label{ex:simplexequiv}
	Suppose that $P$ is a simplex with vertices $u_0,\ldots,u_{\dim P} \in N$. Then the coefficient of $t^m$ in $h^*(P,\psi;t)$ equals the permutation representation of $W$ acting on $\{ u \in \widetilde{N} :  u = \sum_{i = 0}^{\dim P} \lambda_i (u_i,1), 0 \le \lambda_i < 1, \sum_i \lambda_i = m \}$ 	\cite{StapledonEquivariant}*{Proposition~6.1}. 
	The coefficient of $t^m$ in $\ell^*(P,\psi;t)$ equals the permutation representation of $W$ acting on $\{ u \in \widetilde{N} :  u = \sum_{i = 0}^{\dim P} \lambda_i (u_i,1), 0 < \lambda_i < 1, \sum_i \lambda_i = m \}$ \cite{StapledonCalabi12}*{Example~3.5}. In particular, $h^*(P,\psi;t)$ and $\ell^*(P,\psi;t)$ are polynomials in this case.
\end{example}

\begin{definition}
The action of $W$ on a $W$-invariant lattice polyhedral subdivision $\cS$ is \emph{polynomial} if $h^*(F_z,\psi_z;t)$ in $R(W_z)[t]$ for all $z \in \face(\cS)$. 
\end{definition}

For example, if $\cS$ is a triangulation, i.e., all faces of $\cS$ are simplices, then the action of $W$ on $\cS$ is polynomial by Example~\ref{ex:simplexequiv}.
 %\cite{StapledonEquivariant}*{Proposition~6.1}. 
 If $W$ acts trivially, then 
the action of $W$  on $\cS$ is polynomial. 

Assume that the action of $W$ on $\cS$ is polynomial. 
Then we have well-defined elements $h_{\face(\cS)}^*,\ell_{\face(\cS)}^* \in \II^W(\face(\cS))$ such that for  $z \le z' \in \face(\cS)$, 
\[
h_{\face(\cS)}^*(z,z') = \begin{cases}
	h^*(F_{z'},\psi_{z'};t) &\textrm{ if } z = \hat{0}_\Gamma, \\
	0 &\textrm{ otherwise,} 
\end{cases}
\]
\[
\ell_{\face(\cS)}^*(z,z') = \begin{cases}
	\ell^*(F_{z'},\psi_{z'};t) &\textrm{ if } z = \hat{0}_\Gamma, \\
	0 &\textrm{ otherwise.} 
\end{cases}
\]
By definition, $h_{\face(\cS)}^* = \ell_{\face(\cS)}^* \cdot g_{\face(\cS)}^{-1}$.  In this case, equivariant Ehrhart reciprocity is equivalent to the statement that $\ell^*_{\face(\cS)}$ is symmetric in the sense that 
$(\ell^*_{\face(\cS)})^{\rev} = \ell_{\face(\cS)}^*$ \cite{StapledonCalabi12}*{Lemma~3.6}.

We want to understand the behavior of these invariants under refinement.
Suppose that $\cS'$ is a  $W$-invariant lattice polyhedral subdivision of $P$ that refines $\cS$.
%Suppose that 
%$\cS'$ is a  $W$-invariant lattice polyhedral subdivision of $P$, and 
%$\cS'$ refines $\cS$ in the sense that for any $z' \in \face(\cS')$, $F_{z'}$ is contained in a face of $\cS$. 
As in Example~\ref{ex:propermapequiv},  the identity map on $\widetilde{N}_\R$ induces a proper surjective map of fans from 
$\Sigma_{\cS'}$
 %=  \{ C(F_z \times \{1\}) : z \in \face(\cS') \}$ 
 to 
$\Sigma_{\cS}$ 
%=  \{ C(F_z \times \{1\}) : z \in \face(\cS) \}$  
 with corresponding strong formal subdivision $\sigma: X \to Y$, with $X = \face(\cS')$ and $Y = \face(\cS)$.  Here $F_{\sigma(z')}$ is the smallest face of $\cS$ containing $F_{z'}$ for all $z' \in \face(\cS')$. % such that $F_z \subset F_{\sigma(z)} \in \cS'$. 
Let $\Gamma = \Cyl(\sigma)$ be the corresponding non-Hausdorff mapping cylinder with its induced action of $W$, 
its $W$-invariant natural weak rank function  $\rho_\Gamma$ given by \eqref{eq:rhoGammanotw}, and its equivariant $W$-kernel $\kappa_\Gamma$ given by \eqref{eq:kappagammaequiv}. 
That is, for $z \le z' \in \Gamma$,
\[
\rho_\Gamma(z,z') = \begin{cases}
	\dim (V_{z'}/V_z)  + 1 &\textrm{ if } z \in X, z' \in Y, \\
	\dim (V_{z'}/V_z)  &\textrm{ otherwise. }
\end{cases}
\]
Recall that $\widetilde{\psi}$ induces a representation $\widetilde{\psi}_{z,z'}: W_{z,z'} \to \GL(V_{z'}/V_z)$, where $W_{z,z'} = W_{z} \cap W_{z'}$. Then 
\[
\kappa_\Gamma(z,z') = \begin{cases}
	(t - 1)\det(tI - \widetilde{\psi}_{z,z'}) &\textrm{ if } z \in X, z' \in  Y, \\
\det(tI - \widetilde{\psi}_{z,z'}) &\textrm{ otherwise. }
\end{cases}
\]

Assume further that the action of $W$ on both $\cS'$ and $\cS$ is polynomial.
We have well-defined elements  $h_\Gamma^*, \ell_\Gamma^*  \in \II^W(\Gamma)$ such that  for
 %Define an element $h_\Gamma^* \in \II^W(\Gamma)$ as follows. For 
 $z \le z' \in \Gamma$, 
\[
h^*_\Gamma(z,z') = \begin{cases}
	h^*(F_{z'},\psi_{z'};t) &\textrm{ if } z = \hat{0}_\Gamma, \\
	0 &\textrm{ otherwise.} 
\end{cases}
\]
\[
\ell^*_\Gamma(z,z') = \begin{cases}
	\ell^*(F_{z'},\psi_{z'};t) &\textrm{ if } z = \hat{0}_\Gamma, \\
	0 &\textrm{ otherwise.} 
\end{cases}
\]
Then $h^*_\Gamma$ and $\ell^*_\Gamma$ restrict to $h^*_X$ and $\ell^*_X$ respectively. Also, $\ell_\Gamma^* = h_\Gamma^* \cdot ( g_\Gamma^{-1}|_X + g_\Gamma^{-1}|_Y)$. 
%Also, let $\ell^* = h^* \cdot g_\Gamma^{-1} \in \II^W(\Gamma)$. For any $z \in \Gamma$, $\ell^*(\hat{0}_\Gamma,z) \in R(W_z)[t]$ is the \emph{equivariant $h^*$-polynomial} $\ell^*(F_z,\psi_v;t)$ of $F_z$. This definition agrees with \cite{StapledonCalabi12}*{Definition~3.3} and restricts to the usual local $h^*$-polynomial when $W$ acts trivially. \alan{ref}

When $W$ acts trivially,
 %the assumption that the action of $W$ on $\cS$ is polynomial holds, and 
 the proposition below is equivalent to $(2)$ in \cite{KatzStapledon16}*{Lemma~7.12}. 

\begin{proposition}\label{prop:equivariantEhrhartapp}
Let $N$ be a lattice and let $P \subset N_\R$ be a full-dimensional lattice polytope. % and let $W$ be a finite group. 
Consider an affine representation $\psi: W \to \Aff(N)$ of a finite group $W$ such that $P$ is $W$-invariant. Let $\cS'$ and $\cS$ be $W$-invariant lattice polyhedral subdivisions of $P$ such that $\cS'$ refines $\cS$. Suppose that the action of $W$ on $\cS'$ is polynomial. Then the action of $W$ on $\cS$ is polynomial and we have the following equality in $\II^W(\Gamma)$,
\begin{equation}\label{eq:hstarresult}
	(t - 1) \cdot h^*_\Gamma|_{X/Y} = h^*_\Gamma \cdot \kappa_\Gamma|_{(X/Y)^\circ} =  h^*_\Gamma|_X \cdot \kappa_\Gamma|_{(X/Y)^\circ}.
\end{equation} 
In particular, if there exists a $W$-invariant lattice polyhedral subdivision $\cS$ of $P$ such that the action of $W$ on $\cS$ is polynomial, then the action of $W$ on the trivial subdivision of $P$ is polynomial, and 
% $h^*(P,\psi;t) \in R(W)[t]$ is a polynomial, and 
\begin{equation}\label{eq:hstarPformula}
h^*(P,\psi;t)  = %\sum_{\substack{z \in \face(\cS) \\ \sigma(z) = \hat{1}_{\face(P)}
		\sum_{\substack{z \in \face(\cS) \\ \sigma(z) = \hat{1}_{\face(P)} }} \frac{|W_z|}{|W|} \Ind^{W}_{W_{z}} \left( h^*(F_{z},\psi_{z};t) \det(tI - \widetilde{\psi}_{z,\hat{1}_\Gamma}) \right) \in R(W)[t].
\end{equation}
\end{proposition}
\begin{proof}
	The second statement follows from the first by setting $\cS'$ to be $\cS$, and setting $\cS$ to be the trivial subdivision of $P$, and evaluating both sides of \eqref{eq:hstarresult} at the interval 
	$\Gamma = [ \hat{0}_\Gamma, \hat{1}_\Gamma ]$, using the definition of multiplication in $\II^W(\Gamma)$ (see \eqref{eq:equivariantmultiplication}). 
	
	Fix  $w \in W$ and $y \in Y^w$. Since every $w$-fixed point in $\Int(C_y)$  lies in $\Int(C_x)$ for a unique $x \in X^w$ such that $\sigma(x) = y$, we have the following equality in $\C(t) \cap \C[[t]]$,
		\[
	\ev_w(\Ehr^\circ(F_y, \psi_y;t)) = \sum_{\substack{x \in X^w \\ \sigma(x) = y}}  \ev_w(\Ehr^\circ(F_x, \psi_x;t)).
	\]
    Applying \eqref{eq:reciprocity} to both sides gives 
	\[
	\frac{ t^{\dim V_y}\ev_w(h^*(F_y, \psi_y;t^{-1}))}{\det(I - \widetilde{\psi_y}(w)t)} = \sum_{\substack{x \in X^w \\ \sigma(x) = y}} 	\frac{ t^{\dim V_x}\ev_w(h^*(F_x, \psi_x;t^{-1}))}{\det(I - \widetilde{\psi_x}(w)t)}. 
	\]
	Replacing $t$ with $t^{-1}$ gives the following equality in $\C(t)$,
	\[
	\frac{ \ev_w(h^*(F_y, \psi_y;t))}{\ev_w(\det(tI - \widetilde{\psi_y}))} = \sum_{\substack{x \in X^w \\ \sigma(x) = y}} 	\frac{ \ev_w(h^*(F_x, \psi_x;t))}{\ev_w(\det(tI - \widetilde{\psi_x}))}. 
	\]
	Equivalently, using the multiplicative property of $\kappa_\Gamma$ and \eqref{eq:kappagammaequiv}, 
	\[
\ev_w(h^*(F_y, \psi_y;t)) =  \frac{1}{t - 1} \sum_{\substack{x \in X^w \\ \sigma(x) = y}} 	 \ev_w(h^*(F_x, \psi_x;t)) \ev_w(\kappa_\Gamma(x,y)). 
	\]
	Since the action of $W$ on $\cS'$ is polynomial, and $\ev_w(\kappa_\Gamma(x,y))$ is divisible by $t - 1$ by Lemma~\ref{lem:poleatone} and Lemma~\ref{lem:reducenonequivariant}, 
	the right hand side of the above equation is a polynomial in $t$. 
	We deduce that $\ev_w(h^*(F_y, \psi_y;t))$ is a polynomial in $t$. Since this holds for all choices of $w$ and $y$,  we conclude that  the action of $W$ on $\cS$ is polynomial. Moreover, the above equation can be rewritten as
	\[
	(t - 1)\ev_w(h^*_\Gamma)(\hat{0}_\Gamma,y) = (\ev_w(h^*_\Gamma) \cdot \ev_w(\kappa_\Gamma|_{(X/Y)^\circ}))(\hat{0}_\Gamma,y). 
	\] 
	We deduce that the following equality  in $I(\Gamma^W)$
	\[
	(t - 1) \cdot \ev_w(h^*_\Gamma|_{X/Y}) = \ev_w(h^*_\Gamma) \cdot \ev_w(\kappa_\Gamma|_{(X/Y)^\circ}). 
	\]
	In fact, since the right hand side lies in $\II(\Gamma^w)$,  the above inequality holds in $\II(\Gamma^w)$. Using Lemma~\ref{lem:Ztalgebrahom}, we deduce that 
	\[
		(t - 1) \cdot h^*_\Gamma|_{X/Y} = h^*_\Gamma \cdot \kappa_\Gamma|_{(X/Y)^\circ}. 
	\]
	This establishes \eqref{eq:hstarresult}, using the fact that 
	$h^*_\Gamma \cdot \kappa_\Gamma|_{(X/Y)^\circ} = h^*_\Gamma|_X \cdot \kappa_\Gamma|_{(X/Y)^\circ}$.
\end{proof}

\begin{remark}
By comparison with Definition~\ref{def:equivarianthellpolynomial},  \eqref{eq:hstarresult} in Proposition~\ref{prop:equivariantEhrhartapp} is equivalent to 
\begin{equation*}%
	h^*_\Gamma|_{X/Y} = h^*_\Gamma|_X \cdot g_\Gamma^{-1}|_X \cdot h_\sigma = \ell^*_\Gamma|_X \cdot h_\sigma, 
\end{equation*}
and, after multiplying both sides on the right by $g_\Gamma^{-1}|_Y$, is also equivalent to 
\begin{equation*}%
	\ell^*_\Gamma|_{X/Y} =  \ell^*_\Gamma|_X \cdot \ell_\sigma. 
\end{equation*}
In particular, suppose there exists a $W$-invariant lattice polyhedral subdivision $\cS$ of $P$ such that the action of $W$ on $\cS$ is polynomial. 
 Setting $\cS'$ to be $\cS$, and setting $\cS$ to be the trivial subdivision of $P$ above,
 then evaluating %\eqref{eq:hstarintermsofh} and \eqref{eq:localhstarintermsofh} 
the above equations
at the interval 
$\Gamma = [ \hat{0}_\Gamma, \hat{1}_\Gamma ]$ gives 
\begin{equation}\label{eq:hstarintermsofh}
	h^*(P,\psi;t)  = %\sum_{\substack{z \in \face(\cS) \\ \sigma(z) = \hat{1}_{\face(P)}
			\sum_{z \in \face(\cS)} \frac{|W_z|}{|W|} \Ind^{W}_{W_{z}} \left( \ell^*(F_{z},\psi_{z};t) h_\sigma(z,\hat{1}_\Gamma) \right) \in R(W)[t],
\end{equation}
\begin{equation}\label{eq:localhstarintermsofh}
	\ell^*(P,\psi;t)  = %\sum_{\substack{z \in \face(\cS) \\ \sigma(z) = \hat{1}_{\face(P)}
			\sum_{z \in \face(\cS)} \frac{|W_z|}{|W|} \Ind^{W}_{W_{z}} \left( \ell^*(F_{z},\psi_{z};t) \ell_\sigma(z,\hat{1}_\Gamma) \right) \in R(W)[t].
\end{equation}
Observe that when $z = \hat{0}_{\face(\cS)}$ in the sums above, the corresponding contributions are $h_\sigma(\Gamma)$ and $\ell_\sigma(\Gamma)$ respectively. 
When $W$ acts trivially, \eqref{eq:hstarintermsofh} and \eqref{eq:localhstarintermsofh} are equivalent to $(3)$ and $(4)$ in  \cite{KatzStapledon16}*{Lemma~7.12} respectively.  
\end{remark}

\begin{example}\label{ex:triangulations}
	Suppose there exists  a $W$-invariant lattice triangulation $\cS$ of $P$. Recall that the action of $W$ on $\cS$ is polynomial
	by Example~\ref{ex:simplexequiv}. %\cite{StapledonEquivariant}*{Proposition~6.1}. 
	In this case, Example~\ref{ex:simplexequiv} gives explicit combinatorial formulas for the coefficients of the polynomials $h^*(F_{z},\psi_{z};t)$ and $\ell^*(F_{z},\psi_{z};t)$ as permutation representations for all $z \in X$. 
	In particular,  \eqref{eq:hstarPformula} agrees with the formula \cite{StapledonEquivariantCommutativeTriangulations}*{Proposition~4.40}.
	 If we further assume that $\cS$ is a unimodular triangulation, i.e., the fan $\Sigma_{\cS}$ is smooth, then it follows from Example~\ref{ex:simplexequiv} that $\ell^*_\sigma(\hat{0}_\Gamma,x) =  \delta_\Gamma(\hat{0}_\Gamma,x)$ for all $x \in X$. 
%	 
%	  \eqref{ex:fanequiv} and \cite{StapledonEquivariant}*{Proposition~6.1} that $h^*_\Gamma(\hat{0}_\Gamma,x) = g_\Gamma(\hat{0}_\Gamma,x) = 1$ for all $x \in X$. 
%	  %is $1$ on all intervals in $X$ and $0$ otherwise. 
%	  In particular, $\ell^*_\Gamma(\hat{0}_\Gamma,x) = (h^*_\Gamma \cdot g_\Gamma^{-1})(\hat{0}_\Gamma,x) = \delta_\Gamma(\hat{0}_\Gamma,x)$. 
	   By \eqref{eq:hstarintermsofh} and \eqref{eq:localhstarintermsofh}, we deduce that $h^*(P,\psi;t) = h_\sigma(\Gamma)$ is the equivariant $h$-polynomial of $\sigma$, and 
	   $\ell^*(P,\psi;t) = \ell_\sigma(\Gamma)$ is the equivariant local $h$-polynomial of $\sigma$ (c.f. \cite{StapledonEquivariantCommutativeTriangulations}*{Remark~4.41} and \cite{DDEquivariantEhrhart}*{Theorem~5.2}).

%	  In particular, $\ell^*_\Gamma|_X = \delta_\Gamma|_X$, and we deduce that 
%	$h^*_\Gamma|_{X/Y} = h_\sigma$ and $\ell^*_\Gamma|_{X/Y} = \ell_\sigma$ in this case (c.f. \cite{StapledonEquivariantCommutativeTriangulations}*{Remark~4.41} and \cite{DDEquivariantEhrhart}*{Theorem~5.2}). 
%	% Then \eqref{eq:hstarintermsofh} implies that 
%	% comparison with Definition~\ref{def:equivarianthellpolynomial} yeilds that 

\end{example}

\bibliography{KLStheory}

\begin{bibdiv}
\begin{biblist}

\bib{ASV21}{article}{
      author={Ardila, Federico},
      author={Schindler, Anna},
      author={Vindas-Mel\'{e}ndez, Andr\'{e}s~R.},
       title={The equivariant volumes of the permutahedron},
        date={2021},
        ISSN={0179-5376},
     journal={Discrete Comput. Geom.},
      volume={65},
      number={3},
       pages={618\ndash 635},
}

\bib{ASV20}{article}{
      author={Ardila, Federico},
      author={Supina, Mariel},
      author={Vindas-Mel\'{e}ndez, Andr\'{e}s~R.},
       title={The equivariant {E}hrhart theory of the permutahedron},
        date={2020},
        ISSN={0002-9939},
     journal={Proc. Amer. Math. Soc.},
      volume={148},
      number={12},
       pages={5091\ndash 5107},
}

\bib{Athanasiadis12b}{article}{
      author={Athanasiadis, C.},
       title={Cubical subdivisions and local {$h$}-vectors},
        date={2012},
        ISSN={0218-0006},
     journal={Ann. Comb.},
      volume={16},
      number={3},
       pages={421\ndash 448},
  url={https://doi-org.ezproxy.lib.utexas.edu/10.1007/s00026-012-0140-y},
}

\bib{Athanasiadis12}{article}{
      author={Athanasiadis, C.},
       title={Flag subdivisions and {$\gamma$}-vectors},
        date={2012},
        ISSN={0030-8730},
     journal={Pacific J. Math.},
      volume={259},
      number={2},
       pages={257\ndash 278},
         url={https://doi-org.ezproxy.lib.utexas.edu/10.2140/pjm.2012.259.257},
}

\bib{AthanasiadisSurveySubdivisions}{incollection}{
      author={Athanasiadis, Christos~A.},
       title={A survey of subdivisions and local {$h$}-vectors},
        date={2016},
   booktitle={The mathematical legacy of {R}ichard {P}. {S}tanley},
   publisher={Amer. Math. Soc., Providence, RI},
       pages={39\ndash 51},
}

\bib{BMSimpleHomotopy}{article}{
      author={Barmak, Jonathan~Ariel},
      author={Minian, Elias~Gabriel},
       title={Simple homotopy types and finite spaces},
        date={2008},
        ISSN={0001-8708,1090-2082},
     journal={Adv. Math.},
      volume={218},
      number={1},
       pages={87\ndash 104},
}

\bib{BBFKCombinatorialIntersectionCohomology}{article}{
      author={Barthel, Gottfried},
      author={Brasselet, Jean-Paul},
      author={Fieseler, Karl-Heinz},
      author={Kaup, Ludger},
       title={Combinatorial intersection cohomology for fans},
        date={2002},
        ISSN={0040-8735,2186-585X},
     journal={Tohoku Math. J. (2)},
      volume={54},
      number={1},
       pages={1\ndash 41},
}

\bib{BMIntersectionHomologyKalai}{article}{
      author={Braden, Tom},
      author={MacPherson, Robert},
       title={Intersection homology of toric varieties and a conjecture of
  {K}alai},
        date={1999},
        ISSN={0010-2571,1420-8946},
     journal={Comment. Math. Helv.},
      volume={74},
      number={3},
       pages={442\ndash 455},
}

\bib{BPIntersectionCohomology}{misc}{
      author={Braden, Tom},
      author={Proudfoot, Nicholas},
       title={Intersection cohomology without spaces},
         how={preprint arXiv:2510.09488},
        date={arXiv:2510.09488, 2025},
}

\bib{BrentiTwistedIncidence}{article}{
      author={Brenti, Francesco},
       title={Twisted incidence algebras and {K}azhdan-{L}usztig-{S}tanley
  functions},
        date={1999},
        ISSN={0001-8708,1090-2082},
     journal={Adv. Math.},
      volume={148},
      number={1},
       pages={44\ndash 74},
}

\bib{BrentiPKernels}{article}{
      author={Brenti, Francesco},
       title={{$P$}-kernels, {IC} bases and {K}azhdan-{L}usztig polynomials},
        date={2003},
        ISSN={0021-8693,1090-266X},
     journal={J. Algebra},
      volume={259},
      number={2},
       pages={613\ndash 627},
}

\bib{BLIntersectionCohomologyNonrational}{article}{
      author={Bressler, Paul},
      author={Lunts, Valery~A.},
       title={Intersection cohomology on nonrational polytopes},
        date={2003},
        ISSN={0010-437X,1570-5846},
     journal={Compositio Math.},
      volume={135},
      number={3},
       pages={245\ndash 278},
}

\bib{ChanLocal}{article}{
      author={Chan, Clara},
       title={On subdivisions of simplicial complexes: characterizing local
  {$h$}-vectors},
        date={1994},
        ISSN={0179-5376,1432-0444},
     journal={Discrete Comput. Geom.},
      volume={11},
      number={4},
       pages={465\ndash 476},
}

\bib{ChanSurvey}{incollection}{
      author={Chan, Clara},
       title={A survey of {$h$}-vectors and local {$h$}-vectors},
        date={1994},
       pages={81\ndash 98},
        note={Modern aspects of combinatorial structure on convex polytopes
  (Japanese) (Kyoto, 1993)},
}

\bib{CHK23}{article}{
      author={Clarke, Oliver},
      author={Higashitani, Akihiro},
      author={K\"{o}lbl, Max},
       title={The equivariant {E}hrhart theory of polytopes with order-two
  symmetries},
        date={2023},
        ISSN={0002-9939,1088-6826},
     journal={Proc. Amer. Math. Soc.},
      volume={151},
      number={9},
       pages={4027\ndash 4041},
}

\bib{DDEquivariantHilbertTranslative}{misc}{
      author={D'Al\`{i}, Alessio},
      author={Delucchi, Emanuele},
       title={Equivariant {H}ilbert and {E}hrhart series under translative
  group actions},
         how={preprint arXiv:2312.14088},
        date={2024},
}

\bib{DDEquivariantEhrhart}{misc}{
      author={D'al\`{i}, Alessio},
      author={Delucchi, Emanuele},
       title={Equivariant {H}ilbert and {E}hrhart series under translative
  group actions},
         how={preprint arXiv:2312.14088v2},
        date={arXiv:2312.14088v2, 2023},
}

\bib{deCataldoMiglioriniMustata18}{article}{
      author={de~Cataldo, M.},
      author={Migliorini, L.},
      author={Musta\c{t}\u{a}, M.},
       title={Combinatorics and topology of proper toric maps},
        date={2018},
        ISSN={0075-4102},
     journal={J. Reine Angew. Math.},
      volume={744},
       pages={133\ndash 163},
         url={https://doi-org.ezproxy.lib.utexas.edu/10.1515/crelle-2015-0104},
}

\bib{dMGPSS20}{article}{
      author={de~Moura, A.},
      author={Gunther, E.},
      author={Payne, S.},
      author={Schuchardt, J.},
      author={Stapledon, A.},
       title={Triangulations of simplices with vanishing local
  {$h$}-polynomial},
        date={2020},
     journal={Algebr. Comb.},
      volume={3},
      number={6},
       pages={1417\ndash 1430},
         url={https://doi-org.ezproxy.lib.utexas.edu/10.5802/alco},
}

\bib{DKTPosetSubdivisions}{misc}{
      author={Dornian, Patrick},
      author={Katz, Eric},
      author={Tsang, Ling~Hei},
       title={Poset subdivisions and the mixed cd-index},
         how={preprint arXiv:2001.01875},
        date={arXiv:2001.01875, 2020},
}

\bib{DyerHecke}{article}{
      author={Dyer, M.~J.},
       title={Hecke algebras and shellings of {B}ruhat intervals},
        date={1993},
        ISSN={0010-437X,1570-5846},
     journal={Compositio Math.},
      volume={89},
      number={1},
       pages={91\ndash 115},
}

\bib{EKDecompositionTheoremCDIndex}{article}{
      author={Ehrenborg, Richard},
      author={Karu, Kalle},
       title={Decomposition theorem for the {${\bf cd}$}-index of {G}orenstein
  posets},
        date={2007},
        ISSN={0925-9899,1572-9192},
     journal={J. Algebraic Combin.},
      volume={26},
      number={2},
       pages={225\ndash 251},
}

\bib{EKS22}{article}{
      author={Elia, Sophia},
      author={Kim, Donghyun},
      author={Supina, Mariel},
       title={Techniques in equivariant {E}hrhart theory},
        date={2023},
        ISSN={0219-3094},
     journal={Ann. Comb.},
}

\bib{FMVChowFunctions}{misc}{
      author={Ferroni, Luis},
      author={Matherne, Jacob~P.},
      author={Vecchi, Lorenzo},
       title={Chow functions for partially ordered sets},
         how={preprint arXiv:2411.04070v2},
        date={arXiv:2411.04070v2, 2024},
}

\bib{FREulerian}{misc}{
      author={Ferroni, Luis},
      author={Riccardi, Roberto},
       title={Eulerian posets and {Z}-polynomials},
         how={preprint arXiv:2510.17679},
        date={arXiv:2510.17679, 2025},
}

\bib{GPYEquivariantKLS}{article}{
      author={Gedeon, Katie},
      author={Proudfoot, Nicholas},
      author={Young, Benjamin},
       title={The equivariant {K}azhdan-{L}usztig polynomial of a matroid},
        date={2017},
        ISSN={0097-3165,1096-0899},
     journal={J. Combin. Theory Ser. A},
      volume={150},
       pages={267\ndash 294},
}

\bib{JKSS}{article}{
      author={Juhnke-Kubitzke, Martina},
      author={Murai, Satoshi},
      author={Sieg, Richard},
       title={Local {$h$}-vectors of quasi-geometric and barycentric
  subdivisions},
        date={2019},
        ISSN={0179-5376,1432-0444},
     journal={Discrete Comput. Geom.},
      volume={61},
      number={2},
       pages={364\ndash 379},
}

\bib{KaruRelativeHardLefschetz}{article}{
      author={Karu, Kalle},
       title={Relative hard {L}efschetz theorem for fans},
        date={2019},
        ISSN={0001-8708,1090-2082},
     journal={Adv. Math.},
      volume={347},
       pages={859\ndash 903},
}

\bib{KatzStapledon16}{article}{
      author={Katz, E.},
      author={Stapledon, A.},
       title={Local {$h$}-polynomials, invariants of subdivisions, and mixed
  {E}hrhart theory},
        date={2016},
        ISSN={0001-8708},
     journal={Adv. Math.},
      volume={286},
       pages={181\ndash 239},
         url={http://dx.doi.org/10.1016/j.aim.2015.09.010},
}

\bib{LPS1}{article}{
      author={Larson, Matt},
      author={Payne, Sam},
      author={Stapledon, Alan},
       title={Resolutions of local face modules, functoriality, and vanishing
  of local {$h$}-vectors},
        date={2023},
        ISSN={2589-5486},
     journal={Algebr. Comb.},
      volume={6},
      number={4},
       pages={1057\ndash 1072},
}

\bib{LPSLocalMotivic}{misc}{
      author={Larson, Matt},
      author={Payne, Sam},
      author={Stapledon, Alan},
       title={The local motivic monodromy conjecture for simplicial
  nondegenerate singularities},
        date={arXiv:2209.03553, 2023},
}

\bib{LSLefschetz}{misc}{
      author={Larson, Matt},
      author={Stapledon, Alan},
       title={Lefschetz properties of local face modules},
         how={preprint arXiv:2504.02038},
        date={arXiv:2504.02038, 2025},
}

\bib{ProudfootAGofKLSpolynomials}{article}{
      author={Proudfoot, Nicholas},
       title={The algebraic geometry of {K}azhdan-{L}usztig-{S}tanley
  polynomials},
        date={2018},
        ISSN={2308-2151,2308-216X},
     journal={EMS Surv. Math. Sci.},
      volume={5},
      number={1-2},
       pages={99\ndash 127},
}

\bib{ProudfootEquivariantKLS}{article}{
      author={Proudfoot, Nicholas},
       title={Equivariant incidence algebras and equivariant
  {K}azhdan-{L}usztig-{S}tanley theory},
        date={2021},
        ISSN={2589-5486},
     journal={Algebr. Comb.},
      volume={4},
      number={4},
       pages={675\ndash 681},
}

\bib{SaitoMixed}{article}{
      author={Saito, Takahiro},
       title={On the mixed {H}odge structures of the intersection cohomology
  stalks of complex hypersurfaces},
        date={2020},
        ISSN={0034-5318,1663-4926},
     journal={Publ. Res. Inst. Math. Sci.},
      volume={56},
      number={1},
       pages={55\ndash 82},
}

\bib{STMonodromies}{article}{
      author={Saito, Takahiro},
      author={Takeuchi, Kiyoshi},
       title={On the monodromies and the limit mixed {H}odge structures of
  families of algebraic varieties},
        date={2023},
        ISSN={0026-2285,1945-2365},
     journal={Michigan Math. J.},
      volume={73},
      number={3},
       pages={631\ndash 668},
}

\bib{SerreLinearRepresentations}{book}{
      author={Serre, Jean-Pierre},
       title={Linear representations of finite groups},
     edition={French},
      series={Graduate Texts in Mathematics},
   publisher={Springer-Verlag, New York-Heidelberg},
        date={1977},
      volume={Vol. 42},
        ISBN={0-387-90190-6},
}

\bib{Stanley92}{article}{
      author={Stanley, R.},
       title={Subdivisions and local {$h$}-vectors},
        date={1992},
        ISSN={0894-0347},
     journal={J. Amer. Math. Soc.},
      volume={5},
      number={4},
       pages={805\ndash 851},
         url={http://dx.doi.org/10.2307/2152711},
}

\bib{StanleyDecompositions}{article}{
      author={Stanley, Richard~P.},
       title={Decompositions of rational convex polytopes},
        date={1980},
     journal={Ann. Discrete Math.},
      volume={6},
       pages={333\ndash 342},
}

\bib{StanleyGeneralized87}{incollection}{
      author={Stanley, Richard~P.},
       title={Generalized {$H$}-vectors, intersection cohomology of toric
  varieties, and related results},
        date={1987},
   booktitle={Commutative algebra and combinatorics ({K}yoto, 1985)},
      series={Adv. Stud. Pure Math.},
      volume={11},
   publisher={North-Holland, Amsterdam},
       pages={187\ndash 213},
}

\bib{StanleyFlagfVectors}{article}{
      author={Stanley, Richard~P.},
       title={Flag {$f$}-vectors and the {$cd$}-index},
        date={1994},
        ISSN={0025-5874,1432-1823},
     journal={Math. Z.},
      volume={216},
      number={3},
       pages={483\ndash 499},
}

\bib{StanleySurveyEulerian}{incollection}{
      author={Stanley, Richard~P.},
       title={A survey of {E}ulerian posets},
        date={1994},
   booktitle={Polytopes: abstract, convex and computational ({S}carborough,
  {ON}, 1993)},
      series={NATO Adv. Sci. Inst. Ser. C: Math. Phys. Sci.},
      volume={440},
   publisher={Kluwer Acad. Publ., Dordrecht},
       pages={301\ndash 333},
}

\bib{StanleyEnumerative}{book}{
      author={Stanley, Richard~P.},
       title={Enumerative combinatorics. {V}olume 1},
     edition={Second},
      series={Cambridge Studies in Advanced Mathematics},
   publisher={Cambridge University Press, Cambridge},
        date={2012},
      volume={49},
        ISBN={978-1-107-60262-5},
}

\bib{StapledonEquivariant}{article}{
      author={Stapledon, Alan},
       title={Equivariant {E}hrhart theory},
        date={2011},
        ISSN={0001-8708},
     journal={Adv. Math.},
      volume={226},
      number={4},
       pages={3622\ndash 3654},
}

\bib{StapledonRepresentations11}{article}{
      author={Stapledon, Alan},
       title={Representations on the cohomology of hypersurfaces and mirror
  symmetry},
        date={2011},
        ISSN={0001-8708},
     journal={Adv. Math.},
      volume={226},
      number={6},
       pages={5268\ndash 5297},
}

\bib{StapledonCalabi12}{article}{
      author={Stapledon, Alan},
       title={New {C}alabi-{Y}au orbifolds with mirror {H}odge diamonds},
        date={2012},
        ISSN={0001-8708},
     journal={Adv. Math.},
      volume={230},
      number={4-6},
       pages={1557\ndash 1596},
}

\bib{Stapledon17}{article}{
      author={Stapledon, Alan},
       title={Formulas for monodromy},
        date={2017},
        ISSN={2197-9847},
     journal={Res. Math. Sci.},
      volume={4},
       pages={Paper No.~8, 42~pp.},
         url={http://dx.doi.org/10.1186/s40687-017-0097-x},
}

\bib{StapledonLWPosets}{misc}{
      author={Stapledon, Alan},
       title={Subdivisions of lower {E}ulerian posets},
         how={to appear},
        date={2025},
}

\bib{StapledonEquivariantCommutativeTriangulations}{misc}{
      author={Stapledon, Alan},
       title={Equivariant {E}hrhart theory, commutative algebra and invariant
  triangulations of polytopes},
         how={preprint arXiv:2311.17273},
        date={arXiv:2311.17273, 2023},
}

\end{biblist}
\end{bibdiv}
\bibliographystyle{amsalpha}

\end{document}